\documentclass[reqno,11pt]{amsart}

\usepackage{amsmath}
\usepackage{amsthm}
\usepackage[hidelinks]{hyperref}
\usepackage{ragged2e}
\usepackage{enumitem}
\setlist[enumerate,1]{label=(\roman*),ref=(\roman*)}
\usepackage{tikz-cd}
\usepackage{mathtools}
\usepackage{mathrsfs}  
\usepackage{graphicx} 
\usetikzlibrary{arrows.meta}

\usepackage[
backend=biber,
style=alphabetic,
]{biblatex}

\DeclareFieldFormat{postnote}{#1}

\usepackage{amssymb}
\newcommand{\bb}{\mathbb}

\newcommand{\Gal}{\operatorname{Gal}}

\newcommand{\divv}{\operatorname{div}}

\newcommand{\an}{\operatorname{an}}
\newcommand{\vphi}{\varphi}
\newcommand{\Supp}{\operatorname{Supp}}

\newcommand{\Res}{\operatorname{Res}}

\newcommand{\ovl}{\overline}

\newcommand{\cal}{\mathcal}
\newcommand{\Spec}{\operatorname{Spec}}
\newcommand{\ord}{\operatorname{ord}}

\newcommand{\eps}{\varepsilon}
\newcommand{\Hom}{\operatorname{Hom}}

\newcommand{\id}{\operatorname{id}}

\newcommand{\End}{\operatorname{End}}

\newcommand{\h}{\widehat{h}}

\newcommand{\Rat}{\operatorname{Rat}}

\newcommand{\Frac}{\operatorname{Frac}}

\newcommand{\Ht}{\widehat{H}}
\newcommand{\Jac}{\operatorname{Jac}}
\newcommand{\Pic}{\operatorname{Pic}}

\newcommand{\supp}{\operatorname{supp}}
\newcommand{\reg}{\operatorname{reg}}
\newcommand{\Hilb}{\operatorname{Hilb}}

\newcommand{\Fal}{\operatorname{Fal}}
\newcommand{\val}{\operatorname{val}}
\newcommand{\gon}{\operatorname{gon}}
\newcommand{\Tr}{\operatorname{Tr}}
\newcommand{\scr}{\mathscr}
\newcommand{\C}{\mathbb{C}}

\newtheorem{theorem}{Theorem}
\numberwithin{theorem}{section}
\newtheorem{lemma}[theorem]{Lemma}
\newtheorem{proposition}[theorem]{Proposition}

\numberwithin{example}{section}

\newtheorem{definition}{Definition}
\numberwithin{definition}{section}
\newtheorem{corollary}[theorem]{Corollary}

\newtheorem{conjecture}[theorem]{Conjecture}
\newtheorem{rmk}{Remark}
\numberwithin{rmk}{section}

\makeindex
\title[Uniform boundedness over function fields]{Uniform boundedness of small points on abelian varieties over function fields}
\author{Nicole Looper and Jit Wu Yap}
\begin{document}

\begin{abstract} \vspace*{-1em}
Let $k$ be a field of characteristic $0$ and let $K = k(B)$ be the function field of a geometrically irreducible projective curve $B$ over $k$. Let $A/K$ be a $g$-dimensional abelian variety with $\Tr_{K/k}(A) = 0$. We prove that any $K$-rational torsion point $x$ of $A$ has order uniformly bounded in terms of $g$ and the gonality of $B$. We also prove a uniform lower bound on the N\'{e}ron--Tate height $\h_{A,L}(x)$ in terms of the stable Faltings height $h_{\Fal}(A)$ for any $K$-rational point $x$ whose forward orbit is Zariski dense, proving the Lang--Silverman conjecture over function fields of characteristic $0$.  
\end{abstract}

\maketitle

\tableofcontents \vspace*{-4em}

\newpage 
\section{Introduction}
In his landmark work \cite{Maz77, Maz78}, Mazur proved a uniform bound on the order of a $\bb{Q}$-rational torsion point on any elliptic curve $E/\bb{Q}$. In fact, Mazur even gave an exact list of all of the possible torsion subgroups of $E(\bb{Q})$. This uniform boundedness result was generalized to number fields $K$ by Kamienny \cite{Kam86}, Kamienny--Mazur \cite{KM95} and Merel \cite{Mer96}, who proved the existence of an upper bound on the order of a $K$-rational torsion point on any elliptic curve over $K$ depending only on $[K:\bb{Q}]$. This bound was later made effective by Parent \cite{Par99}. 
\par 
There are analogous statements over function fields. For $K = \bb{C}(B)$ the function field of a smooth projective curve $B/\bb{C}$, an upper bound on the order of a $K$-torsion point for a non-isotrivial elliptic curve $E/K$ was given in terms of the genus $g(B)$ by Levin \cite{Lev68}. Upper bounds depending only on the gonality $\gon(B)$ were obtained by Abramovich \cite{Abr96} and Nguyen--Saito \cite{NS96}. 
\par 
For elliptic curves over function fields of arbitrary characteristic, an upper bound depending only on the genus was given by Cojocaru and Hall \cite{CH05}. Poonen \cite{Poo07} subsequently proved an upper bound depending only on the gonality.
\par 
For higher dimensional abelian varieties over number fields, results over one-parameter families were obtained by Cadoret--Tamagawa \cite{CT12} and Ellenberg--Hall--Kolwaski \cite{EHK12}. Statistical results were also shown by Laga--Thorne \cite{LT25}.
\par 
Our first main result is a uniform boundedness statement for abelian varieties $A/K$, where $K = k(B)$ is a function field of a curve $B$ over a field $k$ of characteristic $0$. By a curve $B$, we mean a smooth geometrically integral projective variety of dimension $1$. We say that an abelian variety $A/K$ has no $K$-isotrivial part if its $K/k$-trace $\Tr_{K/k}(A)$ is $0$. We write $\gon(B)$ for the gonality of $B$.

\begin{theorem} \label{UniformIntroTheorem1}
Let $K = k(B)$ be the function field of a curve $B/k$, where $\operatorname{char} k=0$. There exists an integer $N = N(g,\gon(B))$, depending only on $g$ and $\gon(B)$, such that for any abelian variety $A/K$ of dimension $g$ with no $K$-isotrivial part, any torsion point $x \in A(K)$ has order at most $N$.     
\end{theorem}

Our methods also yield a uniform lower bound on the canonical height of a non-torsion $K$-rational point. A conjecture of Lang \cite[Page 92]{Lan78} predicts for any number field $K$, there is a uniform constant $c = c(K) > 0$ such that if $E/K$ is an elliptic curve and $x \in E(K)$ is non-torsion, then 
\begin{equation}\label{eqn:Lang}\h_E(x) \geq c(K) \cdot \log \left|N_{K/\bb{Q}}(\cal{D}_{E/K}) \right|\end{equation}
where $\cal{D}_{E/K}$ is the minimal discriminant of $E/K$ and $\h_E$ is the N\'{e}ron--Tate height for $E$.
\par 
Silverman obtained partial results on Lang's conjecture in \cite{Sil81} and the analogous statement for function fields $K = \bb{C}(B)$ was proven by Hindry and Silverman \cite{HS88}, where $c$ only depends on the genus $g(B)$ of the base curve. Silverman also extended the conjecture to abelian varieties in \cite{Sil84}, where we require the orbit $\bb{Z} \cdot x$ to be Zariski dense (see also \cite{Paz10}).
\par 
Our second main result is a Lang--Silverman type result for abelian varieties over $K = k(B)$, with a lower bound proportional to the stable Faltings height. Given an abelian variety $A/K$, let $K' = k'(B')$ be a finite extension over which $A/K'$ has semistable reduction. Let $\cal{N}'/B'$ be the N\'{e}ron model of $A/K'$. Then we define the stable Faltings height to be 
$$h_{\Fal}(A):= \frac{1}{[K':K]} \deg(  e^* \Omega^g_{\cal{N}'/B'})$$
where $e: B' \to \cal{N}'$ is the identity section and $\Omega_{\cal{N}'/B'}^g$ is the $g$th exterior power of the relative cotangent bundle for $\cal{N}'/B'$. This quantity is independent of the choice of $K'$. 

\begin{theorem} \label{UniformIntroTheorem2}
Let $K = k(B)$ be the function field of a curve $B/k$, where $\operatorname{char} k=0$. There exists a $c = c(g,\gon(B)) > 0$, depending only on $g$ and $\gon(B)$, such that if $A/K$ is a $g$-dimensional abelian variety with no $K$-isotrivial part, $L$ is a symmetric ample line bundle on $A$, and $x \in A(K)$ is a point such that $\bb{Z} \cdot x$ is Zariski dense in $A$, then
$$\h_{A,L}(x) \geq c \max\{h_{\Fal}(A),1\}.$$
\end{theorem}

For Theorems \ref{UniformIntroTheorem1} and \ref{UniformIntroTheorem2}, it is easy to reduce to the case of $k = \bb{C}$. There have been many partial results toward Theorem \ref{UniformIntroTheorem1}. Bakker and Tsimerman \cite{BT18} prove it for abelian varieties with real multiplication. Over function fields of characteristic $0$, uniform bounds on the maximal $n$ for which a level $n$-structure exists over $K$, i.e., for which all $n^{2g}$-torsion points are $K$-rational, have been proven by Nadel \cite{Nad89} when the genus $g(B)$ equals $0$ or $1$, and by Hwang--To \cite{HT06} for general $B$. The methods used in these articles rely on studying the hyperbolicity of the moduli space that parametrizes abelian varieties with an $n$-torsion point/level $n$ structure and thus do not carry over to the arithmetic setting.
\par 
Our approach is different from the ones above. Its underlying philosophy is the same as that used in Hindry--Silverman's proof of Theorem \ref{UniformIntroTheorem2} for elliptic curves \cite{HS88}. These methods largely transfer to the number field setting. The only place where we use that $K$ is a function field in a crucial way is in invoking the Arakelov inequality. The Arakelov inequality was first proven by Faltings \cite{Fal83} and then refined by Deligne \cite{Del87}. 

\begin{theorem} \label{IntroArakelovTheorem1} \textnormal{(Arakelov Inequality)}
Let $K = \bb{C}(B)$ be the function field of a curve $B/\bb{C}$ and let $A/K$ be a $g$-dimensional abelian variety with semistable reduction. Let $S$ be the set of places of bad reduction of $A/K$. If $2g(B)-2+|S| \geq 0$, then 
$$h_{\Fal}(A) \leq \frac{g}{2}(2g(B) - 2 + |S|).$$
\end{theorem}

One might conjecture that an analogue of Theorem \ref{IntroArakelovTheorem1} should hold over number fields. For $g = 1$, this is Szpiro's conjecture \cite{Szp79} which is equivalent to the $abc$ conjecture. The fact that Szpiro's conjecture would imply Theorem \ref{UniformIntroTheorem1} for elliptic curves over number fields was already observed by Frey \cite{Fre89}. Hindry--Silverman \cite{HS88} further show that Szpiro's conjecture implies \eqref{eqn:Lang}. Some partial results in this direction in higher dimensions were also obtained by Clark and Xarles \cite{CX08}. 
\par 
In \cite{HP16, HP22}, Hindry and Pacheco conjectured that a positive characteristic version of Theorem \ref{IntroArakelovTheorem1} holds under some suitable hypotheses, generalizing a prior result of Szpiro \cite{Szp79}. It should be possible to obtain analogous results in positive characteristic if their conjecture is true. 

\begin{conjecture}[Hindry--Pacheco \cite{HP16, HP22}] \label{IntroCharpConjecture} 
Let $K = \ovl{\bb{F}}_p(B)$ be the function field of a curve $B$ over $\ovl{\bb{F}}_p$ and let $A/K$ be a $g$-dimensional abelian variety with semistable reduction and nonzero Kodaira-Spencer map for which $\Tr_{K/\ovl{\bb{F}}_p}(A) = 0$. Let $S$ be the set of places of bad reduction of $A$. Then there are constants $c_1(g),c_2(g,K)>0$ such that if $p\ge c_1$, then $$h_{\Fal}(A) \leq c_2(|S|+1).$$ 
\end{conjecture}

When $K$ is a number field, there have been many results toward bounding the number of $K$-torsion points or the N\'{e}ron--Tate height of a non-torsion $K$-rational point in terms of the Faltings height $h_{\Fal}(A)$ \cite{Mas84, Dav91, Paz13, GR25}. We note that both \cite{Dav91} and \cite{Paz13} have bounds that are uniform over certain families of abelian varieties, proving special cases of Theorem \ref{UniformIntroTheorem1} and \ref{UniformIntroTheorem2} for number fields.

\subsection{Proof Strategy} As mentioned earlier, our proof strategy for Theorems \ref{UniformIntroTheorem1} and \ref{UniformIntroTheorem2} follows that of Hindry and Silverman. This argument has been generalized and applied to arithmetic dynamics in various uniformity settings \cite{Ben07, Ing09, Loo19, Loo21, Loo21b, Yap25} and we will follow this more modern viewpoint. First, we may reduce to the case where $k = \bb{C}$; see Section \ref{sec: ZarhinTrick}. Let $K = \bb{C}(B)$. Given a place $v \in M_K$, we let $A_v^{\an}$ denote the Berkovich analytification of $A$ over $\bb{C}_v$ the completion of $\ovl{K_v}$. Let $\cal{N}_v/R_v$ be the N\'{e}ron model of $A$ over the valuation ring $R_v$ of $K_v$, and let $\cal{N}_{v,s}$ denote its special fiber. Our argument exploits the following two properties regarding points on $A(K)$ with small N\'{e}ron--Tate height. 

\begin{enumerate}
\item\label{item:iintro} If $v\in M_K$ is a place of stable bad reduction and the number of components of $\cal{N}_{v,s}$ is small, then any point $x \in A(K)$ is subject to strong restrictions on where it lives in $A_v^{\an}$.

\item\label{item:iiintro} Any generic sequence of $\Gal(\ovl{K}/K)$-stable finite sets of points $F_n\subseteq A(\overline{K})$ with N\'eron--Tate heights tending to $0$ and $|F_n| \to \infty$ equidistributes to the Haar measure on the skeleton of $A_v^{\an}$. 
\end{enumerate} If $F$ is a large set of $K$-rational points, these two properties are in tension with each other. If one has some uniform control on this tension as $A$ varies, then this yields a uniform upper bound on $|F|$. 
\par
We first restrict to the case of elliptic curves. For simplicity, let us assume that our elliptic curve $E/K$ has semistable reduction and that $E$ has at least one place of bad reduction. Hindry and Silverman first apply the Arakelov inequality (or Szpiro's conjecture if one is working over a number field) to conclude that there is a subset $T \subseteq M_K$ with $|T| \geq \frac{h(E)}{M}$ for some uniform $M = M(g(B))>0$, such that $(i)$ holds for all $v\in T$. By ``small" in \ref{item:iintro} we mean that the number of components is uniformly bounded by some integer $D$. Then replacing our small points $x$ with $[D!]x$, we may assume that for all $v \in T$, the reductions of our points lie in the identity component of the N\'{e}ron model.
\par 
To exploit \ref{item:iiintro}, given $N$ distinct points $x_1,\ldots,x_N \in E(K)$, Hindry and Silverman look at the averaged quantity 
$$\frac{1}{N(N-1)} \sum_{i \not = j} \hat{\lambda}_v(x_i - x_j)$$
where $\hat{\lambda}_v$ is the local N\'{e}ron--Tate height for the place $v$. The two main properties of this average are that if $E_N$ is a sequence of sets in $E(\overline{K})$ with $|E_N|=N$, then $$\liminf_{N \to \infty} \frac{1}{N(N-1)} \sum_{i \not = j} \hat{\lambda}_v(x_i - x_j) \geq 0,$$
and if equality holds, then the sets $E_N$ must equidistribute to the Haar measure $\mu_{E,v}$ on $E_v^{\an}$ \cite{BP05}. However, for $v \in T$, equidistribution cannot happen as our points all retract to a finite set of points on $E_v^{\an}$. This leads to a positive average: indeed, by a direct computation, Hindry and Silverman show that when $E_N$ consists solely of points reducing to the identity component of the N\'{e}ron model at $v$, we must have 
\begin{equation} \label{eq: IntroPos1}
\frac{1}{N(N-1)} \sum_{i \not = j } \hat{\lambda}_v(x_i - x_j) \geq \frac{1}{12} \log^+|j(E)|_v.
\end{equation} On the other hand, for any nonarchimedean $v \not \in T$, we have the so called Elkies lower bound \cite{Fal84, BR07}
\begin{equation} \label{eq: IntroElkies1}
\frac{1}{N(N-1)} \sum_{i \not = j} \hat{\lambda}_v(x_i - x_j) \geq -\frac{C}{N} \log^+|j(E)|_v
\end{equation}
where $C > 0$ is independent of $E$. As $|T| \geq \frac{h(E)}{M}$, this leads to a contradiction after summing over all $v \in M_K$ for $N$ sufficiently large; indeed, for any set $\{x_1,\dots,x_N\}$ of distinct $K$-rational points of $E$,
$$\frac{1}{N(N-1)} \sum_{v \in M_K} \sum_{i \not = j} \hat{\lambda}_v(x_i - x_j) = \frac{1}{N(N-1)} \sum_{i \not = j} \h_E(x_i - x_j),$$
and this quantity is small if $\h_E(x_i)$ is small for all $i$. 
\par 
To generalize this argument to higher dimensions, there are three main obstacles. The first is to find a suitable analogue of $\hat{\lambda}_v(x-y)$ for higher dimensional abelian varieties. The second is to establish the Elkies lower bound \eqref{eq: IntroElkies1} in terms of this analogue. The third is to prove a corresponding uniform positive lower bound analogous to \eqref{eq: IntroPos1} when the points $x_1,\ldots,x_N$ reduce to the identity component of the N\'{e}ron model that is independent of $A$ and $v$. The greatest challenge lies in finding analogues of these three that are mutually compatible.
\par 
\subsection{Higher Dimensional Arakelov-Green's Functions} The first two obstacles have been overcome by the first author in her recent preprint \cite{Loo24}. Given an abelian variety $A/K$ and a very ample line bundle $L$ by which $A$ is embedded into $\bb{P}^N$, we assume we have an endomorphism $f: \bb{P}^N \to \bb{P}^N$ that extends multiplication by $[m]$ for some $m \geq 2$. Let $\dim H^0(A,L^n) = c(n)$. Given a basis $\cal{B}_n \subseteq H^0(A,L^n)$, we may define a higher-dimensional Arakelov-Green's function  
$$g_{\cal{B}_n,v}: (A^{\an}_v)^{c(n)} \to \bb{R} \cup \{+\infty\}$$
which generalizes the averaged quantity $\sum_{i \not = j} \hat{\lambda}_v(x_i- x_j)$. The value 
$$g_{\cal{B}_n,v}(x_1,\ldots,x_{c(n)})$$ 
will be a real number unless there is a nonzero $s \in H^0(A,L^n)$ for which $x_1,\ldots,x_{c(n)}$ all lie on $\divv(s)$, in which case it will be $+\infty$. If a different basis $\cal{B}'_n$ is used, then $g_{\cal{B}'_n,v}$ differs from $g_{\cal{B}_n,v}$ by a constant. 
\par 
In \cite{Loo24}, inspired by earlier work of Baker \cite{Bak06} for the case of endomorphisms on $\bb{P}^1$, the first named author constructed a sequence of special bases $\{\cal{B}_n\}_{n=1}^{\infty}$, depending on the data of a homogeneous lift $F: \bb{A}^{N+1} \to \bb{A}^{N+1}$ of $f$, such that for every place $v \in M_K$, we have the inequality 
\begin{equation} \label{eq: IntroElkies2}
g_{\cal{B}_n,v}(x_1,\ldots,x_{c(n)}) \geq -C_{f,v} \frac{\log n}{n} + r(F)
\end{equation}
where $C_{f,v}$ is a constant depending on $f$ and $v$ that can be made explicit. Here, $r(F)$ is a negative multiple of $\log |\Res(F)|$ where $\Res$ is the Macaulay resultant, chosen so that $g_{\cal{B}_n,v}$ itself does not depend on the lift $F$. Equation \eqref{eq: IntroElkies2} may be viewed as an analogue of Elkies lower bound \eqref{eq: IntroElkies1}. This construction has also been recently used by the second author \cite{Yap24} to prove a quantitative version of Yuan's equidistribution theorem \cite{Yua08} for polarized dynamical systems over archimedean places. 
\par 
Using results of Boucksom and Eriksson \cite{BE21} and their extensions by Boucksom--Gubler--Martin \cite{BGM22}, the first author showed that if one has a sequence of $c(n)$-tuples $\{\vec{P}_n\}_{n=1}^{\infty}$, and if
$$\liminf_{n \to \infty} g_{\cal{B}_n,v}(\vec{P}_n) = r(F)$$
for the special bases $\cal{B}_n$, then the tuples $\vec{P}_n$ must equidistribute to the non-archimedean canonical measure $\mu_{A,v}$ on $A_v^{\an}$.
\par 
We now give some background on the equidistribution of small points. Equidistribution of points with small N\'{e}ron--Tate height towards the Haar measure was first observed and proven by Szpiro--Ullmo--Zhang \cite{SUZ97} over archimedean places, and was then used to prove the Bogomolov conjecture by Ullmo \cite{Ulm98} and Zhang \cite{Zha98}. For elliptic curves $E$, a non-archimedean equidistribution theorem was established by Baker--Petsche \cite{BP05}, and analogous results for dynamical systems on $\bb{P}^1$ were proven by Baker--Rumely \cite{BR06}, Autissier--Chambert-Loir \cite{Aut01, CL06} and Favre--Rivera-Letelier \cite{FRL06}. For higher-dimensional abelian varieties (and more generally polarized dynamical systems), Chambert-Loir \cite{CL06} defined a probability measure $\mu_{A,v}$ that lives on the Berkovich analytification $A_v^{\an}$ and plays the role of the Haar measure at archimedean places. 
\par 
In \cite{Gub07, Gub07b, Gub10}, Gubler uses convex and tropical geometry to study $\mu_{A,v}$ and establishes an equidistribution theorem for small points over function fields. We note that Yuan and Zhang have extended the equidistribution theorem to the more general setting of adelic line bundles on quasi-projective varieties \cite{Yua08, YZ24} (see also \cite{Gau23}). Gubler uses his equidistribution result to prove the geometric Bogomolov conjecture for abelian varieties with totally degenerate reduction \cite{Gub07b}. The geometric Bogomolov conjecture was subsequently proven in full by Xie--Yuan \cite{XY22}, building on the work of Yamaki \cite{Yam18}. For function fields over $\bb{C}$, there is a different approach by Gao--Habegger \cite{GH19} and Cantat--Gao--Habegger--Xie \cite{CGHX21}. 
\par 
In another direction, Berman, Boucksom and Witt Nystrom \cite{BBW11}  used the variational approach of Szpiro--Ullmo--Zhang to prove an equidistribution result for Fekete points on complex manifolds. Given a complex manifold $X$ and a hermitian line bundle $L$, one can define a notion of an $n$-transfinite diameter with respect to a basis $\cal{B}_n$ of $H^0(X,L^n)$. Our Arakelov-Green's function $g_{\cal{B}_n,v}$ is essentially $- \log$ of this transfinite diameter. A $c(n)$-tuple is said to be $n$-Fekete if the value of the Arakelov-Green's function evaluated on that tuple realizes the ($-\log$ of the) $n$-transfinite diameter. Any sequence of $n$-Fekete points equidistributes to the equilibrium measure associated to $(X,L)$ as $n \to \infty$. This was extended to the non-archimedean setting by Boucksom--Eriksson \cite{BE21} and Boucksom--Gubler--Martin \cite{BGM22}. We note that Rumely \cite{Rum89} and Rumely--Lau--Varley \cite{RLV00} had previously studied the relation between Arakelov theory and transfinite diameters. 
\par 
Going back to Gubler's work, he shows that if $A$ has stable bad reduction at $v$, then one can define a continuous retraction map $r: A_v^{\an} \to \bb{R}^k/\Lambda$ where $\bb{R}^k/\Lambda$ is a real torus with $k \geq 1$. The dimension $k$ of the real torus is exactly the toric rank of the semiabelian variety corresponding to the special fiber of the N\'{e}ron model of $A$. The pushforward of $\mu_{A,v}$ under $r$ is the Haar measure on $\bb{R}^k/\Lambda$. On the other hand, if $x \in A(K)$ reduces to the identity component of the special fiber at $v$, then $r(x) = 0 \in \bb{R}^k/\Lambda$. In particular, if $\vec{P}_n$ is a tuple with $r(\vec{P}_n) = 0$, then there exist $N, \delta > 0$ depending on $A$ and $v$ such that if $n \geq N$, then
\begin{equation} \label{eq: IntroPos2} 
g_{\cal{B}_n,v}(\vec{P}_n) \geq \delta + r(F).
\end{equation} The main challenge in using \eqref{eq: IntroPos2} for the purposes of this paper is to find a uniform $(N,\delta)$ that makes \eqref{eq: IntroPos2} hold for all $g$-dimensional abelian varieties $A$ over a non-archimedean field $K$.

\subsection{Degeneration by Ultrafilters}  To achieve this, we will use a degeneration argument to show a weaker statement that suffices for our paper. Degeneration arguments have shown to be very useful in proving uniformity results in both arithmetic dynamics and Diophantine geometry. DeMarco--Krieger--Ye \cite{DKY20} and Yuan \cite{Yua24} both use degeneration arguments to prove results toward the uniform Bogomolov conjecture, and Poineau \cite{Poi24a, Poi24b} uses degenerations to prove the Bogomolov--Fu--Tschinkel conjecture. The approach by degenerations differs from the breakthroughs on the  uniform Mordell--Lang conjecture by Dimitrov--Gao--Habbeger \cite{DGH21}, K\"uhne \cite{Kuh21} and Gao--Ge--K\"uhne \cite{GGK21}, which rely on studying the non-degeneracy of families of subvarieties. 
\par 
We will use a method developed by Luo \cite{Luo21, Luo22}, who uses ultrafilters to construct rescaling limits for sequences of rational maps $f_n: \bb{P}^1 \to \bb{P}^1$ of degree $d$ over $\bb{C}$. This was formalized using Berkovich spaces by Favre--Gong \cite{FG24, Gon25}, and it is thus readily adapted to the setting of abelian varieties. The second author has also applied it to prove some uniform boundedness statements in arithmetic dynamics \cite{Yap25}. We remark that there are related constructions by Amini--Nicolussi for the moduli space of curves \cite{AN22, AN24} and also by Poineau \cite{Poi25} and Song \cite{Song25}. 
\par 
We briefly describe the construction here. Let $X = \cal{A}_{g,3}$ be the fine moduli space of principally polarized abelian varieties with level $3$ structure over $\bb{C}$, and fix a compactification $\ovl{X}$ of $X$. Following Silverman \cite{Sil87}, for any valued field $K$ containing $\bb{C}$, we can define a local height function $\lambda_{\partial X}$ with respect to the boundary $\partial X$. This function behaves like $-\log \mathrm{dist}( \cdot , \partial X)$.  
\par 
Now let $(A_n/K_n)$ be a sequence of principally polarized abelian varieties with level $3$ structure over algebraically closed complete non-archimedean valued fields $(K_n, | \cdot |_n)$ contaning $\bb{C}$ as a trivially valued subfield, such that each $A_n$ has bad reduction over $K_n$. Then $\lambda_{\partial X}(A_n) > 0$ and we can thus define $\eps_n = \lambda_{\partial X}(A_n)^{-1}$. For $\eps = (\eps_n)$, we let 
$$\mathscr{A}^{\eps} = \left\{ (x_n) \in \prod_{n=1}^{\infty} K_n: \sup_n |x_n|^{\eps_n} < \infty\right\},$$
which can be viewed as the product of the rescaled valued fields $(K_n, |\cdot |^{\eps_n})$ in the category of Banach rings. Then by Proposition \ref{DegenerationExistence2}, there is a unique morphism $x: \Spec \scr{A}^{\eps} \to X$ such that the precomposition of $x$ with the inclusion $\iota_n: \Spec K_n \to \Spec \scr{A}^{\eps}$ yields the morphism defined by $A_n/K_n$.
\par 
It was already known to Berkovich \cite[Proposition 1.2.3]{Ber90} that the Berkovich spectrum $\cal{M}(\scr{A}^{\eps})$ of $\scr{A}^{\eps}$ under the norm $||(x_n)|| = \sup_n |x_n|^{\eps_n}_n$ is in bijection with the set of ultrafilters $\beta \bb{N}$ under the map 
$$\omega \to \left( (x_n) \mapsto \lim_{\omega} |x_n|^{\eps_n} \right);$$
see also \cite[Theorem 3.8]{FG24}. If $\cal{H}(\omega)$ is the residue field of $\scr{A}^\eps$ for a non-principal ultrafilter $\omega$, we obtain an induced morphism $x_{\omega}: \Spec \cal{H}(\omega) \to X$ which gives us an abelian variety $A_{\omega}$ over $\cal{H}(\omega)$. The field $\cal{H}(\omega)$ is algebraically closed and is a non-archimedean field under $| \cdot |_{\omega}$. One shows that $A_{\omega}$ has bad reduction over $\cal{H}(\omega)$. Observe that $A_{\omega}$ might not be defined over a discretely valued subfield of $\cal{H}(\omega)$, so we require an extension of Gubler's results to general non-archimedean fields. This was achieved by Gubler and Stadloder \cite{GS23}.
\par 
Now let $U$ be a Zariski open of $X$ such that we have a closed embedding $\cal{A}_U \xhookrightarrow{} \bb{P}^N_U$, where $\cal{A}_U$ is the universal abelian scheme over $U$. Assume that we have an endomorphism $f: \bb{P}^N_U \to \bb{P}^N_U$ that lifts multiplication by $[m]$ for some $m \geq 2$. Then after fixing a homogeneous lift $F: \bb{A}_U^{N+1} \to \bb{A}_U^{N+1}$ of $f$, the first author's Arakelov-Green's function construction may be carried out in a relative way over $U$, and we may speak of a consistent choice of a special basis $\cal{B}_k$ for any abelian variety $A/K$ whose associated point in $X$ lies in $U(K)$. 
\par 
For any $C > 0$, we say that an abelian variety $A/K_v$ whose associated point in $X$ lies in $U(K_v)$ is \emph{$C$-centered} if $\lambda_{\partial U}(A) \leq C \lambda_{\partial X}(A)$. The key uniformity result, which serves as the higher-dimensional version of \eqref{eq: IntroPos1}, is the following theorem (see Theorem \ref{UniformNeronBound1}). It is proven by degeneration via ultrafilters. Fix an increasing function $h: \bb{N} \to \bb{N}$ such that $h(n) = c(k_n)$ for some $k_n$ depending on $n$. (Recall that $c(n):=\dim H^0(A,L^n)$.) For $v\in M_K$ a place of stable bad reduction, we say that $A_v:=A\times_K K_v$ is \emph{$n$-good} if given $x_1,\ldots,x_{h(n)} \in A(K)$ that reduce to the identity component of the N\'{e}ron model at $v$, we have 
$$g_{\cal{B}_{k_n},v}(x_1,\ldots,x_{c(k_n)}) \geq \frac{1}{n} \lambda_{\partial X}(A)+r(F_K).$$

\begin{theorem} \label{UniformIntroTheorem3} Let $U$ be a Zariski open of $X=\cal{A}_{g,3}$. For any increasing function $h: \bb{N} \to \bb{N}$ where for each $n$, we have $h(n) = c(k_n)$ for some $k_n\in\bb{Z}_{>0}$, and any $C > 0$, there exists a nonempty finite list of integers $\{n_1,\ldots,n_m\}$ such that for any non-archimedean complete algebraically closed field $K$ containing $\bb{C}$ as a trivially valued subfield, and for any principally polarized abelian variety $A/K$ of dimension $g$ with bad reduction which is $C$-centered with respect to $U$, we have that
$$A/K \text{ is } n_i\text{-good for some } i \text{ with } 1 \leq i \leq m.$$
The set $\{n_1,\ldots,n_m\}$ can depend on the open set $U$ along with the above-constructed rational map $f$ and homogeneous lift $F$. However, it is independent of $K$ and the dimension $g$ principally polarized abelian variety $A/K$ which is $C$-centered with respect to $U$.
\end{theorem}

Theorem \ref{UniformIntroTheorem3} is certainly weaker than asking for a uniform $N$ and $\delta$ such that \eqref{eq: IntroPos2} holds uniformly for all $n \geq N$ and all $g$-dimensional abelian varieties, but it is enough for our purposes, as $\cal{A}_{g,3}$ can be covered with finitely many such $U$. Using this, we prove Theorem \ref{UniformIntroTheorem1} for semistable abelian varieties with at least one place of bad reduction in Section \ref{sec: UniformSemistable}. 
\par 
To extend our argument to the non-semistable case, we obtain positive lower bounds on $g_{\cal{B}_n,v}(x_1,\ldots,x_{c(n)})$ under the assumption that $A$ has bad reduction at $v$, but good reduction after a finite extension. This is handled in Sections \ref{sec: NeronModels} and \ref{section:fakebadbound}. We address the case of everywhere good reduction in Section \ref{sec: GoodReduction}, and then use Zarhin's trick and reduction to the principally polarized case to prove Theorems \ref{UniformIntroTheorem1} and \ref{UniformIntroTheorem2}, with the additional restriction of dependence on the genus $g(B)$. We finally upgrade the bound to depend only on the gonality via Weil restrictions, using the case of $gd$-dimensional abelian varieties to deduce the case of $g$-dimensional abelian varieties over $K = \bb{C}(B)$ with $\gon(B) = d$. 

\subsection{Acknowledgements} We would like to thank Benjamin Church, Laura DeMarco, Chen Gong, Richard Griffon, Alexander Petrov, Bjorn Poonen, Joseph Silverman and Shou-Wu Zhang for helpful discussions regarding the paper. We especially thank Joseph Silverman for his foundational works that have inspired several key parts of this paper. 
\par
The first author was supported by a Sloan Research Fellowship and NSF CAREER grant DMS-2337942.

\section{Background} In this section, we go through some preliminary background. We first review the theory of global and local heights for closed subschemes over quasi-projective varieties over function fields. This was set forth by Silverman \cite{Sil87}; see also the exposition by Matsuzawa \cite{Mat20}. We then recall some facts about abelian varieties and schemes, in particular about the moduli space of principally polarized abelian varieties $\cal{A}_{g,3}$ with level $3$ structure and its compactifications. We finally end with some facts regarding the Hodge bundle on the moduli space of curves.

\subsection{Heights over Function Fields} Let $K = \bb{C}(B)$ be the function field of a smooth projective curve $B$ over $\bb{C}$. We let $M_K$ be the following set of inequivalent nontrivial absolute values: for each $v \in B(\bb{C})$, we may define an absolute value $| \cdot |_v: K^{\times} \to \bb{R}$ via
$$|f|_z := e^{-\ord_z(f)} \text{ for } f \in K^{\times},$$
where $\ord_z(f)$ is the order of vanishing of $f$ at $z$. The set $M_K$ satisfies the product formula; that is to say, for any $f \in K^{\times}$ we have 
$$\sum_{v \in M_K} \log |f|_v = 0.$$
This allows us to define a global height function on $h: \bb{P}^n(K) \to \bb{R}_{\geq 0}$ by letting
$$h([x_0: \cdots : x_n]) = \sum_{v \in M_K} \log \max\{|x_0|_v,\ldots,|x_n|_v\},$$
where the product formula ensures that this is well-defined independent of our choice of $x_i$. Under suitable normalizations \cite[Remark 1.5.22]{BG06}, we may extend this to a well-defined function $h: \bb{P}^n(\ovl{K}) \to \bb{R}_{\geq 0}$, which we will call the standard Weil height. Unlike the case of number fields where every number field is canonically realized as an extension of $\bb{Q}$, there is no canonical realization of the function field $K$ as an extension of $\bb{C}(t)$, and hence our height function depends on a choice of reference field. If $K'/K$ is a finite extension and we take $K'$ to be our reference field, we have $h_{K'}(x) = [K':K]h_K(x)$ for all $x \in \bb{P}^n(K')$.
\par 
In general, given a projective variety $X$ over $K$ and an line bundle $L$, we may define a height function $h_L: X(\ovl{K}) \to \bb{R}$. However, this height function is only well-defined up to an $O(1)$ term and depends on some choice of the data of $L$. However, if both $X$ and $L$ are defined over $\bb{C}$, then one has a canonical choice of $h_L$. The same holds for height functions $h_Z$ associated to closed subscehemes $Z$ of $X$ \cite{Sil87}, as we now explain.  
\par 
Let $X$ be a projective variety over $\bb{C}$ and $D$ a Cartier divisor on $X$. Let $L,M$ be basepoint free line bundles on $X$ for which $\mathcal{O}(D) \simeq L \otimes M^{-1}$. Fix a meromorphic section $s_D$ of $\mathcal{O}(D)$. Picking bases of the spaces of global sections $s=\{s_0,\ldots,s_n\}$ of $L$ and $t=\{t_0,\ldots,t_m\}$ of $M$, we may define the local height for any valued field $K\supseteq\mathbb{C}$ relative to the presentation $\cal{D} = (s_D,L,s,M,t)$ to be 
$$\lambda_{\cal{D}}(p) = \max_{k} \min_{l} \log \left|\frac{s_k}{t_l s_D}(p)\right|$$
for all $p \in (X \setminus D)(K)$, where $\frac{s_k}{t_l s_D}$ is a rational function on $X$. A priori, this depends on the presentation $\cal{D}$. In the usual setting of heights over global fields, one shows that the local heights for two different presentations differ by an $O(1)$ term. However if we assume that $\bb{C}$ is trivially valued in $K$, they turn out to be equal.

\begin{proposition} \label{LocalHeightEqual1}
Let $\cal{D},\cal{D}'$ be two presentations that are defined over $\bb{C}$ and let $K$ be a non-archimedean field containing $\bb{C}$ as a trivially valued subfield. Then $\lambda_{\cal{D}} = \lambda_{\cal{D}'}$ on $(X \setminus D)(K)$.  
\end{proposition}

\begin{proof}
By \cite[Remark 2.2.13]{BG06}, there exists a finite list of elements $p_1,\ldots,p_m \in \bb{C}^*$ such that 
$$|\lambda_{\cal{D}} - \lambda_{\cal{D}'}| \leq \log^+\max\{|p_1|,\ldots,|p_m|\}$$
where $| \cdot |$ is the absolute value from $K$. Since $\bb{C}$ is trivially valued, it follows that $\lambda_{\cal{D}} = \lambda_{\cal{D}'}$.
\end{proof}

Now pick a local affine chart $U$ over $\bb{C}$ with affine coordinates $(a_1,\ldots,a_N)$, where our Cartier divisor $D$ is given locally by the function $f\in\mathcal{O}(U)$. Then we have a simple formula for $\lambda_{\cal{D}}$ when evaluated on points $(a_1,\ldots,a_N)$ with $|a_i| \leq 1$.

\begin{proposition} \label{LocalHeightFunction1}
Let $U \subseteq X$ be an affine open over $\bb{C}$ with affine coordinates $a_1,\ldots,a_N$ such that our Cartier divisor $D$ is given by $f\in\mathcal{O}(U)$. If $x \in U(K)$ has coordinates $(x_1,\ldots,x_N)$ with $|x_i| \leq 1$, then 
$$\lambda_{\cal{D}}(x) = \log |f(x)|^{-1}.$$
\end{proposition}

\begin{proof}
First assume that the line bundles $L$ and $M$ are trivialized by $U$. Let $B$ be the bounded set $B = \{x\in U:|x_1|,\ldots,|x_N| \leq 1\}$, and let $C_B(f)=\sup_{x\in B}\{|f(x)|\}$. Since $\mathbb{C}$ is trivially valued, we have $C_B(f)\le 1$. By the proof of \cite[Lemma 4.24]{Mat20}, we obtain $$\lambda_{\cal{D}}(x) = \log |f(x)|^{-1}$$
for all $x \in B$ as desired.
\par 

We now remove the assumption that $L$ and 
$M$ are trivial on all of $U$.
Because $L$ and $M$ are locally free of rank 
$1$, we may cover $U$ by finitely many distinguished affines $D(g_1),\ldots,D(g_n)$,
$g_i \in \mathcal{O}(U)$ such that both $L$ and 
$M$ are trivial on each $D(g_i)$. Since $\{D(g_i)\}$ covers $U$, the functions $g_i$ generate the unit ideal in $\mathcal{O}(U)$, so there exist coefficients $c_i\in \mathbb{C}$ with \[\sum_i c_i g_i = 1\quad\text{in }\mathcal{O}(U).\] Let \[B=\{x\in U(K)\mid |x_j|\le 1\text{ for all affine coordinates } x_j \text{ on }U\},\] and fix $x\in B$. Evaluating the identity $\sum c_i g_i = 1$ at $x$, and using that $\mathbb{C}$ is trivially valued, yields \[1=\left|\sum_i c_i g_i(x)\right|=\max_i |c_i g_i(x)|=\max_i |g_i(x)|.\] Hence at least one $g_{i_0}$
satisfies $|g_{i_0}(x)| = 1$. Thus $x \in D(g_{i_0})(K)$, and in this affine chart both $L$ and $M$ are trivial by hypothesis. On the distinguished affine $D(g_{i_0})$, the coordinate ring is $\mathcal{O}(U)[g_{i_0}^{-1}]$,
so the coordinates on $D(g_{i_0})$
are the original affine coordinates $x_1,\ldots,x_N$ on $U$,
together with the function $g_{i_0}^{-1}$.
For our specific point $x$, we have $|x_j|\le 1$ by definition of $B$, and $|g_{i_0}(x)^{-1}| = 1$
because $|g_{i_0}(x)| = 1$. Thus all coordinates of $x$ on $D(g_{i_0})$ have norm $\le 1$. Since $L$ and $M$ are trivial on this distinguished affine, the conclusion of the first paragraph applies to $x$, yielding \[\lambda_{\mathcal{D}}(x) = \log |f(x)|^{-1}.\] This completes the proof.\end{proof}

We now move on to local heights associated to closed subschemes. Let $Y$ be a closed subscheme of $X$ defined over $\bb{C}$ and let $Y = D_1 \cap \cdots \cap D_r$ be the scheme-theoretic intersection of $D_1,\dots,D_r$, where each $D_i$ is a Cartier divisor defined over $\bb{C}$. Choose a collection $\cal{Y} = (\cal{D}_1,\ldots,\cal{D}_r)$ of presentations of the $D_i$, so that we obtain local height functions $\lambda_{\cal{D}_i}$. We then define
\begin{equation}\label{eqn:intht}\lambda_{\cal{Y}}=\min\{\lambda_{\cal{D}_1}, \lambda_{\cal{D}_2},\ldots, \lambda_{\cal{D}_r}\}.\end{equation}

\begin{proposition} \label{LocalHeightsEqual2}
For any two presentations $\cal{Y}$ and $\cal{Y}'$, the local height functions $\lambda_{\cal{Y}}$ and $\lambda_{\cal{Y}'}$ are equal. 
\end{proposition}

\begin{proof}It suffices to prove the case where $\cal{Y}' = \cal{Y} \cup \{\cal{E}\}$ for $\cal{E}$ a presentation of some Cartier divisor $E$. Let $U$ be an affine open for which the presentations $\cal{D}_i,\cal{E}$ are trivialized. By \cite[Claim 4.26]{Mat20}, on the bounded subset $B$ of $U(K)$ (points whose affine coordinates have norm $\le 1$),
we have \[\lambda_{\mathcal Y'}(x)\ge\lambda_{\mathcal{Y}}(x)-\gamma(U,B;\mathcal{Y},\mathcal{E})\] for some constant $\gamma(U,B;\mathcal{Y},\mathcal{E})$ depending on $U,B,\mathcal{Y}$, and $\mathcal{E}$. By Proposition~\ref{LocalHeightFunction1}, the correction terms appearing in $\gamma(U,B;\mathcal{Y},\mathcal{E})$ vanish, so $\gamma(U,B;\mathcal{Y},\mathcal{E})=0$ and hence $\lambda_{\mathcal Y'}\ge\lambda_{\mathcal Y}$ on $B$.  

The reverse inequality is immediate from the definitions of $\lambda_{\mathcal{Y}}$ and $\lambda_{\mathcal{Y}'}$, so $\lambda_{\mathcal{Y}}=\lambda_{\mathcal{Y}'}$ on $B$. By the proof of \cite[Theorem 2.2.11]{BG06}, we may find a cover $\{U_i\}$ of $X$ such that the union of the bounded subsets $B_i$ of each $U_i$ contains all points of $X$. Hence we have $\lambda_{\mathcal{Y}} = \lambda_{\mathcal{Y}'}$ globally as desired.
\end{proof}

As a result, we obtain the following description of the local height functions of a closed subscheme when restricted to an affine open. 

\begin{proposition} \label{LocalHeightFunction2}
Let $U \subseteq X$ be an affine open over $\bb{C}$ with affine coordinates $a_1,\ldots,a_N$ such that on $U$, the closed subscheme $Y=D_1\cap\cdots\cap D_r$ is cut out by the functions $f_1,\ldots,f_m \in\mathcal{O}(U)$. If $x \in U(K)$ has coordinates $(x_1,\ldots,x_N)$ with $|x_i| \leq 1$, then we have 
$$\lambda_{\cal{Y}}(x) = \min_{1 \leq i \leq m}\{\log |f_i(x)|^{-1}\}$$ for any presentation $\mathcal{Y}$ of the Cartier divisors $D_i$.
\end{proposition}

\begin{proof} For $\mathcal{Y}=(\mathcal{D}_1,\dots,\mathcal{D}_r)$, we have $\lambda_{\cal{Y}}=\min \lambda_{\mathcal{D}_i}$, and by Proposition \ref{LocalHeightsEqual2}, $\lambda_{\mathcal{Y}}$ is independent of the choice of presentation. Restricting to the affine open $U$, we know by Proposition \ref{LocalHeightFunction1} that if $D_i$ is given by $g_i$, we have $\lambda_{\mathcal{D}_i} = \log |g_i(x)|^{-1}$ if $x \in U(K)$ has affine coordinates satisfying $|x_i|\leq1$ for all $i$. It thus suffices to show that 
$$\min_{1 \leq i \leq r} \{ \log |g_i(x)|^{-1}\} = \min_{1 \leq i \leq m} \{ \log |f_i(x)|^{-1}\}.$$
Since, on $U$, the closed subscheme $Y$ is cut out by $\{f_1,\ldots,f_m\}$ and also by $\{g_1,\ldots,g_{r}\}$, it follows that the ideals of $\mathcal{O}(U)$ generated by $\{f_1,\ldots,f_m\}$ and $\{g_1,\ldots,g_{r}\}$ are equal. Thus for any $f_i$, there exist $c_j \in \bb{C}$ such that $f_i = \sum_{j=1}^{r} c_j g_j$. Since $|c_j|=1$ for all $j$, it follows that 
$$|f_i(x)| \leq \max_{1 \leq j \leq r} |g_j(x)| \implies \log |f_i(x)|^{-1} \geq \min_{1 \leq j \leq r} \log |g_j(x)|^{-1}.$$
Thus 
$$\min_{1 \leq i \leq m} \log |f_i(x)|^{-1} \geq \min_{1 \leq j \leq r} \log |g_j(x)|^{-1}.$$
The reverse inequality follows similarly and we get equality as desired. 
\end{proof}

We now list some basic properties that these local heights associated to closed subschemes satisfy. These can all be found in \cite{Sil87}, but the main thing we want to emphasize as per Proposition \ref{LocalHeightsEqual2} is that our local height functions are canonical. We thus drop the presentations in our notation and simply write $\lambda_Y$ in lieu of $\lambda_{\mathcal{Y}}$, etc. Furthermore, the constants appearing in the bounds only depend on the geometry of $X$ over $\bb{C}$ and not on the non-archimedean field $K$. 

\begin{proposition} \label{LocalHeightProperties1}
Let $X$ be a projective variety over $\bb{C}$ and let $Y,Z$ closed subschemes of $X$. Then for any non-archimedean field $K$ with $\bb{C}$ a trivially valued subfield, we have functions 
$$\lambda_Y, \lambda_Z: X(K) \to \bb{R} \cup \{+\infty\}$$

satisfying the following properties:

\begin{enumerate}
\item If $Z \subseteq Y$ then $\lambda_Z \leq \lambda_Y$.

\item If $\Supp Z \subseteq \Supp Y$, then there is a constant $c > 0$, independent of $K$, such that $\lambda_Z \leq c \lambda_Y$.

\item Let $\vphi: X \to X'$ be a morphism of projective varieties over $\bb{C}$. Let $Y'$ be a closed subscheme of $X'$. Then $\lambda_{\vphi^* Y} = \lambda_{Y'}.$

\item Let $U \subseteq X$ be an open set and let $W \subseteq X$ be the complement of $U$. If $Z \vert_U \subseteq Y \vert_U$, then there exists a constant $c > 0$, independent of $K$, such that $\lambda_Z \leq \lambda_Y + c \lambda_W$.
\end{enumerate}
\end{proposition}

\begin{proof}
The first claim is clear as given any presentation $\cal{Y}$ of $Y$, we may add $\cal{Y}$ to any presentation $\cal{Z}$ of $Z$ and this does not change $\lambda_Z$. As the local heights $\lambda_Z$ and $\lambda_Y$ are defined by taking minima as in \eqref{eqn:intht}, it follows that $\lambda_Z \leq \lambda_Y$. 
\par 
The second claim follows similarly. Since $\Supp Z \subseteq \Supp Y$, the Nullstellensatz implies that $nZ \subseteq Y$ for some positive integer $n>0$. As $\lambda_{nY} = n \lambda_Y$, we obtain $\lambda_Z \leq n \lambda_Y$ as desired.
\item 
The third claim follows immediately from definitions. For the fourth claim, it suffices to find $n$ such that $Z \subseteq Y + nW$, which is done in \cite[Theorem 2.1(g)]{Sil87}. 
\end{proof}

We now state Lemma 5.1 of \cite{Sil87}. 

\begin{proposition} \label{LocalHeightInequality1}
Let $\vphi: X \to X'$ be a dominant rational map of projective varieties. Let $U\subseteq X$ be an open for which $\vphi \vert_U$ is a morphism. Let $Z$ be the reduced closed subscheme for $X \setminus U$. For a closed subscheme $Y$ of $X'$, let $\vphi^* Y$ be the Zariski closure of the preimage of $Y$ under $\vphi\vert_U$. Then there exists a constant $c > 0$ such that for all non-archimedean fields $K$ with $\bb{C}$ a trivially valued subfield, we have 
$$|\lambda_Y \circ \vphi - \lambda_{\vphi^* Y}| \leq c \lambda_Z$$
on $U$.
\end{proposition}

\begin{proof}
There is a sequence of blowups of $X$ with centers contained in $Z$ yielding a birational map
$\pi:\widetilde{X}\dashrightarrow X$
such that the composite
$\varphi\circ\pi:\widetilde{X}\dashrightarrow X'$
extends to a morphism
$\widetilde{\varphi}:\widetilde{X}\to X'.$
Moreover, $\pi$ is an isomorphism over $U$ and on $U':=\pi^{-1}(U)$ we have $\widetilde{\varphi}=\varphi\circ\pi$. Let $Z'=(\tilde{X}\setminus U')_{\mathrm{red}}$ be the reduced closed subscheme of $\tilde{X}\setminus U'$. Then we have 
$$|\lambda_{(\vphi \circ \pi)^*Y} - \lambda_{\tilde{\vphi}^* Y}| \leq c \lambda_{Z'}$$
for some $c > 0$ by (iv) of Proposition \ref{LocalHeightProperties1}, as $(\vphi \circ \pi)^* Y$ and $\tilde{\vphi}^*Y$ are both equal on $U'$. Since $\Supp Z' = \Supp \pi^*Z$, we have 
$$c' \lambda_Z \circ \pi = c' \lambda_{\pi^* Z} \geq \lambda_{Z'}$$
for some $c' > 0$. Thus 
\begin{equation}\label{eqn:Zhtbd}\lambda_{(\vphi \circ \pi)^*Y} - \lambda_{\tilde{\vphi}^*Y} \leq (cc') \lambda_Z \circ \pi.\end{equation} Pushing forward along the isomorphism $\pi:U'\to U$ gives, for $u\in U$,
\[\lambda_{\widetilde{\varphi}^*Y}(\pi^{-1}(u))
=\lambda_{\varphi^*Y}(u),
\]
because on $U'$ the divisor $\widetilde{\varphi}^*Y$ is the pullback of $\varphi^*Y$ via $\pi$. Similarly,
\[\lambda_{(\varphi\circ\pi)^*Y}(\pi^{-1}(u))
=\lambda_Y(\varphi(u)),
\]
because on $U'$ we have $\lambda_{(\varphi\circ\pi)^*Y}=\lambda_Y\circ(\varphi\circ\pi)$. Combining with \eqref{eqn:Zhtbd} gives
$$|\lambda_Y \circ \vphi - \lambda_{\vphi^* Y}| \leq (cc') \lambda_Z$$
as desired. 
\end{proof}

Now let $X$ be a quasi-projective variety over $\bb{C}$ and let $\ovl{X}$ be a compactification of $X$, that is to say, a proper variety containing $X$ as an open dense subscheme. Letting $\partial X$ be the reduced closed subscheme whose support is $\ovl{X} \setminus X$, we obtain a local height function $\lambda_{\partial X}$ that depends on our compactification. It turns out that any two such local height functions are commensurate.

\begin{proposition} \label{LocalHeightBoundary1}
Let $\lambda_{\partial X}, \lambda'_{\partial X}$ be two such local height functions associated to two different compactifications. Then there exists a $c \geq 1$ such that 
$$c^{-1} \lambda_{\partial X} \leq \lambda'_{\partial X} \leq c \lambda_{\partial X}$$
for all non-archimedean fields $K$ with $\bb{C}$ a trivially valued subfield. 
\end{proposition}

\begin{proof}
Let $\ovl{X}$ and $\ovl{X}'$ be the two compactifications yielding $\lambda_{\partial X}$ and $\lambda'_{\partial X}$ respectively. We apply Proposition \ref{LocalHeightInequality1} to the rational map $\vphi: \ovl{X} \to \ovl{X}'$ which is the identity on $X$. Letting $Y = \ovl{X}' \setminus X$, we see that since $\vphi^*Y=\emptyset$, the local height function $\lambda_{\varphi^*Y}$ associated to it is $0$. Hence taking $Z=\overline{X}\setminus X$ and applying Proposition \ref{LocalHeightInequality1}, we obtain 
$$|\lambda_Y \circ \vphi | \leq c \lambda_Z.$$
But $\lambda_Y \circ \vphi$ is simply $\lambda_{Z'}$ on $X$, where $Z'=\overline{X}'\setminus X$, and so we obtain $\lambda_{Z'} \leq c \lambda_Z$. The other inequality follows from reversing the roles of $Z$ and $Z'$. 
\end{proof}

Given a product formula field $K$, $X$ quasi-projective over $K$, and $Z \subseteq X$ a closed subscheme, we may associate to it a global height function $h_Z$, which for all $x\in (X\setminus Z)(K)$ is given by
\begin{equation}\label{localglobalht} h_Z(x)=\sum_{v \in M_K}\lambda_{Z,v}(x),\end{equation}
where $\lambda_{Z,v}$ is the local height function $\lambda_Z$ for the field $K_v$. In general, $h_Z(x)$ is only canonical up to a constant term, but in the special case where $K_v$ contains $\bb{C}$ as a trivially valued subfield for all $v\in M_K$ and $X,Z$ are defined over $\bb{C}$, we have seen that each $\lambda_{Z,v}$ is independent of the underlying presentation of $Z$, and hence so is $h_Z$. The right-hand side of \eqref{localglobalht} is a priori only defined for $x\in(X \setminus Z)(K)$, but if $Z$ is a Cartier divisor, then one may choose a linearly equivalent divisor $D'$ whose support does not meet the given $x$. It turns out that $h_D=h_{D'}$ and so for Cartier divisors, we obtain a well-defined height function on $X(K)$. We still have the following inequality relating heights associated to ample divisors with those associated to closed subschemes. 

\begin{proposition} \label{GlobalHeightAmple1}
Let $D$ be an ample divisor and $Z \subseteq X$ a closed subscheme defined over $\bb{C}$. Then there exists a constant $c > 0$ such that for any product formula field $K$ with $\bb{C}$ a trivially valued subfield of $K_v$ for all $v \in M_K$, we have 
$$h_Z \leq c h_D$$
for all $x \in (X \setminus Z)(K).$
\end{proposition}

\begin{proof}
We write $Z = \cap_{i=1}^{n} E_i$ as an intersection of Cartier divisors $E_i$. As $\lambda_Z \leq \lambda_{E_i}$ on $(X \setminus E_i)(K)$, it follows that $h_Z \leq h_{E_i}$ on $(X \setminus E_i)(K)$. On the other hand, for each $i$, there exists a constant $c_i$ such that $h_{E_i} \leq c_i h_D$; for example, we may choose $c_i$ such that $c_i D - E_i$ is very ample. Thus taking $c = \max\{c_1,\ldots,c_n\}$, we obtain 
$$h_Z \leq c h_D$$
on $(X \setminus E_i)(K)$ for each $E_i$. As $\cap_{i=1}^n E_i = Z$, it follows that $h_Z \leq c h_D$ on $(X \setminus Z)(K)$ as desired.
\end{proof}

Another comparison theorem that we need is the following. 

\begin{proposition} \label{GlobalHeightAmple2}
Let $X,Y$ be quasi-projective varieties over $\bb{C}$ and $\vphi: X \to Y$ a morphism. Let $\ovl{X},\ovl{Y}$ be compactifications of $X,Y$ respectively and let $L,M$ be ample line bundles on $\ovl{X},\ovl{Y}$. Then there exists a constant $c > 0$ such that for any product formula field $K$ for which $\bb{C}$ is a trivially valued subfield of $K_v$ for all $v \in M_K$, we have  
$$h_M \circ \vphi \leq c h_L$$
for all $x \in X(K)$. 
\end{proposition}

\begin{proof}
By the linearity relation $h_{M^{\otimes \ell}}=\ell h_M$ for $\ell\in\mathbb{N}$, we may assume without loss that $M$ is very ample. Let $D_1,\ldots,D_n$ be effective Cartier divisors such that $\mathcal{O}(D_i)$ is isomorphic to $M$ for all $i$ and $\cap D_i=\emptyset$. For each $D_i$, we may consider the local height function $\lambda_{D_i}$. By Proposition \ref{LocalHeightInequality1}, we know that 
$$|\lambda_{D_i}\circ\vphi-\lambda_{\vphi^*D_i}|\leq c_i\lambda_Z$$ for some $c_i>0$, where $\varphi^*D_i$ denotes the Zariski closure in $\ovl{X}$ of the preimage of $D_i$ under $\varphi$ and $Z=\overline{X}\setminus X$. Summing over $v\in M_K$, we obtain 
\begin{equation}\label{eqn:summedcibd}|h_{D_i}\circ\vphi-h_{\vphi^* D_i}|\leq c_i h_Z\end{equation}
for all $x\in(X\setminus\vphi^* D_i)(K)$. As Proposition \ref{GlobalHeightAmple1} implies the existence of a constant $c$ such that $\max\{h_Z,h_{\vphi^*D_i}\}\leq ch_L$ for all $i$, we may let $c'=\max\{c_1,\ldots,c_n\}$ in order to conclude from \eqref{eqn:summedcibd} that $$h_{D_i}\circ\vphi\leq (c'+1)ch_{L}$$
for all $x\in(X\setminus\vphi^*D_i)(K)$. Since $\cap D_i=\emptyset$, it follows that $(\cap\vphi^*D_i)\cap X=\emptyset$, and so by our choice of $D_i$, we have
$$h_M\circ\vphi\leq (c'+1)ch_L$$
for all $x\in X(K)$ as desired. 
\end{proof}

Note that if $L$ is a very ample line bundle on a projective variety $X$ over $\bb{C}$ and $K = \bb{C}(B)$ where $B$ is a smooth projective curve, then any $x \in X(K)$ spreads to a morphism $x: B \to X$ and $h(x)$ is equivalent to $\deg_B(x^* L)$, i.e., the degree of the pullback $x^* L$ as a line bundle on $B$ \cite[Example 2.4.11]{BG06}.

\subsection{Abelian Varieties and Schemes} \label{sec: abelianschemes}
Let $K$ be a field and $A/K$ be an abelian variety. We let $A^t$ denote the dual abelian variety of $A$ over $K$. It can be identified with the connected component $\Pic^0(A/K)$ of the identity of the Picard variety of $A$. Given a line bundle $L$ of $A$, we may define a morphism 
$$\lambda(L): A \to A^t$$
given by sending $a \mapsto t_a^*L \otimes L^{-1}$ where $t_a : A \to A$ is the translation by $a$ map. If $L$ is ample, then $\lambda(L)$ is an isogeny, i.e., a surjective homomorphism with finite kernel. 
\par 
A polarization $\lambda$ of $A$ is a homomorphism $\lambda: A \to A^t$ such that over $\ovl{K}$, it is of the form $\lambda(L)$ for some ample line bundle $L$ on $A_{\ovl{K}}$. We say that the polarization is principal if $\lambda$ is an isomorphism. 
\par 
In general for an abelian scheme $A/S$ where $S$ is an artbirary scheme, it is a theorem of Raynaud \cite[I.1.9]{FC90} that the dual $A^t$ exists and is also an abelian scheme. We may define a polarization of $A/S$ to again be a homomorphism $\lambda: A \to A^t$ such that for each geometric point $\ovl{s}$ of $S$, we have an equality $\lambda_{\ovl{s}} = \lambda(\cal{L}_{\ovl{s}})$ for some ample line bundle $\cal{L}_{\ovl{s}}$ on $A_{\ovl{s}}$. For any polarization $\lambda: A \to A^t$, there exists a symmetric ample invertible sheaf $\cal{L}$ on $A$ such that $\lambda(\cal{L}) = 2 \lambda$ \cite[Proposition 6.10]{MFK94}, i.e., such that $\cal{L}$ induces the polarization $2 \lambda_{\ovl{s}}$ on each geometric point $A_{\ovl{s}}$. 
\par 
For an abelian scheme $A/S$, let $A[n]=\ker[n]$ denote the $n$-torsion subgroup. A principal polarization $\lambda$ on $A/S$ induces a non-degenerate skew-symmetric pairing $A[n]\times A[n]\to\mu_n$ \cite[I.1.7]{FC90}. 
A level $n$ structure for an abelian scheme $A/S$ of relative dimension $g$ is an isomorphism of group schemes between $A[n]$ and $(\bb{Z}/n\bb{Z})^{2g}$, where $A[n]=\ker[n]$. For a level $n$ structure to exist, it is necessary for $n$ to be invertible in $O_S$, so that $S$ is a $\bb{Z}[1/n]$-scheme. 
\par 
For $n \geq 1$, we let $\cal{A}_{g,n}$ denote the fine moduli space of principally polarized abelian varieties (ppavs) with level $n$ structure over $\bb{C}$.  We will let $\cal{A}_g$ denote the moduli space of ppavs over $\bb{C}$. For $n \geq 3$, this is a quasi-projective scheme over $\bb{C}$ by Mumford \cite[Theorem 7.10]{MFK94}. We now fix $n=3$. There is a universal abelian scheme $\cal{A}$ over $\cal{A}_{g,3}$ with a principal polarization $\lambda:\cal{A}\to\cal{A}^t$. Then there is a universal invertible sheaf $\cal{L}$ on $\cal{A}$ such that if $(A,\lambda)$ is a principally polarized abelian variety over $K$ with level $3$ structure, then the pullback of $\cal{L}$ to $A$ induces $2 \lambda$. 
\par 
We are interested in two different compactifications of $\cal{A}_{g,3}$. The first is known as the minimal (Satake or Bailly--Borel) compactification, and we will denote it by $\cal{A}^*_{g,3}$. This is a projective variety over $\bb{C}$ which contains $\cal{A}_{g,3}$ as an open dense subvariety. Furthermore, there is an ample line bundle $\omega$ on $\cal{A}_{g,3}^*$ known as the Hodge bundle. On $\cal{A}_{g,3}$, it is equivalent to $\bigwedge^g e^*\Omega^1_{\cal{A}}$, where $\Omega^1_{\cal{A}}$ is the relative cotangent bundle of the universal abelian scheme $\cal{A}$ over $\cal{A}_{g,3}$ and $e: \cal{A}_{g,3} \to \cal{A}$ is the identity section. The line bundle $\omega$ can be used to compute the Faltings height \cite{Fal83} of a semistable principally polarized abelian variety $A/K$ where $K=\bb{C}(B)$ is the function field of a smooth projective curve $B$ over $\bb{C}$ as follows: passing to a finite extension $K'=\bb{C}(B')$, we may assume that $A/K'$ has level $3$ structure, and so we obtain a $K'$-point on $\cal{A}_{g,3}$. Since $\mathcal{A}_{g,3}^*$ is projective (hence proper), this $K'$-point of $\mathcal{A}_{g,3}$ defined by $A/K'$ extends from the generic point $\operatorname{Spec}K'$ of $B'$ to all of $B'$. In other words, by the valuative criterion for properness, any rational map $B'\dashrightarrow\mathcal{A}_{g,3}^*$ coming from the moduli point of $A/K'$ extends uniquely to a regular morphism $p_A:B'\to\mathcal{A}_{g,3}^*$. We may then define the stable Faltings height $h_{\Fal}(A)$ as
$$h_{\Fal}(A) = \frac{1}{[K':K]} \deg_{B'}(p_A^*\omega).$$
By \cite[Section 2]{Fal83}, this agrees with the usual definition via N\'{e}ron models.
\par 
The second compactification of $\cal{A}_{g,3}$ that we are interested in is the arithmetic toroidal compactifications which were constructed by Faltings--Chai \cite{FC90}. Over $\bb{C}$, these compactifications were known before Faltings--Chai \cite{AMRT75} but it was known that there was a semiabelian scheme over the compactification. By \cite[V.5.8]{FC90}, there exists a compactification $\ovl{\cal{A}}_{g,3}$ of $\cal{A}_{g,3}$ which is a projective variety over $\bb{C}$. More importantly for us, there exists a universal semiabelian scheme $G \to \ovl{\cal{A}}_{g,3}$ extending the universal abelian scheme $\cal{A} \to \cal{A}_{g,3}$. Note that Faltings and Chai only compactify the moduli space of ppavs with symplectic level $n$ structure. However, $\cal{A}_{g,3}$ consists of a union of disjoint copies of the moduli space with symplectic level $n$ structure \cite[IV.6.2(b)]{FC90} and hence we can simply compactify each copy individually.  

\par 
We now turn to the setting of abelian varieties $A$ over a function field $K = \bb{C}(B)$ where $B$ is a smooth projective curve over $\bb{C}$. Given a symmetric ample line bundle $L$, we may define the N\'{e}ron--Tate height $\h_{A,L}$ as 
$$\h_{A,L}(x) = \lim_{n \to \infty} \frac{1}{n^2} h_L([n]x) \text{ for } x \in A(\ovl{K})$$
where $h_L$ is a Weil height for the line bundle $L$. In general, we will omit the dependence on $L$ and simply write $\h_A$. By \cite[Lemma 2.3]{GGK21}, two such line bundles $L,L'$ induce the same N\'{e}ron--Tate height function if $c_1(L) = c_1(L')$ where $c_1$ is the first Chern class. Furthermore by \cite[Lemma 2.1]{GGK21}, two ample line bundles $L,L'$ have the same first Chern class if and only if they induce the same polarization on $A$. 

On the other hand, there is a universal polarization $\lambda_{\mathrm{univ}}$ on $\mathcal{A}_{g,3}$, characterized by the property that pulling it back along any moduli point $(A,\lambda)$ gives the polarization $\lambda$ on $A$. The universal symmetric ample line bundle $\mathcal{L}$ on the universal abelian scheme is normalized so that its first Chern class represents $2\lambda_{\mathrm{univ}}$, and therefore pulling back $\mathcal{L}$ along $(A,\lambda)$ produces a symmetric ample line bundle on $A$ whose first Chern class is $2\lambda$. It follows that every $K$-point $(A,\lambda)$ of $\mathcal{A}_{g,3}$ determines a unique Néron--Tate height on $A(\overline{K})$ via the pullback of the universal symmetric ample line bundle $\mathcal{L}$.

\subsection{Moduli Space of Curves} \label{sec: ModuliCurves} We now recall some properties of the moduli space of curves following \cite[Section 3.1]{Yua24}. Let $M_g$ be the coarse moduli space of smooth, projective, connected curves of genus $g$ over $\mathbb{C}$, and let $\overline{M}_g$ be the coarse moduli space of stable curves of genus $g$. Then $\ovl{M}_g$ is a proper variety over $\bb{C}$ and carries with it a $\bb{Q}$-line bundle $\lambda$, which we will call the Hodge bundle. The line bundle $\lambda$ has the following characterization \cite[Section 4]{CH88}: given any family $f: X \to S$ of stable curves, we obtain a morphism $p: S \to \ovl{M}_g$ where $p$ maps $s$ to the point on $\ovl{M}_g$ corresponding to the stable curve $f^{-1}(s)$. Then $p^* \lambda$ is the line bundle given by $\det(f_* \omega_{X/S})$, where $\omega_{X/S}$ is the relative dualizing sheaf. 
\par 
Now let $K = \bb{C}(B)$ be the function field of a smooth projective curve $B$ over $\bb{C}$ and let $C/K$ be a smooth curve of genus $g \geq 2$. Let $K' = \bb{C}(B')$ be an extension for which $C/K'$ has semistable reduction. Then there exists a stable model $\cal{C}/B'$ and hence we obtain a morphism $p: B' \to \ovl{M}_g$. We can define the stable Faltings height of our curve as $$h_{\Fal}(C)=\frac{1}{[K':K]}\deg_{B'}(p^* \lambda).$$ This is also equal to $h_{\lambda}(C)$, the height on $C$ arising via pullback of the Hodge bundle $\lambda$ to the point of $\ovl{M}_g(K)$ corresponding to $C/K$. Let $\cal{C}'/B'$ be the minimal regular model of $C/K'$. Let $\pi:\mathcal{C}\to B'$ and $\pi':\mathcal{C}'\to B'$ denote the respective structure morphisms of $\cal{C}$ and $\cal{C}'$. 
We have an equality \begin{equation}\label{eqn:detsequal}\deg(\pi'_* \omega_{\mathcal{C}'/B'})=\deg(\pi_* \omega_{\mathcal{C}/B'}).\end{equation} Noether's formula \cite[Section 6]{Fal84} gives \begin{equation}\label{eqn:Noether}\deg(\pi'_* \omega_{\cal{C}'/B'})=\frac{1}{12}( \omega^2_{\cal{C'}/B'}+N)\end{equation} where $N$ is the total number of nodes on $\cal{C}'$. 
\par
By construction, $p^*\lambda \cong \det(\pi_*\omega_{\mathcal{C}/B'})$, and since 
$\det(\pi_*\omega_{\mathcal{C}/B'})$ is a line bundle on $B'$ whose degree equals the degree 
of the vector bundle $\pi_*\omega_{\mathcal{C}/B'}$ on $B'$, we obtain 
$\deg_{B'}(p^*\lambda)=\deg_{B'}(\pi_*\omega_{\mathcal{C}/B'})$.
Moreover, as the generic fiber of $\cal{C}'/B'$ is of genus $g \geq 2$, we have the inequality $\omega^2_{\cal{C'}/B'}\geq 0$ \cite[Proposition 1]{Szp79}. Combining these facts with \eqref{eqn:detsequal} and \eqref{eqn:Noether}, we conclude that \begin{equation} \label{eq: CurveHeightNode}[K':K]h_{\lambda}(C)=[K':K]h_{\mathrm{Fal}}(C)=\deg(\pi'_* \omega_{\cal{C}'/B'})\geq\frac{N}{12}.\end{equation}

\section{Lower Bound on Arakelov-Green's functions}\label{section:AG}
In this section, we introduce a higher-dimensional version of the Arakelov-Green's functions due to the first author \cite{Loo24} along with some of the main results in \cite{Loo24}, which appear here as Theorem \ref{thm:basisbound} and Proposition \ref{TransfiniteDiameterBound1}. Theorem \ref{thm:basisbound} gives us an explicit lower bound on the values of the Arakelov-Green's functions while Proposition \ref{TransfiniteDiameterBound1} gives us a criterion to obtain a positive lower bound for the Arakelov-Green's function on an abelian variety with bad reduction. 
\par 
To be specific, for an abelian variety $A/K$ with bad reduction where $K$ is a non-archimedean field, the theory of Gubler \cite{Gub10} provides a retraction map $r:A^{\an}\to\bb{R}^{m}/\Lambda$ where $\bb{R}^{m}/\Lambda$ is a real torus and $m \geq 1$. If we have a collection of points that retract to $0$, then we can obtain a positive lower bound on our Arakelov-Green's functions. We then further study when can we say that a point $x \in A(K)$ lies in $r^{-1}(0)$.  

\subsection{A General Lower Bound} \label{sec: LowerBound} Let $(k,|\cdot|)$ be a valued field, with $|\cdot|$ extended to the algebraic closure $\ovl{k}$ of $k$. Let $X\subseteq\mathbb{P}^N$ be a projective variety over $k$, and let $f:X\to X$ be a degree $d\ge2$ endomorphism of $X$. Assume that $f$ extends to a morphism $f: \bb{P}^N \to \bb{P}^N$, which by abuse of notation we also denote by $f$. Let $F: \bb{A}^{N+1} \to \bb{A}^{N+1}$ be a homogeneous lift of $f$. We write $F = [F_0: \cdots ; F_n]$. If $P\in\mathbb{A}^{N+1}(k)$, we write \[||F(P)||=\max\{|F_0(P)|,\dots,|F_N(P)|\}.\] We define the \emph{homogeneous filled Julia set} of $F$ to be the set \[\mathcal{K}=\{P\in\mathbb{A}^{N+1}(\ovl{k}): ||F^n(P)||\not\to\infty \text{ as }n\to\infty\}.\] Throughout, we will let \[\pi:\mathbb{A}^{N+1}\setminus\{(0,\dots,0)\}\to\mathbb{P}^N\] denote the natural projection. 
	
We now build Arakelov-Green's functions associated to this dynamical system. For each $n\ge 1$, let $c(n)=\dim H^0(X,\mathcal{O}(n))$, where $\mathcal{O}(1)$ is the pullback bundle on $X$ induced by the embedding $X\xhookrightarrow{}\bb{P}^N$. We will assume that $X$ is normally embedded in $\bb{P}^N$, i.e., that $H^0(\bb{P}^N,\mathcal{O}(n)) \to H^0(X,\mathcal{O}(n))$ is surjective for all $n \geq 1$. Given a choice of bases $\{\mathcal{B}_n\}_{n=1}^\infty$ of $H^0(X,\mathcal{O}(n))$ respectively, there is a naturally induced such function \[g_n:X^{c(n)}\to\mathbb{R}\cup\{\infty\}\] for each $n$, constructed as follows. For a homogeneous lift $F:\mathbb{A}^{N+1}\to\mathbb{A}^{N+1}$ of $f$, let $\widehat{H}_F:\mathbb{A}^{N+1}\setminus\{0,\dots,0\}$ be given by \[\widehat{H}_F(P)=\lim_{n\to\infty}\frac{1}{d^n}\log||F^n(P)||.\] (That this limit exists is shown in \cite[proof of Theorem 9]{KS07}.) Write \begin{equation}\label{eqn:rF} r(F)=\frac{1}{d^N (N+1)(d-1)} \log |\Res(F)|^{-1},\end{equation} where $\Res(F)$ denotes the Macaulay resultant of $(F_0,\ldots,F_N)$. For $P\in X(k)$, write $\widetilde{P}$ for a lift of $P$ to $\mathbb{A}^{N+1}(k)$, with coordinates in $k$. The Arakelov-Green's function $g_n:X^{c(n)}\to\mathbb{R}\cup\{\infty\}$ associated to the basis $\mathcal{B}_n$ of $A_n$ is given by \begin{equation}\label{eqn:AGfn}g_n\left(P_1,\dots,P_{c(n)}\right)=\frac{1}{c(n)}\sum_{i=1}^{c(n)}\widehat{H}_F(\widetilde{P_i})-\frac{1}{n\cdot c(n)}\log\left|\mathrm{det}\left(\eta_j(\widetilde{P_i})\right)_{\eta_j\in\mathcal{B}_n}\right|+r(F).\end{equation} It is readily verified that for a fixed $\mathcal{B}_n$, this function is independent of the choice of lifts $\widetilde{P_i}$ and $F$. Moreover, if $\mathcal{B}_n'$ is another basis for $A_n$, then its associated Arakelov-Green's function differs by a constant from that associated to $\mathcal{B}_n$. In the special case of $X=\mathbb{P}^1$ and $\mathcal{B}_n$ the standard monomial basis of $H^0(\mathbb{P}^1,\mathcal{O}(n))=A_n$, this $g_n$ coincides (up to replacing $r(F)$ by $\mathrm{Res}(F)$) with the function $\mathscr{E}(g):(\mathbb{P}^1)^{n+1}\to\mathbb{R}\cup\{\infty\}$ given by \[\mathscr{E}(g)(P_1,\dots,P_{n+1})=\frac{1}{n(n+1)}\sum_{i\ne j}g(P_i,P_j)\] for $g$ the dynamical Arakelov-Green's functions introduced by Baker and Rumely (see \cite[\S3.4]{BR06} and \cite[\S10.2]{BR10}). 

\begin{definition}\label{def:dBn} Given a basis $\mathcal{B}_n$ of $H^0(X,\mathcal{O}(n))$ of order $c(n)$, let \begin{equation*}d_{\mathcal{B}_n}(\mathcal{K}\cap\pi^{-1}(X))=\exp\left(\sup_{\substack{\left(\widetilde{P_1},\dots,\widetilde{P}_{c(n)}\right)\\\in(\mathcal{K}\cap\pi^{-1}(X(\ovl{k}))^{c(n)}}}\frac{1}{n\cdot c(n)}\log\left|\mathrm{det}\left(\eta_j(\widetilde{P_i})\right)_{\eta_j\in\mathcal{B}_n}\right|\right).\end{equation*}  \end{definition}

From here on out, we normalize $F$ so that $r(F)=0$ (by passing to an extension of $k$ if necessary), and consider the collection $\mathcal{G}$ of polynomials of the form $\left(F_i^{(\ell)}\right)^j$, where $0\le i\le N$, $\ell\in\mathbb{Z}_+$, and $1\le j\le d-1$. Let \begin{equation}\label{eqn:basisform}\mathcal{F}=\bigcup_{j=\lfloor t_1\rfloor}^{\lfloor t_2\rfloor}\left\{\eta_iG_i: G_i\in\mathcal{G}^j\textup{ and }\eta_i\textup{ is a standard monomial of deg.}<d(N+1)\right\}\end{equation} where \[\mathcal{G}^j=\left\{G_1G_2\cdots G_j:G_\ell\in\mathcal{G}\textup{ for all }1\le \ell\le j\right\}.\] 

\begin{definition}\label{def:Hn}  We will denote by $\mathcal{B}_n$ any basis of $H^0(X,O(n))$ consisting of vectors of the form (\ref{eqn:basisform}) whenever $n\ge d(N+1)$, and consisting of the standard monomial basis otherwise. We will denote by $c(n)$ the quantity $\dim_k H^0(X,\mathcal{O}(n))$, even though this quantity clearly depends on $X$ and its embedding into $\mathbb{P}^N$ in addition to $n$. When there is potential ambiguity as to the underlying dynamical system or the lift $F$ used in defining $\mathcal{B}_n$, we will refer to $\mathcal{B}_n$ and $c(n)$ as being \emph{with respect to $(X,f)$} (as well as the lift $F$). The notation \[\left(\eta_j(P_1,\dots,P_{c(n)})\right)_{\eta_j\in\mathcal{B}_n}\] will denote the $c(n)\times c(n)$ matrix whose $i$-th row corresponds to $P_i$ and whose $j$-th column corresponds to $\eta_j\in\mathcal{B}_n$.\end{definition}

We will call the above bases $\cal{B}_n$ good. Note that this notion depends on the polarized dynamical system $(X,f)$ along with its lift $F$. Also note that for a good basis to be known to exist by definition, we must have a surjection $H^0(\bb{P}^N,\mathcal{O}(n))\to H^0(X,\mathcal{O}(n))$. For any fixed embedding $X \xhookrightarrow{} \bb{P}^N$, this is true for all sufficiently large $n$. Conversely if we have a surjection $H^0(\bb{P}^N,\mathcal{O}(n))\to H^0(X,\mathcal{O}(n))$, then a good basis always exists since $\cal{F}$ is a spanning set for $H^0(\bb{P}^N,\mathcal{O}(n))$ \cite[Section 2.2]{Loo24}.

We have the following theorem \cite[Theorem 3.1]{Loo24}.

\begin{theorem}\label{thm:basisbound} Let $k$ be a valued field with absolute value $|\cdot|$, and let $X\subseteq\mathbb{P}^N$ be a projective variety  Let $f:X\to X$ be a morphism of degree $d\ge2$ which extends to $f: \bb{P}^N \to \bb{P}^N$ and let $F:\mathbb{A}_k^{N+1}\to\mathbb{A}_k^{N+1}$ be a homogeneous lift of $f$. Let $\mathcal{K}$ be the homogeneous filled Julia set of $F$. For this triple $(X,f,F)$, and for $n \geq 1$, let $c(n)$ and $\mathcal{B}_n$ be as in Definition \ref{def:Hn}. Let \[\pi:\mathbb{A}^{N+1}\setminus\{(0,\dots,0)\}\to\mathbb{P}^N\] be the natural projection. Let $R$ be the diameter of $\mathcal{K}\cap\pi^{-1}(X)$ with respect to the distance defined by the supnorm. There is a constant $C=C(d,N)$ with the following property. For any $c(n)$-tuple $\vec{P}=\left(P_1,\dots,P_{c(n)}\right)$ of points in $\mathcal{K}\cap\pi^{-1}(X)$ and any $n\ge 2$, we have \begin{equation}\label{eqn:avglb}\log\left|\textup{det}\left(\eta_j(\vec{P})\right)_{\eta_j\in\mathcal{B}_n}\right|\le C\max\{\log R,1\}(\log n)\cdot c(n).\end{equation} (Here we define $\log(0)=-\infty$.)\end{theorem}

Note that Theorem \ref{thm:basisbound} gives a lower bound on $g_n(P_1,\ldots,P_{c(n)})$. Indeed as we may choose any lifts $\tilde{P}_i$ of $P_i$, we choose lifts (possibly after passing to an extension of $k$) such that the $\widehat{H}_F(\tilde{P}_i)$ are arbitrarily close to $0$. Thus we have the lower bound 
\begin{equation}\label{eqn:Greengenbd}g_n(P_1,\ldots,P_{c(n)}) \geq -C \max\{\log R,1\} \frac{\log n }{n}+r(F).\end{equation}
In particular, we have 
$$\liminf_{n \to \infty} \inf g_n(P_1,\ldots,P_{c(n)}) \geq r(F).$$

Given an endomorphism $f: \bb{P}^N \to \bb{P}^N$ along with a lift $F: \bb{A}^{N+1} \to \bb{A}^{N+1}$, we would like to bound the supnorm diameter $R$ of the homogeneous filled Julia set $\mathcal{K}$ of $F$ explicitly. This can be done by using \cite[Lemma 6]{Ing22}. For simplicity, let's assume that $| \cdot |$ is non-archimedean as that will be the only case we use. Letting \[\lambda(f)=-\log|\Res(F)|+(N+1)(d^2)^N\log ||F||\] be a local height function for the hypersurface $\{\Res(F)=0\}$ inside $\Rat_{d^2}^N$, we have
$$\frac{-\lambda(f)}{d-1}\leq\Ht_F(P)-\log||P||-\frac{1}{d-1}\log||F||\leq 0$$
for all $P \in \bb{A}^{N+1}$. Inside $\mathcal{K}$, we have $\Ht_F(P)\leq 0$ and so 
$$\log||P||\leq\frac{\lambda(f)}{d-1}+\frac{1}{d-1}\log||F||.$$ Note that $\lambda(f)\geq0$ since $|\cdot|$ is non-archimedean. Hence we may take 
\begin{equation}\label{eqn:Rbound}R=O(\lambda(f)+\log||F||)\end{equation}
where our constant only depends on $d$ and $N$.
\par 
Now suppose $K$ is a product formula field and assume further that $F$ has a coefficient that is $1$, so that $\log||F||_v\geq0$ for all $v \in M_K$. This allows us to assume that for any $M_K'\subseteq M_K$,
\begin{equation}\label{eqn:Fnormbd}\sum_{v \in M_K'} N_v \log ||F||_v \leq \sum_{v \in M_K} N_v \log ||F||_v = h(f).\end{equation} Thus, by the non-negativity of $\lambda_v$ at all $v\in M_K$,
$$\sum_{v \in M_K'} (\lambda_v(f) + \log ||F|_v) \leq \sum_{v \in M_K} (\lambda_v(f) + \log ||F||_v) = O(h(f)).$$ In particular, for any $M_K' \subseteq M_K$, combining this with \eqref{eqn:Greengenbd} and \eqref{eqn:Rbound} gives
\begin{equation} \label{eq:ElkiesLower1}
\sum_{v \in M_K'}g_{\cal{B}_k}(x_1,\ldots,x_{c(k)})\geq\left(\sum_{v\in M_K'}r(F)_v\right)-\frac{C\log k}{k}O(h(f)).
\end{equation}

\subsection{A Positive Lower Bound for Abelian Varieties} \label{sec : AbVarPos}
Let $K$ be a complete nonarchimedean valued field, and let $A$ be an abelian variety over $K$ embedded into $\bb{P}^N$ along with an endomorphism $f: \bb{P}^N \to \bb{P}^N$ of degree $d^2$ that extends multiplication by $[d]$ on $A$. We assume without loss that $A$ is normally embedded into $\bb{P}^N$, so that $H^0(\bb{P}^N,\mathcal{O}(n))$ surjects onto $H^0(A,\mathcal{O}(n))$. Pick a homogeneous lift $F: \bb{A}^{N+1} \to \bb{A}^{N+1}$. Let $\{\mathcal{B}_n\}$ be a choice of basis for $H^0(A,\mathcal{O}(n))$ of the form given in Definition \ref{def:Hn}. For each $n$, let $g_n(x_1,\dots,x_{c(n)})$ be the associated Arakelov-Green's function. By Theorem \ref{thm:basisbound}, if $\mathcal{K}$ is the homogeneous filled Julia set of $F$, $\pi:\mathbb{A}^{N+1}\setminus\{0\}\to\mathbb{P}^N$ is the natural projection, and $\cal{B}_n$ is a sequence of good bases for $H^0(A,\mathcal{O}(n))$,  we obtain 
\begin{equation} \label{eq:TransfiniteBound1}
\limsup_{n} d_{\cal{B}_n}(\cal{K} \cap \pi^{-1}(A)) \leq \exp(-r(F)).
\end{equation}
Given a $c(n)$-tuple $\vec{P} \in A^{\an}$, we say that $\vec{P}$ is \emph{Fekete} if 
$$g_{\cal{B}_n}(\vec{P}) = \log d_{B_n}(\cal{K} \cap \pi^{-1}(A)).$$
Given a sequence $(\vec{P}_n)$ of $c(n)$-tuples of points $P_i \in A^{\an}$, along with a sequence of bases $(\cal{B}_n)$, not necessarily a sequence of good bases, we say that the sequence is \emph{asymptotically Fekete} if  
$$\lim_{n \to \infty} \left(g_{\cal{B}_n}(\vec{P}_n) + \log d_{\cal{B}_n}(\cal{K} \cap \pi^{-1}(A)) \right) = 0.$$
Note that the notions of Fekete and asymptotically Fekete are independent of our choice of sequence of bases $\{\cal{B}_n\}$.  
\par 

By \cite[Theorem 3.5]{Loo24}, using the results of \cite{BGM22}, it follows that there exists a unique probability measure $\mu_A$, supported on $A^{\an}$, such that for any sequence $(\vec{P}_n)$ that is asymptotically Fekete, we have 
\begin{equation} \label{eq: Convergence1}
\frac{1}{c(n)} \sum_{x \in \vec{P}_n} \delta_x \overset{\ast}{\rightharpoonup} \mu_A.
\end{equation}
In fact, we can describe $\mu_A$. On $A^{\an}$, as our line bundle $L = \mathcal{O}(1)$ is symmetric and ample, we have an isomorphism $\vphi: L^{n^2} \simeq [n]^*L$ for any $n \geq 1$. Fix some $n \geq 2$. There is then a canonical metric $| \cdot |$ by \cite[Theorem 2.2]{Zha95} or \cite[Theorem 9.5.3]{BG06}, unique up to a multiplicative constant, such that $(\vphi^* \circ [n]^*| \cdot |)^{1/n^2} = | \cdot |$ for our $n$. Our metric is a limit of model metrics and hence semipositive, or psh in the sense of \cite{BGM22}. By \cite[Corollary D]{BGM22}, the measure $\mu_A$ coincides with the canonical measure $c_1(L, | \cdot |)^{\wedge g}$ where $g = \dim A$, see \cite{Gub07} for more details on the definition of $c_1(L, |\cdot|)^{\wedge g}$. Note that Chambert--Loir's original definition \cite{CL06} of the Monge--Amp\`ere measures assumed that the valuation on $K$ comes from the algebraic closure of a field with discrete valuation but this assumption was subsequently removed by Gubler. There is a continuous retraction map \cite{Gub10, GS23}
$$r: A \to S(A) \simeq \bb{R}^k/\Lambda$$
where $S(A)$ is the skeleton of $A$, with dimension $0\le k\le g$, and $\Lambda$ is a lattice in $\bb{R}^k$. 
\par
We now assume that $A$ has bad reduction over $K$, so that $k \geq 1$. Let $S = \supp \mu_A$. Then \cite[Theorem 1.2]{GS23} tells us that $r(S) = \bb{R}^k/\Lambda$, i.e., the measure is supported on the entire skeleton. Suppose we are given a sequence of $c(n)$-tuples $\vec{P}_n$ such that $r(\vec{P}_n) = 0$. Define \[\mathfrak{d}_{\mathcal{B}_n}(\vec{P}_n)=\exp(-g_{\mathcal{B}_n}(\vec{P}_n)-\log d_{\mathcal{B}_n}(\mathcal{K}\cap\pi^{-1}(A))).\] We claim that $\log\mathfrak{d}_{\cal{B}_n}(\vec{P}_n)$ is bounded away from $0$.

\begin{proposition} \label{TransfiniteDiameterBound1}
Let $\cal{B}_n$ be a sequence of good bases and $(\vec{P}_n)$ a sequence of $c(n)$-tuples of points such that $r(\vec{P}_n) = 0$. Then there exists a $\delta > 0$ such that 
$$\limsup \log\mathfrak{d}_{\cal{B}_n}(\vec{P}_n) \leq - \delta.$$
In other words, there exists a $\delta > 0$ such that for all sufficiently large $n$,
$$\log\mathfrak{d}_{\cal{B}_n}(\vec{P}_n) \leq - \delta.$$
\end{proposition}

\begin{proof}
Assume otherwise. Then there exists a sequence $n_1 < n_2 < \cdots$ such that $$\lim_{n\to\infty} \log\mathfrak{d}_{\cal{B}_{n_i}}(\vec{P}_{n_i})=0.$$
We now replace our sequence of tuples so that for $n = n_i$, we keep the original tuple but for each $n \not = n_i$, we replace $\vec{P}_n$ with a Fekete tuple. Let $(\vec{Q}_n)$ be this new sequence. Then since we have the upper bound 
$$\limsup_{i\to\infty} \log\mathfrak{d}_{\cal{B}_{n_i}}(\vec{P}_{n_i}) \leq 0,$$
it follows that $(\vec{Q}_n)$ is an asymptotically Fekete sequence of tuples. Applying \eqref{eq: Convergence1}, we obtain
$$\frac{1}{c(n)} \sum_{x \in \vec{Q}_n} \delta_x \overset{\ast}{\rightharpoonup} \mu_A$$
and hence passing to the subsequence $\{n_i\}$ yields
$$ \frac{1}{c(n_k)} \sum_{x \in \vec{P}_{n_k}} \delta_x \overset{\ast}{\rightharpoonup} \mu_A$$
too. However, for all $k$ and for all $x \in \vec{P}_{n_k}$, we have $r(x) = 0$. As $r$ is continuous, $r^{-1}(0)$ is a closed set. It follows that we must have $\supp \mu_A \subseteq r^{-1}(0)$. But this is impossible as $r(\supp \mu_A)$ surjects onto $\bb{R}^k/\Lambda$, which is positive-dimensional as $k \geq 1$ by hypothesis. 
\end{proof} In the setting of Section \ref{sec: LowerBound}, if we pick our lift $F$ so that $r(F) = 0$, by \eqref{eq:TransfiniteBound1} we have
$$\limsup_{n} \log d_{\cal{B}_n}(\cal{K} \cap \pi^{-1}(A)) \leq 0.$$
In particular, Proposition \ref{TransfiniteDiameterBound1} implies that there is some $\delta > 0$ for which 
$$g_{\cal{B}_n}(\vec{P}) \geq \delta > 0.$$
\par
We want to understand the kernel $r^{-1}(0)$. We recall the construction of the retraction map by Gubler \cite[Section 4]{Gub10}. Let $R$ be the valuation ring of $K$. Following \cite[Section 4]{Gub10}, by \cite[Theorem 8.2]{BL84}, there is a unique connected formal group scheme $\cal{A}_1$ over $R$ with generic fiber $A_1$ (in the sense of Raynaud) such that $\cal{A}_1$ has semiabelian reduction and $A_1$ is isomorphic to an open analytic subgroup of $A^{\an}$.  
\par 
The semistable reduction theorem says that there is a short exact sequence 
$$1 \to T_1 \xrightarrow{\iota_1} A_1 \xrightarrow{q_1} B \to 1$$
where $T_1$ is the maximal formal affinoid subtorus of $A_1$ and $B$ is the generic fiber of a formal abelian scheme $\cal{B}$. We may identify $T_1$ with the subset $\{|x_1| = \cdots = |x_n| = 1\}$ of the usual torus $T = (\bb{G}^n_m)^{\an}$. The uniformization $E$ of $A$ is given by an analytic group $E = (A_1 \times T)/T_1$, where $T_1$ acts by 
$$t_1 \cdot (a,t) = (\iota_1(t_1) + a, t_1^{-1} \cdot t).$$
This gives us an exact sequence 
$$1 \to T \to E \to B \to 1$$
The closed immersion $T_1 \to A_1$ extends to $T \xrightarrow{\iota}A^{\an}$ and hence, from the map $A_1\times T\to A^{\an}$ given by $(a_1,t)\mapsto a+\iota(t)$ (whose kernel contains $T_1$), we obtain a homomorphism $\pi: E \to A^{\an}$. This turns out to be a surjection and the kernel $\ker(\pi) = M$ is a discrete subgroup of $E$. 
\par 
Let $q_1: A_1 \to B$ be the quotient map from $A_1$ to $B$. We may find a chart of formal affinoids $\{V_i\}$ with respect to $\cal{B}$ covering $B$ such that $q_1^{-1}(V_i) \simeq V_i \times T_1$. We may choose sections $s_i:V_i\to A_1$ such that on each chart $V_i\times T_1$ of $A_1$, we have $s_i:v\mapsto (v,1)\in V_i\times T_1$. Then the transition maps $g_{i,j} = s_i - s_j$ are maps from $V_i \cap V_j$ to $T_1$. Note also that the $V_i \times T$ form an atlas for $E$ since $E = (A_1 \times T)/T_1$. Fix coordinates $x_1,\ldots,x_n$ on $T = (\bb{G}^n_m)^{\an}$. Then we obtain coordinates on $V_i \times T$ just by looking at the $T$-component.  These depend on our choice of chart, but the function $|x_i|$ does not as the transition maps $g_{i,j}$ are valued in $T_1$. We hence obtain a well-defined continuous map $$\begin{array}{llll} \val: & E & \to & \bb{R}^n \\ & p & \mapsto & (-\log |x_1|(p),\ldots, -\log |x_n|(p)). \end{array} $$ The discrete subgroup $M$ maps to a lattice $\Lambda$ and thus we get a map 
$$r: A^{\an} \to \bb{R}^n/\Lambda$$ that forms a retraction of $A^{\an}\cong E/M$.
This construction is independent of the choice of formal charts $\{V_i\}$ covering $B$.

\begin{proposition} \label{Retraction1} 
Under the retraction map $r$, we have $r(A_1) = 0$.
\end{proposition}

\begin{proof}
The retraction map $r: A^{\an} \to \bb{R}^n/\Lambda$ is induced by the valuation map $\val: E \to \bb{R}^n$. Recall that we have an atlas of formal affinoid charts $\{V_i\}$ covering $B$ satisfying $q^{-1}(V_i) \simeq V_i \times T_1$. Then as noted previously, $E$ has an atlas given by the $V_i \times T$ and the functions $|x_1|,\ldots,|x_n|$ are defined by looking at the $T$-coordinates under any chart of this atlas. But in this atlas, we have $A_1 \simeq V_i \times T_1$, and so clearly $\val(A_1) = 0$. Hence we must have $r(\pi(A_1)) = 0$. But by construction of $E$, $A_1\subseteq A^{\an}$ is contained in $\pi(A_1)$, and hence $r(A_1)=0$ as desired.\end{proof}

We now give a criterion for a point $x \in A(K)$ to lie in $A_1$. Recall that $A_1$ is the generic fiber of a unique connected formal group scheme $\cal{A}_1$ over $R$ such that $\cal{A}_1$ has semiabelian reduction and that $A_1$ is isomorphic to an open analytic subgroup of $A^{\an}$. 

\begin{proposition} \label{KernelReduction1}
Let $G$ be a semiabelian scheme over $R$ with generic fiber $A$. If $x \in A(K)$ lifts to $\tilde{x} \in G(R)$, then $x \in A_1$. In particular, $r(x) = 0 \in S(A)$.  
\end{proposition}

\begin{proof}
 Since $G$ is finitely presented over $R$, we may take its formal completion $\cal{G}$, which is a formal group scheme over $R$. By \cite[Theorem 5.3.1]{Con99}, its generic fiber $\cal{G}^{\an}$ is an open analytic subgroup of $A^{\an}$. As by construction $\cal{G}$ has the same special fiber as $G$, we see that $\cal{G}$ clearly has semi-abelian reduction. By uniqueness of $\cal{A}_1$ and connectedness of $\cal{G}$, we thus must have $A_1 = \cal{G}^{\an}$. Now since $x$ lifts to an $R$-point of $G$, it lifts to an $R$-point of $\cal{G}$ too and hence is a $K$-point of $\cal{G}^{\an}$ as $\cal{G}^{\an}(K) = \cal{G}(R)$. In particular, $x \in A_1$.\end{proof}

When $A$ is defined over a discretely valued field $K' \subseteq K$, the existence of such a semiabelian scheme $G$ follows from the existence of N\'{e}ron models. In general we may not have a N\'{e}ron model but we can use the Faltings--Chai compactification to obtain a semiabelian model. Recall that we have a compactification $\ovl{\cal{A}}_{g,3}$ of $\cal{A}_{g,3}$ such that the universal abelian scheme $\cal{A}$ over $\cal{A}_{g,3}$ extends to a semiabelian scheme $\cal{G}$. Then after choosing level $3$ structure, $A$ defines a $K$-point on $\cal{A}_{g,3}$ and hence by properness, we obtain an $R$-point of $\ovl{\cal{A}}_{g,3}$. This gives us a semiabelian scheme $G/R$ with generic fiber $A/K$. Note that \cite[I.2.7]{FC90} proves the uniqueness of such a semiabelian model $G$ if $R$ is noetherian and normal. We can also obtain uniqueness for general non-noetherian valuation rings by descending $G$ to a noetherian base, but this will not be necessary for us.  

\section{Degenerations of Sequences of Points}
We now introduce the theory of degenerations of sequences of objects using ultrafilters. This was introduced in complex dynamics for rational maps on $\bb{P}^1$ by Luo \cite{Luo21, Luo22} and then formalized using Berkovich spaces by Favre and Gong \cite{FG24}. As the approach of Favre--Gong is algebraic, it is easy to extend their method to objects with a fine moduli space and in particular to abelian varieties. This will be useful to us as quantities coming from Arakelov theory are usually continuous with respect to degenerations \cite{DEJ19, DKY20, Fal21, Poi24b, Wil21, Yua24, YZ24}. We first review the general construction of degenerations on a general quasi-projective variety $X$ over $\bb{C}$ and then specialize this to the case of $X = \cal{A}_{g,3}$ later. 

\subsection{Degeneration via Ultrafilters}
Let $X$ be a quasi-projective variety over $\bb{C}$. We will assume that all our varieties are irreducible unless stated otherwise. Let $(K_n)$ be a sequence of complete algebraically closed non-archimedean fields containing $\bb{C}$ such that $\bb{C}$ is trivially valued. Let $(\eps_n)$ be a sequence of positive reals. We define 
$$\scr{A}^{\eps} = \{ (x_n) \in \prod_{n=1}^{\infty} K_n \mid |x_n|^{\eps_n} \text{ is bounded}.\}.$$
This is a Banach ring under the supremum norm $|(x_n)| = \sup_n |x_n|^{\eps_n}$. We have natural projection maps $\pi_n: \scr{A}^{\eps} \to K_n$ which induce projections $\pi_n: X(\scr{A}^{\eps}) \to X(K_n)$. 

\begin{definition} \label{DegenerateDefinition1}
Let $x_n \in X(K_n)$ be a sequence of points. An $(\eps_n)$-degeneration of $(x_n)$ is a point $x \in X(\scr{A}^{\eps})$ such that $\pi_n(x) = x_n$ for all $n \in \bb{N}$.     
\end{definition}

If an $(\eps_n)$-degeneration exists, it is necessarily unique. 

\begin{proposition} \label{DegenerationUniqueness1}
There exists at most one $(\eps_n)$-degeneration of $(x_n)$. 
\end{proposition}

\begin{proof}
Let $Y=\bigsqcup_{n\ge1}\Spec K_n$. Each projection 
$\pi_n\colon \scr{A}^{\eps}\to K_n$ induces \[\Spec(\pi_n)\colon\Spec K_n\to\Spec \scr{A}^{\eps},\] and these glue to a single morphism \[f\colon Y\to\Spec \scr{A}^{\eps}.\] We claim that $f$ is an epimorphism in the category of separated, reduced schemes. Indeed, given $g,h: \Spec \scr{A}^{\eps} \to X$ such that $g \circ f = h \circ f$, we can consider the diagonal embedding $\Delta: X \to X \times X$, which is a closed immersion by separatedness of $X$. Pulling back $\Delta(X)$ via $(g,h): \Spec \scr{A}^{\eps} \to X \times X$ gives us a closed subscheme $Z$ of $\Spec\scr{A}^\eps$ and the map $f: Y \to \Spec \scr{A}^{\eps}$ must factor through $Z$ since $g \circ f = h \circ f$.
\par 
We claim that $f$ is dominant. Indeed, the image of $f$ consists of the kernels of the projection maps $\pi_n: \scr{A}^{\eps} \to K_n$. Given any nonzero $a \in \scr{A}^{\eps}$ with $a = (a_n)$, there must exist some $a_n \not = 0$ and so $\pi_n(a) \not = 0$, which implies that $\ker(\pi_n) \in D(a)$. Since distinguished opens form a basis, it follows that $f(Y)$ is dense and hence $f$ is dominant as desired. 
\par 
Since $f$ is dominant and factors through $Z$, $Z$ must be dense in $\Spec \scr{A}^{\eps}$. As $\Spec \scr{A}^{\eps}$ is reduced, it follows that $Z = \Spec \scr{A}^{\eps}$ and hence $g = h$ as desired. 
\par 
We now go back to our proposition. the points $x_n\in X(K_n)$ define maps $\Spec K_n\to X$ which glue to \[f'\colon Y\to X,\] and $f'$ is clearly unique. By definition an $(\varepsilon_n)$-degeneration is a morphism \[\tilde x\colon\Spec \scr{A}^{\eps}\to X\] with $\tilde x\circ f=f'$. Given two such morphisms $\tilde{x}_1, \tilde{x}_2$, we have $\tilde{x}_1 \circ f = \tilde{x}_2 \circ f$ as maps from $Y$ to $X$. Since $f$ is an epimorphism of reduced, separated schemes, pre-composition with $f$ is injective on morphisms.  It follows that $\tilde x_1=\tilde x_2$, proving uniqueness.
\end{proof}

Clearly an $(\eps_n)$-degeneration need not always exist. For example if $\eps_n = 1$ and $X = \bb{A}^1$, then $x_n$ corresponds to an element of $K_n$ and an $(\eps_n)$-degeneration exists if and only if $|x_n|$ is bounded from above. However if $X$ is a projective variety, then an $(\eps_n)$-degeneration always exists. 

\begin{proposition} \label{DegenerationExistence1}
Assume that $X$ is a projective variety over $\bb{C}$. Let $x_n\in X(K_n)$ for all $n$. Then an $(\eps_n)$-degeneration of $(x_n)$ always exists. 
\end{proposition}

\begin{proof}
Let $\{U_1,\ldots,U_n\}$ be an affine open cover for $X$ and let $Y = \bigsqcup_{i=1}^{n} U_i$. Then we have a natural morphism $\pi: Y \to X$ and $Y$ is affine. Let $x_n \in X(K_n)$. Then since $X$ is projective, $x_n$ lifts to $\tilde{x}_n \in X(R_n)$ where $R_n$ is the valuation ring of $K_n$. Hence we have an induced morphism $\tilde{x}_n: \Spec R_n \to X$. Since $\{U_i\}$ is an open cover and $R_n$ is a valuation ring, if $\tilde{x}_n^{-1}(U_i)$ contains the maximal ideal of $R_n$ then it contains all of $\Spec R_n$. Thus $\tilde{x}_n \in U_i(R_n) \subseteq Y(R_n)$. 
\par 
Since $Y$ is affine, we may let $x_1,\ldots,x_N$ be affine coordinates for $Y$ over $\bb{C}$. Then $\tilde{x}_n$ has affine coordinates $a_{n,1},\ldots,a_{n,N}$ with $a_{n,i} \in R_n$ and $|a_{n,i}| \leq 1$ for all $1\le i\le N$. Hence if we set $a_i = (a_{n,i})$, we have $a_i \in \scr{A}^{\eps}$. We thus obtain an $\scr{A}^{\eps}$-point $(a_i)_{i=1}^N$ of $Y$ that commutes with the projection map $\pi_n: \scr{A}^{\eps} \to K_n$. Projecting it down to $X$, we obtain our $(\eps_n)$-degeneration as desired.
\end{proof}

\begin{rmk}\label{rmk:Bep} Observe that in Proposition \ref{DegenerationExistence1}, the resulting $(\varepsilon_n)$-degeneration is actually a $\scr{B}^{\eps}$-point of $X$ for $$\scr{B}^{\eps} = \prod_{n=1}^{\infty} R_n,$$ where $R_n$ is the valuation ring of $K_n$.\end{rmk}
\par 
Now if $X$ is only quasi-projective, we may embed it into some projective compactification $\overline{X}$. Then the $(\eps_n)$-degeneration exists on $\overline{X}$ and we can ask whether our $\scr{A}^{\eps}$-point lies on $X$. Recall that for each non-archimedean field $K$ having $\bb{C}$ as a trivially valued subfield, we have a local distance function $\lambda_{\partial X}$ depending on the compactification $\overline{X}$. On an affine open $U_i \subseteq \overline{X}$ such that the reduced closed subscheme $(\overline{X} \setminus X ) \cap U_i$ is given by $\{f_1 = f_2 = \cdots = f_m = 0\}$ for $f_j \in\mathcal{O}(U_i)$, Proposition \ref{LocalHeightFunction2} implies that for $x\in U_i(K)$ whose affine coordinates in $U_i$ have norm $\leq 1$, we have 
\begin{equation}\label{eqn:distancemin}\lambda_{\partial X}(x) = \min_{1 \leq i \leq n} -\log |f_i(x)|.\end{equation} The function $\lambda_{\partial X}$ depends on the choice of compactification, but $\frac{\lambda_{\partial X}}{\lambda'_{\partial X}}$ is bounded from above and below for any two compactifications.

\begin{proposition}\label{DegenerationExistence2}
Let $X$ be a quasi-projective variety over $\bb{C}$. Let $x_n \in X(K_n)$ and let $(\eps_n)$ be a sequence of reals. Then an $(\eps_n)$-degeneration of $(x_n)$ exists on $X$ if and only if  
$\eps_n \lambda_{\partial X}(x_n)$
is uniformly bounded from above. 
\end{proposition}

\begin{proof}
Pick any projective compactification $\ovl{X}$ of $X$ and let $\{U_i\}_{i=1}^{k}$ be a finite affine open cover of $\ovl{X}$. Let $Y = \bigsqcup_{i=1}^{k} U_i$ be the disjoint union of the $U_i$ and let $\pi :Y \to \ovl{X}$ be the natural projection. Then each $x_n \in X(K_n)$ lifts to an $R_n$-point of $\ovl{X}$ and hence to an $R_n$-point $\tilde{x}_n$ of $Y$. We thus obtain a $\scr{B}^{\eps}$-point $\tilde{x}$ of $Y$ which is an $(\eps_n)$-degeneration of $(\tilde{x}_n)$, and so $\pi(\tilde{x})$ is an $(\eps_n)$-degeneration of $(x_n)$. Our proposition is then equivalent to verifying whether $\tilde{x}$ lies on $\pi^{-1}(X)$. 
\par 
Since $Y$ is affine, we may let $Y \simeq \bb{C}[X_1,\ldots,X_N]/I$ for some ideal $I$. Then each $\tilde{x}_n$ has affine coordinates $a_{n,1},\ldots,a_{n,N}$ with $|a_{n,i}| \leq 1$ as it is an $R_n$-point. Let $f_1,\ldots,f_m \in \cal{O}(Y)$ be regular functions that cut out $Y \setminus \pi^{-1}(X)$. Then we have 
$$\lambda_{\partial X}(x_n) = \lambda_{Y \setminus \pi^{-1}(X)}(\tilde{x}_n) = \min_{1 \leq i \leq m} -\log |f_i(\tilde{x}_n)|$$
where the first equality follows from $\pi^{-1}(\partial X) = \pi^{-1}(\ovl{X} \setminus X) = Y \setminus \pi^{-1}(X)$ and the second from \eqref{eqn:distancemin}. 
The affine coordinates of our $\scr{A}^{\eps}$-point $\tilde{x}$ is given by $(a_1,\ldots,a_N)$ where each $a_i = (a_{n,i})$. We wish to check whether our $\scr{A}^{\eps}$-point lies on $\pi^{-1}(X)$. This is equivalent to checking whether the pullback of the affine opens $D(f_i)$ by $\tilde{x}$ cover $\Spec \scr{A}^{\eps}$, which is equivalent to checking if  
$$J = (f_1(a_1,\ldots,a_N), \ldots, f_m(a_1,\ldots,a_N))$$
generate the unit ideal in $\scr{A}^{\eps}$ as $\scr{A}^{\eps}$ is reduced. 
\par 
Now, an element $b = (b_n)$ in $\scr{A}^{\eps}$ is in $(\scr{A}^{\eps})^{\times}$ if and only if $(|b_n|^{-1})^{\eps_n}$ is uniformly bounded from above \cite[Remark 3.7]{FG24}. Hence if $J$ generates the unit ideal, then there exists $c_i$ with $c_i = (c_{n,i})^{\infty}_{n=1}$ with 
\begin{equation} \label{eq: Inverse1}
\sum_{i=1}^{m} c_i f_i(a_1,\ldots,a_N) \in (\scr{A}^{\eps})^{\times}. 
\end{equation} 
As $|c_{i,n}|^{\eps_n}$ is bounded from above as $n$ varies, \eqref{eq: Inverse1} implies that
$$\max_{1 \leq i \leq m} |f_i(a_{n,1},\ldots,a_{n,N})|^{\eps_n}$$
is uniformly bounded away from $0$ and so $\eps_n \lambda_{\partial X}(x_n)$ is uniformly bounded from above. Conversely suppose that $\eps_n \lambda_{\partial X}(x_n) \leq C$ for some $C > 0$. For each $n$, let $i_n$ be an index such that 
$$-\log |f_{i_n}(a_{n,1},\ldots,a_{n,N})| = \lambda_{\partial X}(x_n).$$
We define $c_i \in \scr{A}^{\eps}$ as $c_i = (c_{1,i},\ldots,c_{n,j},\ldots)$ where $c_{n,j} = 1$ if $i_n = j$ and $0$ otherwise. Then 
$$\left|\sum_{i=1}^{m} c_{i} f_i(a_{n,1},\ldots,a_{n,N})\right|^{\eps_n} = e^{- \lambda_{\partial X}(x_n) \eps_n}$$
and hence the left-hand side is uniformly bounded away from $0$ as $\eps_n \lambda_{\partial X}(x_n) $ is uniformly bounded from above. Thus $\sum_{i=1}^{m} c_i f_i(a_1,\ldots,a_N) \in (\scr{A}^{\eps})^{\times}$ and so our $\scr{A}^{\eps}$-point lies in $X$ as desired.
\end{proof}

Hence if $\lambda_{\partial X}(x_n) > 0$ for all $n$, then $\eps_n = (\lambda_{\partial X}(x_n))^{-1}$ is the smallest rescaling parameter (up to commensurability) that we can degenerate so that our $\scr{A}^{\eps}$-point still lies on $X$. Following Favre--Gong's suggestion, we will call this the fundamental rescaling parameter. 

\begin{definition} We will call such a sequence $(\eps_n)$ a fundamental rescaling parameter for the sequence $(x_n)$. Any two fundamental rescaling parameters have their ratios uniformly bounded from above or below. \end{definition}
\par 
In principle, we are free to pick any sequence that is much smaller than $(\eps_n)$, but it will turn out that the fundamental rescaling parameter will be the more useful sequence, based on Proposition \ref{Degenerating1}.
\par 
Recall that the Berkovich spectrum of $\scr{A}^{\eps}$ is in bijection with the set of ultrafilters $\omega$ on $\bb{N}$ \cite[Proposition 1.2.3]{Ber90}, where the bijection \cite[Theorem 3.8]{FG24} sends
$$\omega \mapsto \left( (a_n) \mapsto \lim_{\omega} |a_n|^{\eps_n} \right).$$
As $|a_n|^{\eps_n}$ is bounded from above, this sequence lies in a compact interval and hence its $\omega$-limit exists as a real number. For non-principal ultrafilters, we will let $| \cdot |_{\omega}$ denote the norm associated to $\omega$. There is an associated maximal ideal $\ker(\omega) = \{ a \in \scr{A}^{\eps} \mid |a|_{\omega} = 0\}$ and the quotient $\scr{A}^{\eps}/\ker(\omega)$ gives a residue field $\cal{H}(\omega)$ which is complete with respect to $| \cdot |_{\omega}$ and algebraically closed.
\par 
Given an $(\eps_n)$-degeneration of $(x_n)$ with $(\eps_n)$ being a fundamental rescaling parameter, we obtain an $\scr{A}^{\eps}$-point of $X$ and hence an $\cal{H}(\omega)$ point $x_{\omega}$ of $X$. Let $\cal{R}(\omega)$ be the valuation ring of $\cal{H}(\omega)$. One might ask whether it is possible to lift $x_{\omega}$ to an $\cal{R}(\omega)$ point of $X$. This turns out to never be the case if each $x_n$ cannot be lifted to an $R_n$-point, i.e., if $\lambda_{\partial X}(x_n) > 0$ for all $n$. We first need the following basic proposition.

\begin{proposition} \label{BadReduction1}
Let $X$ be a variety over $\bb{C}$ and let $x \in X(K)$ where $K$ is a complete non-archimedean field with $\bb{C}$ a trivially valued subfield and with valuation ring $R$. Then $x$ lifts to an $R$-point of $X$ if and only if $\lambda_{\partial X}(x) = 0$.
\end{proposition}

\begin{proof}
Let $\ovl{X}$ be a projective compactification and let $U_i$ be an affine cover. Since $\ovl{X}$ is projective, $x$ lifts to an $R$-point $\tilde{x}$ on $\ovl{X}$. Let $U_i$ be an affine chart containing this $R$-point. Then on $U_i$, $x$ has affine coordinates $\leq 1$. If $f_1,\ldots,f_m \in \cal{O}(U_i)$ cut out $(\ovl{X} \setminus X) \cap U_i$, then $\tilde{x} \in X(R)$ if and only if $\tilde{x}$ has empty intersection with $\bigcap_{i=1}^{m} V(f_i)$. Since $R$ is a valuation ring, this is equivalent to saying at least one of the $f_i(x)$ is invertible, which is equivalent to 
$$\lambda_{\partial X}(x) = \min_{1 \leq i \leq m} \{ -\log |f_i(x)| \} = 0$$
as desired.\end{proof}

\begin{proposition} \label{Degenerating1}
Let $X$ be a variety over $\bb{C}$ and $(K_n)$ a sequence of complete non-archimedean fields with $\bb{C}\subseteq K_n$ trivially valued. Let $x_n \in X(K_n)$ be a sequence of points. Assume that $\lambda_{\partial X}(x_n)  > 0$ for all $n$ and let $\eps_n = \lambda_{\partial X}(x_n)^{-1}$. Let $x$ be an $(\eps_n)$-degeneration of $(x_n)$ and let $x_{\omega}$ be the induced $\cal{H}(\omega)$-point on $X$. Then $x_{\omega}$ does not lift to an $\cal{R}(\omega)$-point on $X$. 
\end{proposition}

\begin{proof}
Note that the degeneration $x$ of $(x_n)$ exists by Proposition \ref{DegenerationExistence2}. The point $x_{\omega}$ lifts to an $\cal{R}(\omega)$-point if and only if $\lambda_{\partial X}(x_{\omega}) = 0$. As in the previous Proposition, we can let $U_1,\ldots,U_m$ be an affine open cover of $\ovl{X}$. For each $x_n$, there exists a $U_{i_n}$ such that the affine coordinates of $x_n$ satisfy $| \cdot | \leq 1$ by properness of $\ovl{X}$. Since there are only finitely many possibilities for $i_n$, we can find a $j$ such that the set $E = \{n \mid i_n = j\}$ is $\omega$-large. 
\par 
Let $f_1,\ldots,f_m \in \cal{O}(U_j)$ cut out $U_j \cap (\ovl{X} \setminus X)$. Then for $n \in E$, we know that 
$$\lambda_{\partial X}(x_n) = \min_{1 \leq i \leq m} - \log |f_i(a_{n,1},\ldots,a_{n,N})|$$
where the affine coordinates of $x_n$ in $U_j$ are $a_{n,1},\ldots,a_{n,N}$. Clearly $x_{\omega}$ has absolute value $\leq 1$ in the coordinates of $U_j$ because $E$ is $\omega$-large. Hence 
$$\lambda_{\partial X}(x_{\omega}) = - \lim_{\omega} \eps_n\min_{1 \leq i \leq m} - \log |f_i(a_{n,1},\ldots,a_{n,N})| = -\lim_{\omega} \eps_n \lambda_{\partial X}(x_n) = 1$$
and so $x_{\omega}$ does not lift to an $\cal{R}(\omega)$-point.\end{proof}

\subsection{Application to Moduli of Abelian Varieties} \label{sec: DegenAbelian}
We now apply the above construction to the moduli space of principally polarized abelian varieties. We would like to work with $\cal{A}_g$ directly, but since it is an algebraic stack and not a variety, we will instead work with $\cal{A}_{g,3}$, the fine moduli space of principally polarized abelian varieties of dimension $g$ with level $3$ structure. We will work over $\bb{C}$, in which case $\cal{A}_{g,3}$ is a quasi-projective irreducible variety over $\bb{C}$ by \cite[IV.6.8]{FC90}. \par 
For an abelian variety $A/K$ where $K$ is a non-archimedean algebraically closed field, we say that $A$ has bad reduction if $A$ does not extend to an abelian scheme $\cal{A}$ over $R$ where $R$ is the valuation ring of $K$. We first show that $\cal{A}_{g,3}$ is able to capture the notion of bad reduction.

\begin{proposition} \label{AbelianBadReduction1}
Let $A/K$ be a principally polarized abelian variety of dimension $g$, where
$K=\Frac(R)$ for a discrete valuation ring $R$. Suppose that $A$ has good
reduction. Then any $K$-point of $\mathcal A_{g,3}$ whose underlying abelian
variety is $A$ extends to an $R$-point of $\mathcal A_{g,3}$. Equivalently, if a
$K$-point of $\mathcal A_{g,3}$ does not extend to $R$, then its underlying
abelian variety has bad reduction.
\end{proposition}

\begin{proof}
The natural morphism
\[
  \mathcal{A}_{g,3} \longrightarrow \mathcal{A}_g
\]
is finite étale, hence proper. Therefore a $K$-point of $\mathcal{A}_{g,3}$
extends to an $R$-point if and only if its image in $\mathcal{A}_g$ does.
Thus it suffices to show that if $A$ has good reduction, then the $K$-point of
$\mathcal{A}_g$ corresponding to $(A,\lambda)$ extends to $R$.

By assumption, there exists an abelian scheme $\mathcal{A}/R$ whose generic
fiber is $A$. Extending the $K$-point of $\mathcal{A}_g$ amounts to extending
the given principal polarization $\lambda$ on $A$ to a principal polarization on
$\mathcal{A}$.

By \cite[Proposition 8.9.1]{Gro64}, since $\mathcal{A}$ is of finite presentation and smooth over
$R$, there exists a finitely generated $\mathbb{Z}$-subalgebra $R'\subseteq R$
with fraction field $K'$ and an abelian scheme $\mathcal A'/R'$ whose base
change to $R$ is $\mathcal{A}$, and whose generic fiber $A'/K'$ pulls back to
$A/K$. Moreover, after replacing $R'$ by its integral closure in $K'$, we may
assume that $R'$ is a normal domain and that the given principal polarization
$\lambda$ descends to a polarization $\lambda'$ on $A'/K'$.

By the proof of \cite[Theorem~1.9]{FC90}, the polarization $\lambda'$ extends
to a polarization $\lambda'_{R'}$ on $\mathcal{A}'/R'$. Since the degree of a polarization is
locally constant in families and the generic fiber is principally polarized, the
extended polarization on $\mathcal{A}'/R'$ is again principal.

For $f:\mathrm{Spec}(R)\to\mathrm{Spec}(R')$, form the base changes \[\mathcal{A} = \mathcal{A}' \times_{\Spec(R')} \Spec(R),\qquad\mathcal{A}^{\vee} = (\mathcal{A}')^{\vee} \times_{\Spec(R')} \Spec(R),\]
and let 
\[p: \mathcal{A} \to \mathcal{A}', \qquad 
q:\mathcal{A} \to \Spec(R),\]
be the projections associated to $\mathcal{A}$. Consider the morphisms
\[u_1:=\lambda'_{R'}\circ p:\mathcal{A}\to(\mathcal{A}')^{\vee},
\qquad u_2:=q:\mathcal{A} \to \Spec(R).\]
These satisfy the compatibility condition \[\pi^{\vee} \circ u_1=f\circ u_2,\]
where $\pi^{\vee}:(\mathcal{A}')^{\vee}\to\Spec(R')$ is the structure morphism. Therefore, by the universal property of $(\mathcal{A}')^{\vee} \times_{\Spec(R')} \Spec(R)$, the pair $(u_1, u_2)$ determines a unique morphism $f^{*}\lambda' : \mathcal{A} \longrightarrow \mathcal{A}^{\vee}$. This is a polarization on $\mathcal{A}$. Since the degree of a polarization is preserved under base change, the morphism $f^{*}\lambda'$ is again principal, and it extends the given
polarization on the generic fiber. Thus the $K$-point of $\mathcal{A}_g$ defined by $(A,\lambda)$ extends to an
$R$-point, and by properness so does the original $K$-point of
$\mathcal{A}_{g,3}$. 
\end{proof}

Let $X = \cal{A}_{g,3}$ and in what follows let $(A_n)$ be a sequence of principally polarized abelian varieties over $K_n$ such that each $A_n$ has bad reduction. After choosing some level $3$ structure, we let $x_n \in X(K_n)$ represent $A_n$. Since $A_n$ has bad reduction, Proposition \ref{AbelianBadReduction1} tells us that $\lambda_{\partial X}(x_n) > 0$. Let $\eps_n = \lambda_{\partial X}(x_n)^{-1}$. By Proposition \ref{DegenerationExistence2}, an $(\eps_n)$-degeneration of $(x_n)$ exists and we obtain a $\scr{A}^{\eps}$-point of $X$. Since $X$ is a fine moduli space, this corresponds to a principally polarized abelian scheme $A_{\scr{A}^{\eps}}$ over $\scr{A}^{\eps}$. Furthermore by Propositions \ref{BadReduction1} and \ref{AbelianBadReduction1}, if $\omega$ is some non-principal ultrafilter, the induced abelian variety $A_{\mathcal{H}(\omega)}$ over $\cal{H}(\omega)$ has bad reduction. 
\par 
We now define the $\omega$-limit of a sequence of points $x_n \in A_n(K_n)$. In general, one can define it as the $\omega$-limit of $(x_n)$ inside the Berkovich analytification of $A_{\scr{A}^{\eps}}$. As this is a compact Hausdorff space, the $\omega$-limit will exist. However, we give a simpler ad hoc definition that is enough for our applications. 
\par 
We first put ourselves in a situation where we have an embedding of $A_{\scr{A}^{\eps}}$ into projective space $\bb{P}^N_{\scr{A}^{\eps}}$. Choose an affine open subset $U \subseteq X$ and a closed immersion of
$U$-schemes $\iota_U : \mathcal{A}_U \hookrightarrow \mathbb{P}^N_U$ for the universal abelian scheme $\mathcal{A}_U \to U$. (This is possible after shrinking $U$, since $\mathcal{A}_U \to U$ is a projective morphism.)
Suppose that each point $x_n$ lies in $U$, and that there is some constant
$c>0$ such that $\lambda_{\partial U}(x_n) \le c\,\lambda_{\partial X}(x_n)$
for all $n$. We claim that the $(\varepsilon_n)$-degeneration of $(x_n)$ constructed
above yields an $\mathscr{A}^\varepsilon$-point of $U$, i.e., a morphism $\phi : \Spec \mathscr{A}^\varepsilon \longrightarrow U$. Indeed, since each $x_n$ lies in $U$, we may view $(x_n)$ as a sequence of
$U(K_n)$-points. By assumption there exists a constant $c>0$ such that
$\lambda_{\partial U}(x_n) \le c\,\lambda_{\partial X}(x_n)$ for all $n$.
Recall that we have chosen $\varepsilon_n = \lambda_{\partial X}(x_n)^{-1}$, so that $\varepsilon_n \lambda_{\partial X}(x_n) = 1$.
Hence \[\varepsilon_n \lambda_{\partial U}(x_n)\le c\varepsilon_n \lambda_{\partial X}(x_n)=c\] and the sequence $\varepsilon_n \lambda_{\partial U}(x_n)$ is uniformly bounded from above. Since $U$ is quasi-projective, Proposition~\ref{DegenerationExistence2} applied with $X=U$ implies that $(x_n)$ admits an $(\varepsilon_n)$-degeneration with
values in $U$.  Thus the $(\varepsilon_n)$-degeneration constructed above gives
a morphism $\phi: \Spec \mathscr{A}^{\varepsilon} \longrightarrow U.$
\par 
Pulling back the universal abelian scheme on $U$ along $\phi$ gives an abelian
scheme \[A_{\mathscr{A}^\varepsilon}:=\mathcal{A}_U \times_U \Spec \mathscr{A}^\varepsilon\longrightarrow \Spec \mathscr{A}^\varepsilon.\]
Since closed immersions are stable under base change, the base change of
$\iota_U$ along $\phi$ is a closed immersion
\[\iota_{\mathscr{A}^\varepsilon}:A_{\mathscr{A}^\varepsilon}=\mathcal{A}_U \times_U \Spec \mathscr{A}^\varepsilon
\hookrightarrow\mathbb{P}^N_U \times_U \Spec \mathscr{A}^\varepsilon\simeq\mathbb{P}^N_{\mathscr{A}^\varepsilon}.\] Thus we obtain an embedding $A_{\mathscr{A}^\varepsilon}\hookrightarrow\mathbb{P}^N_{\mathscr{A}^\varepsilon}$.
\par
We now turn to a general description of $\mathscr{A}^\varepsilon$-points of
$\mathbb{P}_{\mathscr{A}^\varepsilon}$. Note that an $\scr{A}^{\eps}$-point of $\bb{P}^N_{\scr{A}^{\eps}}$ corresponds to $[a_0: \cdots : a_N]$, where each $a_i \in \scr{A}^{\eps}$ and $(a_0,\ldots,a_N)$ generates the unit ideal of $\scr{A}^{\eps}$. These are identified up to scalar multiplication by an element of $(\scr{A}^{\eps})^{\times}$. 

\begin{proposition} \label{ProjectiveCoordinate1}
Let $P_n=[a_{n,0}: \cdots : a_{n,N}]$ be a point of $\bb{P}^N_{K_n}$, where we normalize the $a_{n,i}$ so that $\max\{|a_{n,i}|\} = 1$. If we let $a_i = (a_{n,i})_{n=1}^\infty$, then $[a_0: \cdots : a_N]$ defines an $\scr{A}^{\eps}$-point on $\bb{P}^N_{\scr{A}^{\eps}}$ that is independent of our choice of lifts of $P_n$ satisfying $\max\{|a_{n,i}|\} = 1$. 
\end{proposition}

\begin{proof}
Since $|a_{n,i}| \le 1$ for all $n,i$, each $a_i$ lies in
$\mathscr{A}^\varepsilon$. For each $n$, choose an index $i_n$ with $|a_{n,i_n}|=1$.
Define $c_i=(c_{n,i})_{n=1}^\infty\in\mathscr{A}^\varepsilon$ by
$c_{n,i}=1$ if $i=i_n$ and $0$ otherwise.
Then \[\sum_{i=0}^N c_{n,i} a_{n,i} = a_{n,i_n},\] so $\left|\sum_i c_{n,i} a_{n,i}\right| = 1$ for every $n$.
Thus the element \[u:=\sum_{i=0}^Nc_i a_i\in\mathscr{A}^\varepsilon\] satisfies $|u_n| = 1$ for all $n$, and therefore $u \in
(\mathscr{A}^\varepsilon)^\times$. Hence $(a_0,\ldots,a_N)$ generates the unit ideal in $\mathscr{A}^\varepsilon$, and so defines a point of $\mathbb{P}^N_{\mathscr{A}^\varepsilon}$.

If $(a_{n,i}')_{n=1}^\infty$ is another choice of homogeneous coordinates with
$\max_i |a_{n,i}'|=1$, then for each $n$ the two tuples differ by a scalar
$\lambda_n \in K_n^\times$ with $|\lambda_n|=1$.
Hence $\lambda=(\lambda_n)\in(\mathscr{A}^\varepsilon)^\times$,
and $a_i' = \lambda a_i$.
Therefore $[a_0':\cdots:a_N'] = [a_0:\cdots:a_N]$ in
$\mathbb{P}^N_{\mathscr{A}^\varepsilon}(\mathscr{A}^\varepsilon)$.\end{proof}

Hence given $x_n \in \bb{P}^N(K_n)$, we obtain a well-defined $x \in \bb{P}^N(\scr{A}^{\eps})$, which we call the $(\eps_n)$-degeneration of $(x_n)$. Note that this is the unique section $x: \Spec \scr{A}^{\eps} \to \bb{P}^N_{\scr{A}^{\eps}}$ that induces $x_n: \Spec K_n \to \bb{P}^N_{K_n} \xhookrightarrow{} \bb{P}^N_{\scr{A}^{\eps}}$ for each $n \in \bb{N}$, where uniqueness follows from Proposition \ref{DegenerationUniqueness1}. We claim that if the $x_n$ are chosen to be in $A_n(K_n)$, then $x \in A_{\scr{A}^{\eps}}(\scr{A}^{\eps})$.



\begin{proposition}\label{ProjectiveCoordinate2}
Let $A_{\mathscr{A}^\varepsilon}\subseteq\mathbb{P}^N_{\mathscr{A}^\varepsilon}$ be the
closed subscheme obtained by base change from the projective embedding
$A_U \hookrightarrow \mathbb{P}^N_U$ as described above.
Suppose that for each $n$, the point
$x_n \in \mathbb{P}^N(K_n)$ lies in $A_n(K_n)$.
Let $x : \Spec(\mathscr{A}^\varepsilon)\to\mathbb{P}^N_{\mathscr{A}^\varepsilon}$
be the $\mathscr{A}^\varepsilon$-point of $\mathbb{P}_{\mathscr{A}^\varepsilon}^N$ induced by the $(\varepsilon_n)$-degeneration of the sequence $(x_n)$, obtained via Proposition \ref{ProjectiveCoordinate1}.
Then $x$ factors through $A_{\mathscr{A}^\varepsilon}$; equivalently, $x \in A_{\mathscr{A}^\varepsilon}(\mathscr{A}^\varepsilon)$.
\end{proposition}

\begin{proof}
Let $I \subseteq \mathscr{A}^\varepsilon[T_0,\ldots,T_N]$ be the homogeneous ideal
defining the closed subscheme
$A_{\mathscr{A}^\varepsilon} \subseteq \mathbb{P}^N_{\mathscr{A}^\varepsilon}$.
Choose homogeneous coordinates
\[[a_0:\cdots:a_N] \in \mathbb{P}^N_{\mathscr{A}^\varepsilon}(\mathscr{A}^\varepsilon)\]
representing the morphism
$x:\Spec(\mathscr{A}^\varepsilon) \to \mathbb{P}^N_{\mathscr{A}^\varepsilon}$.
Then $x$ factors through $A_{\mathscr{A}^\varepsilon}$ if and only if \[f(a_0,\ldots,a_N) = 0 \quad\text{in }\mathscr{A}^\varepsilon\] for every homogeneous polynomial $f \in I$.

For each $n$, the specialization map
$\mathscr{A}^\varepsilon \to K_n$ sends $a_i$ to the corresponding
coordinate $a_{n,i}$ of $x_n$, and the specialization of $x$ is $x_n$.
Since $x_n$ lies in $A_n \subseteq \mathbb{P}^N_{K_n}$, we have \[f_n(a_{n,0},\ldots,a_{n,N}) = 0 \quad\text{in }K_n\]
for every generator $f=(f_n) \in I$ and every $n$, where $f_n \in K_n[T_0,\ldots,T_N]$
denotes the specialization of $f$. Hence \[f(a_0,\ldots,a_N)=(f_n(a_{n,0},\ldots,a_{n,N}))_{n=1}^\infty=(0)_{n=1}^\infty=0\] as an element of $\mathscr{A}^\varepsilon$.

Thus every polynomial in $I$ vanishes on $(a_0,\ldots,a_N)$, so $x$ annihilates
the defining ideal of $A_{\mathscr{A}^\varepsilon}$.
It follows that $x$ factors through the closed immersion
$A_{\mathscr{A}^\varepsilon}\hookrightarrow \mathbb{P}^N_{\mathscr{A}^\varepsilon}$.\end{proof}

Hence for any ultrafilter $\omega$, we obtain a well-defined point
$x_{\omega} \in A_{\omega}(\mathcal{H}(\omega))$.  We shall call $x_{\omega}$
the $\omega$-limit of the sequence $(x_n)$.  
Observe that if $x : \Spec \mathscr{A}^{\varepsilon} \to A_{\mathscr{A}^{\varepsilon}}$
is the section arising from the $(\varepsilon_n)$-degeneration of $(x_n)$, then
the point $\Spec \mathcal{H}(\omega) \to A_{\omega}$ given by
the base change of $x$ along $\Spec\mathscr{A}^{\varepsilon} \to \Spec\mathcal{H}(\omega)$ equals $x_\omega$.

\par 
The main property that the $\omega$-limit satisfies is the following. For any homogeneous degree $d$ polynomial $F$ on $\bb{P}^N_{\scr{A}^{\eps}}$ with $\scr{A}^{\eps}$-coefficients, we may define a function on $\bb{P}^N_{\cal{H}(\omega)}$ by
$$\psi_F = \frac{|F_{\omega}(x_0,\ldots,x_N)|_{\omega}}{ \max\{ |x_0|_{\omega} ,\ldots,|x_N|_{\omega}\}^d},$$ where $F_{\omega}$ is the induced homogeneous degree $d$ polynomial over $\mathcal{H}(\omega)$ and $(x_0,\ldots,x_N)$ is any homogeneous lift of our point in $\bb{P}^N_{\cal{H}(\omega)}$. Let $F_n$ be degree $d$ polynomial on $K_n$ induced by $F$, and let $\psi_{F_n}$ be the corresponding function on $\bb{P}^N_{K_n}$. Then 
$$\psi_F(x_{\omega}) = \lim_{\omega} \psi_{F_n}(x_n)^{\eps_n} \implies \log \psi_F(x_{\omega}) = \lim_{\omega} \eps_n \log \psi_{F_n}(x_n)$$
where we take the $\omega$-limit on the extended real line $[-\infty,\infty]$.
\par 
We now recall the relevant compactification of $X=\mathcal{A}_{g,3}$. Faltings--Chai \cite{FC90} constructed, given some choice of auxiliary data, a compactification $\overline{X}=\overline{\mathcal{A}}_{g,3}$ of $X$, which is a proper algebraic stack over $\mathbb{C}$. Different choices of this auxiliary data generally give rise to non-isomorphic compactifications. Moreover, there exist choices of compactification for which $\overline{X}$ is in fact a projective algebraic variety over $\mathbb{C}$ \cite[Theorem~5.8]{FC90}. We will let $\overline{X}=\overline{\mathcal{A}}_{g,3}$ denote such a choice. The compactification $\overline{X}$ carries a universal semiabelian variety $G \to \overline{X}$ extending the universal abelian variety over $X$
\cite[Theorem~IV.6.7]{FC90}.
\par 

For each abelian variety $A_n/K_n$ with associated point $x_n\in X(K_n)$, and $R_n$ the valuation ring of $K_n$, the properness of the compactification $\overline{\mathcal{A}}_{g,3}$ guarantees that $x_n$ extends to a morphism $\Spec(R_n)\to\overline{\mathcal{A}}_{g,3}$, and pulling back the universal semiabelian scheme $G\to\overline{\mathcal{A}}_{g,3}$ along this morphism yields a semiabelian scheme $G_n/R_n$. Now let $x_n \in A_n(K_n)$ be a sequence of points such that each $x_n$ lifts to a point in $G_n(R_n)$. We will show that in this case, the resulting $\omega$-limit point $x_\omega\in A_\omega(\mathcal{H}(\omega))$
extends to an $\mathcal{R}(\omega)$-point of the semiabelian model
$G_\omega$, where $\mathcal{R}(\omega)$ is the valuation ring of $\mathcal{H}(\omega)$, and
$G_\omega$ is defined by the base change \[G_\omega:=G\times_{\overline{\mathcal{A}}_{g,3}} \Spec(\mathcal{R}(\omega)).\] We do so in order to further show the following fact about the retraction of $x_\omega$ (still assuming that $x_n$ lifts to an $R_n$-point of $G_n$ for all $n$). Since $A_{\omega}$ has bad reduction, it has a canonical skeleton $S(A_{\omega})$, which is isomorphic to a real torus $\bb{R}^k/\Lambda$, and a retraction map $r: A_{\omega}^{\an} \to S(A_{\omega})$. We will prove that $r(x_{\omega}) = 0$. By Proposition \ref{KernelReduction1}, it suffices to show that $x_{\omega}$ lifts to an $R(\omega)$-point of $G_{\omega}$. 
\par 

To do this, set
\[
\mathscr{B}^{\varepsilon}
:= \{ (x_n) \in \prod_{n=1}^{\infty} K_n \mid |x_n| \le 1 \}.
\]
Then $\mathscr{B}^{\varepsilon}$ is a subring of $\mathscr{A}^{\varepsilon}$, and
for every ultrafilter $\omega$ the induced homomorphism
$\mathscr{A}^{\varepsilon} \to \mathcal{H}(\omega)$ restricts to a homomorphism $\mathscr{B}^{\varepsilon} \longrightarrow \mathcal{R}(\omega)$, since elements of $\mathscr{B}^{\varepsilon}$ have $\omega$-norm $\le 1$. We will construct a semiabelian scheme $\mathcal{G}/\mathscr{B}^{\varepsilon}$ whose base change to $\mathscr{A}^{\varepsilon}$ is $A_{\mathscr{A}^{\varepsilon}}$, and whose base change to $\mathcal{R}(\omega)$ is $G_\omega$. We begin with the following proposition.

\begin{proposition} \label{DegeneratingSemiAbelian2}
Let $(x_n) \in \mathcal{A}_{g,3}(K_n)$ be a sequence of points and let $(\varepsilon_n)$ be a fundamental rescaling parameter. If $A_{\mathscr{A}^\varepsilon}$ is the abelian scheme corresponding to our $(\varepsilon_n)$-degeneration, then there exists a semiabelian scheme $G_{\mathscr{B}^\varepsilon}$ whose base change to $\mathscr{A}^\varepsilon$ is $A_{\mathscr{A}^\varepsilon}$.
\end{proposition}

\begin{proof}
Let $X=\mathcal{A}_{g,3}$ and let $(x_n)$ be the given sequence.
By assumption, $(\varepsilon_n)$ is a fundamental rescaling parameter, so $\varepsilon_n=\lambda_{\partial X}(x_n)^{-1}$,
and therefore $\varepsilon_n\lambda_{\partial X}(x_n)=1$ is uniformly bounded. Thus by Proposition~\ref{DegenerationExistence2}, the
$(\varepsilon_n)$-degeneration of $(x_n)$ exists on $X$ itself; in other words, the
degeneration defines a morphism $x:\Spec \mathscr{A}^\varepsilon\longrightarrow X\subseteq\overline{\mathcal{A}}_{g,3}$. By Remark~\ref{rmk:Bep}, this morphism extends to a morphism $\tilde x:\Spec \mathscr{B}^\varepsilon\longrightarrow \overline{\mathcal{A}}_{g,3}$. Let $G\longrightarrow \overline{\mathcal{A}}_{g,3}$
denote the universal semiabelian scheme on the Faltings--Chai compactification.
Define
\[G_{\mathscr{B}^\varepsilon}:=G\times_{\overline{\mathcal{A}}_{g,3}}\Spec \mathscr{B}^\varepsilon,\] which is a semiabelian scheme over $\mathscr{B}^\varepsilon$. Base changing along $\mathscr{B}^\varepsilon\to \mathscr{A}^\varepsilon$ gives \[G_{\mathscr{A}^\varepsilon}:=G_{\mathscr{B}^\varepsilon}\times_{\mathscr{B}^\varepsilon}\Spec \mathscr{A}^\varepsilon=G\times_{\overline{\mathcal{A}}_{g,3}}\Spec \mathscr{A}^\varepsilon.\] Since $x$ factors through the open subvariety
$X=\mathcal{A}_{g,3}\subseteq\overline{\mathcal{A}}_{g,3}$, and since the restriction of $G$ to $X$ is the universal abelian scheme, it follows that $G_{\mathscr{A}^\varepsilon}$ is obtained by pulling back the universal abelian scheme along $x$. By definition, this pullback is precisely the degeneration family
$A_{\mathscr{A}^\varepsilon}$. Hence there is a canonical isomorphism $G_{\mathscr{A}^\varepsilon}\cong A_{\mathscr{A}^\varepsilon}$.
\end{proof}

Now suppose we have a sequence of points $(x_n)$ of $A_n(K_n)$ such that $x_n$ lifts to an $R_n$-point of $G_n$ for all $n$. We wish to show that we may glue these points together to obtain a $\scr{B}^{\eps}$-section of $G_{\scr{B}^{\eps}}$

\begin{proposition}\label{SemiAbelianSection1} Let $G_{\mathscr{B}^\varepsilon}$ be as in Proposition \ref{DegeneratingSemiAbelian2}. There exists a section $x : \Spec  \scr{B}^{\eps}\to G_{\scr{B}^{\eps}}$ such that for each $n$, the base change of $x$ along the natural map $\Spec R_n \to \Spec  \scr{B}^{\eps}$ agrees with the given morphism $x_n : \Spec R_n \to G_n$. \end{proposition}

\begin{proof}
For each $n$, the given point $x_n\in G_n(R_n)$ corresponds to a morphism \[x_n:\Spec R_n \longrightarrow G_n= G_{\mathscr B^\varepsilon}\times_{\Spec\mathscr B^\varepsilon}\Spec R_n.\]
Equivalently, by composing with the projection
$G_n\to G_{\mathscr B^\varepsilon}$, we may view $x_n$ as a morphism \[x_n:\Spec R_n \longrightarrow G_{\mathscr B^\varepsilon}\]
whose composition with $G_{\mathscr B^\varepsilon}\to\Spec\mathscr B^\varepsilon$
is induced by the natural projection
$\pi_n:\mathscr B^\varepsilon\to R_n$.

Since $G_{\mathscr B^\varepsilon}$ is a scheme of finite type over
$\mathscr{B}^\varepsilon$, it admits a finite affine open cover
$\{W_i\}_{i=1}^m$ with each $W_i = \Spec C_i$. Because each $R_n$ is a valuation ring, the image of
$\Spec R_n$ under $x_n$ must be contained in a single affine open $W_{i_n}$. On each $W_{i_n}=\Spec(C_{i_n})$, the morphism $x_n$ is given by a ring homomorphism $C_{i_n}\to R_n$. Taking these homomorphisms coordinatewise defines a ring homomorphism 
$$\prod_{i=1}^{m} C_i \to \prod_n R_n = \scr{B^{\eps}}$$
as follows. For each $n$, we first project $\prod_{i=1}^{m} C_i$ down to $C_{i_n}$ and then compose with $x_n: C_{i_n} \to R_n$ to obtain a morphism $\prod_{i=1}^{m} C_i \to R_n$. Taking the product gives us a ring morphism $\prod_{i=1}^{m} C_i \to \scr{B}^{\eps}$.  
\par 
Thus we obtain a $\scr{B^{\eps}}$-point of $\prod_{i=1}^{m} W_i$ which naturally defines a $\scr{B^{\eps}}$-point $x:\Spec\mathscr B^\varepsilon \to G_{\mathscr B^\varepsilon}$. By construction, for each $n$ the base change of $x$ along
$\Spec R_n\to\Spec\mathscr B^\varepsilon$ induces the original morphism
$x_n:\Spec R_n\to G_n$, since both are defined by the same coordinate ring
homomorphism $\Gamma(W_{i_n},\mathcal O)\to R_n$.\end{proof}

Recall that the $(\varepsilon_n)$-degeneration of $(x_n)$ determines a section
\[
x_{\mathscr A} : \Spec \mathscr{A}^\varepsilon \longrightarrow A_{\mathscr{A}^\varepsilon}.
\]
For any ultrafilter $\omega$, base change along
$\Spec \mathcal{H}(\omega) \to \Spec \mathscr{A}^\varepsilon$
defines a point
\[x_\omega \in A_\omega(\mathcal{H}(\omega)),\qquad A_\omega := A_{\mathscr{A}^\varepsilon} \times_{\Spec \mathscr{A}^\varepsilon} \Spec \mathcal{H}(\omega),\] which by definition is the $\omega$-limit of the sequence $(x_n)$. On the other hand, by Proposition~\ref{SemiAbelianSection1} (note that we are making the standing assumption that $x_n$ lifts to an $R_n$-point of $G_n$ for all $n$), there exists a section
\[x_{\mathscr{B}}:\Spec\mathscr{B}^\varepsilon\longrightarrow G_{\mathscr{B}^\varepsilon}\]
whose base change to $\Spec R_n$ recovers each $x_n:\Spec R_n\to G_n$.
Base changing this section along
\[\Spec \mathcal{R}(\omega) \to \Spec \mathscr{B}^\varepsilon\]
produces a section \[x_\omega^G \in G_\omega(\mathcal{R}(\omega)),\qquad
G_\omega:=G_{\mathscr{B}^\varepsilon} \times_{\Spec \mathscr{B}^\varepsilon} \Spec \mathcal{R}(\omega).\] Passing to the generic fiber identifies the induced
$\mathcal{H}(\omega)$-point of $G_\omega$ with a point of
\[A_\omega = G_\omega \times_{\Spec \mathcal{R}(\omega)} \Spec \mathcal{H}(\omega),\] using the canonical identification $A_{\mathscr{A}^\varepsilon}\cong G_{\mathscr{A}^\varepsilon}$ from Proposition~\ref{DegeneratingSemiAbelian2}. By construction, this $\mathcal H(\omega)$-point of $A_\omega$, induced from
$x_\omega^G$, and the $\omega$-limit point $x_\omega$ defined via
$x_{\mathscr A}$ both specialize to the same $x_n$ for every $n$. Since $\mathcal{H}(\omega)$ is the ultraproduct residue field of the $K_n$ with
respect to $\omega$, so that a morphism $\Spec\mathcal{H}(\omega)\to A_\omega$ is uniquely determined by its compositions with the projections to the $A_n$, it follows that these two $\mathcal{H}(\omega)$-points coincide, and hence that the
$\omega$-limit point $x_\omega \in A_\omega(\mathcal{H}(\omega))$ lifts to an
$\mathcal{R}(\omega)$-point of $G_\omega$ (namely $x_\omega^G$). By Proposition~\ref{KernelReduction1},
this implies that the retraction of $x_\omega$ to the skeleton vanishes.

\begin{corollary} \label{SemiAbelianSection2}
Assume each $x_n$ lifts to an $R_n$-point of $G_n$. Then the $\omega$-limit $x_{\omega}$ lifts to an $\cal{R}(\omega)$-point of $G_{\omega}$. In particular, $r(x_{\omega}) = 0$ in $S(A_{\omega})$. 
\end{corollary}

\section{Uniform Positivity for Arakelov-Green's Functions}
We now combine the results in Sections \ref{sec : AbVarPos} and \ref{sec: DegenAbelian} to obtain a uniform positive lower bound on $g_{\cal{B}_n}(x_1,\ldots,x_{c(n)})$ under suitable hypotheses. This will be a key input in obtaining uniform bounds on the number of small points on abelian varieties.
\par 
Let $U \subseteq \cal{A}_{g,3}$ be an open irreducible affine chart such that for the universal abelian scheme $A_U$ on $U$, we have an embedding $A_U\hookrightarrow\bb{P}^N_U$ and furthermore, for some $n \in \bb{N}$, the multiplication by $[n]$ morphism on $A_U$ extends to a morphism $f: \bb{P}^N_U \to \bb{P}^N_U$ of degree $d = n^2$. Fixing a homogeneous lift $F = (F_0,\ldots,F_N)$ of $f$, this means that the homogeneous resultant $\Res(F)$ is an invertible element of $\mathcal{O}(U)$. 
\par 
This allows us to perform the first author's construction of good bases and dynamical Arakelov-Green's functions in a relative setting over $U$. If we let $F^k=(F^{(k)}_0,\ldots,F^{(k)}_N)$, then \cite[Lemma 4]{Ing22} that any homogeneous polynomial $G$ with coefficients in $\mathcal{O}(U)$ of degree at least $(N+1)(d^k-1)+1$ can be written as 
$$G = \sum P_i F^{(k)}_i$$ for appropriate homogeneous polynomials $P_i$ over $\mathcal{O}(U)$.
In particular if we let $\cal{G}$ be the collection of polynomials of the form $(F_i^{(k)})^j$ for $1 \leq j \leq d$, then the set
\begin{equation} \label{eqn: GlobalSpanningSet}
\cal{F} = \bigcup_{j = \lfloor t_1 \rfloor }^{\lfloor t_2 \rfloor} \{\eta_i G_i \mid G_i \in \cal{G}^j \text{ and } \eta_i \text{ monomial of degree } \leq d(N+1)\}
\end{equation}
is a spanning set for the space of degree $n$ polynomials, where we may take $t_1,t_2 = O(\log n)$. 
\par 
Let $L$ be the ample line bundle on $\cal{A}_U$ corresponding to $\mathcal{O}(1)$. By replacing $L$ by $L^4$, we know that for each $s \in U$, $L$ gives a projectively normal embedding of $\cal{A}_s$ into $\bb{P}^N_s$ \cite{Koi76} and hence we have a surjection $H^0(\bb{P}_U^N,\mathcal{O}(n)) \to H^0(\cal{A}_U,L^n)$ For any $K$-point of $U$, pulling back via the $K$-point gives us an embedding $A_K \xhookrightarrow{} \bb{P}^N_K$ and also a surjection \begin{equation}\label{eqn:Lbundle} H^0(\bb{P}^N_K,\mathcal{O}(n)) \to H^0(A_K,L^n).\end{equation} By \cite[Proposition 6.13]{MFK94}, we have $\dim H^0(A_K,L^n)=c(n)$ for some constant $c(n)$ that is independent of our $K$-point on $U$, as $\pi_*(L^n)$ is a locally free sheaf on $U$.  
\par 
As the multiplication by $n$ map extends to a morphism $f$ on $\bb{P}^N_U$ and we have fixed a homogeneous lift $F:\mathbb{A}_U^{N+1}\to\mathbb{A}_U^{N+1}$, pulling back by the $K$-point gives us a consistent choice of a polarized dynamical system $(A_K,f_K)$ and homogeneous lift $F_K$. We define the notion of good basis pointwise on $\cal{A}_U$, where a good basis $\cal{B}_n$ for $A_K$ is defined as in Definition \ref{def:Hn} with respect to the lift $F_K$ of $f_K$. Observe that any good basis is the pullback of a subset of sections of $\cal{F}$ by the $K$-point.
\par  
Now assume that $(K,|\cdot|)$ is a complete valued field. Given any set of sections $\cal{B}_n \subseteq H^0(A_K,L_K^n)$ of size $c(n)$, recall from \S\ref{section:AG} that we may define an Arakelov-Green's function $g_{\cal{B}_n}:A_K^{c(n)}\to\bb{R}\cup\{\infty\}$ by setting
$$g_{\cal{B}_n}(x_1,\ldots,x_{c(n)}) = \frac{1}{c(n)} \sum_{i=1}^{c(n)} \Ht_{F_k}(\tilde{x}_i) - \frac{1}{n \cdot c(n)} \log\left| \det \big( ( \eta_j(\tilde{x}_i))_{\eta \in \cal{B}} \big)  \right|+ r(F_K)$$
where $(f_K,F_K)$ is the pullback of $(f,F)$ by $\Spec K \to U$ and 
$$r(F_K) = \frac{1}{d^N(N+1)(d-1)} \log |\Res(F_K)|^{-1}.$$
Here, the $\tilde{x}_i$ are lifts of $x_i$ to $\bb{A}^{N+1}(K)$. Recall that $g_{\cal{B},n}$ is independent of the choice of lifts of the $x_i$. Note that this definition makes sense even if $\cal{B}_n$ is not a basis, in which case we have $g_{\cal{B},n}(x_1,\ldots,x_{c(n)}) = +\infty$ for all $x_i$.  
\par
We first prove the following basic continuity property. Fix a sequence $(\eps_n)$ of positive real numbers and assume that we have an $\scr{A}^{\eps}$-point of $U$. Then pulling back $A_U/U$ along $\Spec \scr{A}^{\eps} \to U$, we obtain an abelian scheme $A_{\scr{A}^{\eps}}$ along with an embedding $A_{\scr{A}^{\eps}}\hookrightarrow\bb{P}^N_{\scr{A}^{\eps}}$ and a morphism $f: \bb{P}^N \to \bb{P}^N$ extending $[n]$ on $A_{\scr{A}^\eps}$ with homogeneous lift $F$. Then for a set of sections $\cal{B}_k \subseteq H^0(\bb{P}_{\scr{A}^{\eps}}^N,\mathcal{O}(k))$ of size $c(k)$, pulling back via $K_n \to \scr{A}^{\eps}$ and $\cal{H}(\omega) \to \scr{A}^{\eps}$, we get induced sets of sections $\cal{B}_{k,n}$ and $\cal{B}_{k,\omega}$ for $H^0(A_n, L^k_{n})$ and $H^0(A_{\omega}, L_{\omega}^k)$ respectively. We thus get Arakelov-Green's functions $g_{\cal{B}_k,n}$ and $g_{\cal{B}_{k},\omega}$.

\begin{proposition} \label{ArakelovContinuity1} Let $A_{\scr{A}^{\eps}}$ be the abelian scheme constructed above and for each $n$, let $x_{n,1},\ldots,x_{n,c(k)}$ be points on $A_n(K_n)$. Let $x_i$ be the $\omega$-limit of $(x_{1,i},x_{2,i},\ldots)$ for any non-principal ultrafilter $\omega$. Then we have 
$$\lim_{\omega} \eps_n g_{\cal{B}_k,n}(x_{n,1},\ldots,x_{n,c(k)}) = g_{\cal{B}_{k},\omega}(x_1,\ldots,x_{c(k)}),$$
where we take the $\omega$-limit in the extended reals $[-\infty,\infty]$. 
\end{proposition}

\begin{proof}
By assumption, the sequence $(f_n)$ of rational maps admits an $(\varepsilon_n)$-degeneration in $\operatorname{Rat}^N_d$. Equivalently, the associated points of $\operatorname{Rat}^N_d(K_n)$ define an $\mathscr{A}^\varepsilon$-point of $\operatorname{Rat}^N_d$. Under the standard compactification $\operatorname{Rat}^N_d \subseteq\mathbb{P}^{N'}$ with boundary $\{\operatorname{Res}=0\}$, Proposition~\ref{DegenerationExistence2} implies that
the sequence $(\varepsilon_n \log|\operatorname{Res}(f_n)|^{-1})$ is uniformly bounded from above.
\par 
Now choose lifts $F_n$ of $f_n$ such that $||F_n||_n = 1$. Then $F_{\omega} = \lim_{\omega} F_n$ is a lift of $f_{\omega}$ satisfying $||F_{\omega}||_{\omega} = 1$. 
We clearly have 
$$\lim_{\omega} \eps_n r(F_n) = r(F_{\omega})$$
and 
$$\lim_{\omega} \eps_n \log |\det(\eta_j(\tilde{x}_{i,n}))_{\eta \in \cal{B}}|_n = \log |\det(\eta_j(\tilde{x}))_{\eta \in \cal{B}}|_{\omega}$$
by direct computation. It suffices to show that 
$$\Ht_{F_{\omega}}(\tilde{x}) = \lim_{\omega} \eps_n \Ht_{F_n}(\tilde{x}_n).$$
By definition, for all $\omega \in \beta \bb{N}$ we have 
$$\Ht_{F_{\omega}} = \eps_{\omega} \lim_{n \to \infty} \frac{1}{d^n} \log ||F_{\omega}^n||_n$$
where we set $\eps_{\omega} = 1$ for $\omega \in \beta \bb{N} \setminus \bb{N}$. 
We have the telescoping sum
$$\eps_{\omega} \left(\Ht_{F_{\omega}}(\tilde{x}_{\omega}) - \frac{1}{d^k} \log |F_{\omega}^{(k)}(\tilde{x}_{\omega})|_{\omega} \right)$$
$$= \eps_{\omega} \sum_{n \geq k} \frac{1}{d^n} \left(\frac{1}{d} \log |F_{\omega}^{({n+1})}(\tilde{x}_{\omega})|_{\omega} - \log |F_{\omega}^n(\tilde{x}_{\omega})|_{\omega} \right),$$
and by \cite[Lemma5]{Ing22}, we may bound 
$$\eps_{\omega} \left|\frac{1}{d} \log |F^{({n+1})}(\tilde{x}_{\omega})|_{\omega} - \log |F^n(\tilde{x}_{\omega})|_{\omega} \right|$$
in terms of $\eps_{\omega} |\Res(f)|_{\omega}^{-1}$ and $\eps_{\omega} ||F||_{\omega}$. As $\eps_{\omega} |\Res(f)|^{-1}_{\omega}$ is uniformly bounded and $||F||_{\omega} = 1$, there exists a uniform constant $C > 0$ such that 
$$\left| \eps_{\omega} \left( \Ht_{F_{\omega}}(\tilde{x}_{\omega}) - \frac{1}{d^k} \log |F^{(k)}(\tilde{x}_{\omega})|_{\omega} \right) \right| \leq \frac{C}{d^k}$$
for all $\omega \in \beta \bb{N}$. Finally, by direct computation it is clear that 
$$\lim_{\omega} \frac{\eps_n}{d^k} \log |F_n^{(k)}(\tilde{x}_n)|_n = \eps_{\omega} \log |F_{\omega}^{(k)}(\tilde{x}_{\omega})|_{\omega},$$
as $(F_n^{(k)})$ is a sequence of polynomials whose coefficients gives $F_{\omega}^{(k)}$. We thus obtain 
$$\left| (\lim_{\omega} \eps_n \Ht_{F_n}(\tilde{x}_n) - \Ht_{F_{\omega}}(\tilde{x}_{\omega})\right| \leq \frac{C}{d^k}.$$
Taking $k \to \infty$ yields
$$\lim_{\omega} \eps_n \Ht_{F_n}(\tilde{x}_n) = \Ht_{F_{\omega}}(\tilde{x}_{\omega})$$
as desired.\end{proof}

Now let $A$ be a ppav defined over an algebraically closed complete non-archimedean valued field $K$, and assume its corresponding $K$-point of $X$ lies in $U$. Let $L$ be as in \eqref{eqn:Lbundle}.

By properness, we obtain an $R$-point of $\ovl{\cal{A}}_{g,p}$ and hence a semiabelian scheme $\cal{G}$ lifting $A$. Fix any increasing function $h: \bb{N} \to \bb{N}$ such that $h(n) = c(k_n)$ for some $k_n$. If $A$ has bad reduction, we say that $A$ is \emph{$n$-good} if for any collection of $x_1,\ldots,x_{h(n)}$ many points which all lift to $R$-points of $\cal{G}$, we have 
$$g_{\cal{B}_{k_n}}(x_1,\ldots,x_{c(k_n)}) \geq \frac{1}{n} \lambda_{\partial X}(A)+r(F_K)$$
where $X = \cal{A}_{g,3}$ and $\cal{B}_{k_n}$ is any good basis of $H^0(A,L^{k_n})$. Fix a constant $C > 0$. Recall that we have fixed an affine open $U \subset X = \cal{A}_{g,3}$ at the start of the section. We say that our abelian $A/K$ whose induced $K$-point on $\cal{A}_{g,3}$ lies in $U(K)$ is $C$-centered with respect to our affine open $U$ if 
$$\lambda_{\partial U}(A) \leq C \lambda_{\partial X}(A).$$
Here, we define $\lambda_{\partial U}$ using the minimal compactification $\cal{A}_{g,3}^*$ of $\cal{A}_{g,3}$. The main theorem of this section is then the following.

\begin{theorem} \label{UniformNeronBound1}
For any increasing function $h: \bb{N} \to \bb{N}$ where for each $n$, we have $h(n) = c(k)$ for some $k\in\bb{N}$, and any $C > 0$, there exists a nonempty finite list of integers $\{n_1,\ldots,n_m\}$ such that for any non-archimedean complete algebraically closed field $K$ with $\bb{C}$ a trivially valued subfield, and for any principally polarized abelian variety $A/K$ of dimension $g$ with bad reduction which is $C$-centered with respect to $U$, we have that
$$A/K \text{ is } n_i\text{-good for some } i \text{ with } 1 \leq i \leq m.$$
The set $\{n_1,\ldots,n_m\}$ can depend on the open set $U$ along with the rational map $f$ and homogeneous lift $F$. However, it is independent of $K$ and the dimension $g$ principally polarized abelian variety $A/K$ which is $C$-centered with respect to $U$.
\end{theorem}

\begin{proof}
Assume no such nonempty finite set $\{n_1,\ldots,n_m\}$ exists. Then we can find a sequence of dimension $g$ ppavs $A_n$ over $K_n$ having bad reduction, each $C$-centered with respect to $U$, such that for all $1 \leq i \leq n$, the abelian variety $A_n$ is not $i$-good for all $1 \leq i \leq n$. We fix choices of principal level $3$ structures so that each $A_n$ defines a point on $X = \cal{A}_{g,3}$. For each $n$, let $L_n$ be the line bundle on $A_n$ as in \eqref{eqn:Lbundle}. By assumption, after fixing a positive integer $n$, for all $n' \geq n$, if $h(n) = c(k)$, then there exists a good basis $\cal{B}'_{n'}$ of $H^0(A_{n'}, L_{n'}^{k})$, along with points $x_{n',1},\ldots,x_{n',c(k)}$ of $A_{n'}(K_{n'})$ that lift to $R_{n'}$-points of the semiabelian model $G_{n'}$ of $A_{n'}$, such that 
$$g_{\cal{B}'_{n'}}(x_{n',1},\ldots,x_{n',c(k)}) \leq \frac{1}{n} \lambda_{\partial X}(A_{n'}) + r(F_{n'}).$$
Now recall that any good basis $\cal{B}'_{n'}$ is a finite subset of $\cal{F}\subseteq H^0(\bb{P}^N_U,\cal{O}(k))$ where $\cal{F}$ is defined in \eqref{eqn: GlobalSpanningSet}. We now pass to a subsequence of $A_n$'s so that each good basis $\cal{B}'_{n'}$ of $H^0(A_{n'}, L^k_{n'})$ is induced from the same set of sections of $H^0(\bb{P}_U^{N}, \cal{O}(k))$. We may do this as there are only finitely many possible subsets of $\cal{F}$.
\par 
Let $\eps_n = \lambda_{\partial X}(A_n)^{-1}$, which is finite because of Proposition \ref{BadReduction1}. Then by Proposition \ref{DegenerationExistence2}, an $(\eps_n)$-degeneration of $(A_n)$ exists on $\cal{A}_{g,3}$. Furthermore since each $A_n$ is $C$-centered, it follows that our $(\eps_n)$-degeneration actually exists on $U$. Thus we obtain an abelian scheme $A_{\scr{A}^{\eps}}$ over $\Spec \scr{A}^{\eps}$ with an embedding $A_{\scr{A}^{\eps}} \xrightarrow{} \bb{P}^N_{\scr{A}^{\eps}}$ and an extension $f: \bb{P}^N \to \bb{P}^N$ of multiplication by $[n]$ on $A_{\scr{A}^\eps}$ along with a homogeneous lift $F$ of $f$. 

Choose a non-principal ultrafilter $\omega$. By Proposition \ref{Degenerating1}, $A_{\omega}$ has bad reduction over $\cal{H}(\omega)$, so there is a continuous map $r: A_{\omega}^{\an} \to \bb{R}^k/\Lambda$ where $\Lambda$ is a lattice and $k \geq 1$. Let $V$ be the closed set that retracts to the identity $0$ of $S(A_{\omega})$. 

Let $L_\omega$ be the line bundle on $A_\omega^{\an}$ as in \eqref{eqn:Lbundle} for $\cal{H}(\omega)$. Then by Proposition \ref{TransfiniteDiameterBound1}, we know that for any sequence of good bases $(\cal{B}_k)$ of $H^0(A_{\omega}, L_{\omega}^k)$, there exists a $\delta > 0$ such that we have 
$$\liminf_{k \to \infty} \inf_{(x_i)} g_{\cal{B}_k}(x_1,\ldots,x_{c(k)}) > \delta + r(F_{\omega})$$
where we take the infimum over all $c(k)$-tuples of classical points $(x_i)$ such that $x_i \in V \cap A_{\omega}(\cal{H}(\omega))$. Pick an $m\in\bb{Z}_{>0}$ large enough so that $\delta \geq \frac{1}{m}$ and
\begin{equation} \label{eq: Transfinite1}
\inf_{(x_i)} g_{\cal{B}_k}(x_1,\ldots,x_{c(k)}) > \delta + r(F_{\omega})
\end{equation}
where $c(k) = h(m)$ and $\cal{B}_k$ is any good basis of $H^0(A_{\omega},L_{\omega}^k)$. 

We now consider $m'\ge m$ and let $h(m) = c(k)$. We set $\cal{B}'_{\omega}$ to be the subset of $H^0(A_{\omega},L_{\omega}^k)$ that is induced from the same set of sections of $H^0(\bb{P}^N_U,\cal{O}(k))$ as our good bases $\cal{B}'_{n'}$. We now apply Proposition \ref{ArakelovContinuity1} and obtain
$$g_{\cal{B}'_{\omega}}(x_1,\ldots,x_{c(k)}) \leq \lim_{\omega} \left(\frac{1}{m} \lambda_{\partial X}(A_{m'}) \eps_{m'}  + r(F_{m'}) \eps_{m'} \right) = \frac{1}{m} + r(F_{\omega}),$$
where we take the $\omega$-limit over the $m'$ and each $x_i$ is the $\omega$-limit of the sequence $(x_{m',i},x_{m'+1,i},\ldots)$. Note that $\cal{B}'_{\omega}$ must then be a basis of $H^0(A_{\omega},L_{\omega}^k)$ as otherwise $g_{\cal{B}'_{\omega},k}(x_1,\ldots,x_{c(k)})$ equals $\infty$. But by Corollary \ref{SemiAbelianSection2}, each $x_i$ retracts to the identity of $\bb{R}^k/\Lambda$ and hence lies in $V$. By our choice of $m$, this contradicts \eqref{eq: Transfinite1}.
\end{proof}

We will apply Theorem \ref{UniformNeronBound1} in the following setting: let $A$ have level $3$ structure and be defined over a discretely valued field $K$ with valuation ring $R$, and assume that $A$ is semistable over $K$. Let $G/R$ be the connected part of the N\'{e}ron model of $A$. Then $G$ is a semiabelian scheme over $R$ and by uniqueness of the semiabelian model \cite[I.2.7]{FC90}, $G$ corresponds to the semiabelian scheme induced by the $R$-point on $\ovl{\cal{A}}_{g,3}$ that lifts the $K$-point on $\cal{A}_{g,3}$ corresponding to $A$. Hence Theorem \ref{UniformNeronBound1} applies to points $x_1,\ldots,x_{c(k)} \in A(K)$ that lift to $R$-points of $G(R)$.

\section{Uniform Bounds for Semistable Principally Polarized Abelian Varieties} \label{sec: UniformSemistable}
	We now prove uniform bounds on the number of rational torsion points for ppavs of dimension $g$ for principally polarized abelian varieties $A/K$ with semistable reduction and having at least one place of bad reduction.  We will actually prove a more general statement under the assumption that 
    \begin{equation} \label{eq: BadPlaces1} 
    |S_0| \geq \frac{h_{\omega}(A)}{M_0},
    \end{equation}
    where $S_0$ is the set of places of bad reduction for $A$ and $h_{\omega}(A)$ is the Faltings height of $A$. Then our uniform bounds will depend on $\dim A$ and $M_0$. 
    Recall from Section \ref{sec: abelianschemes} that there is an ample line bundle $\omega$ on the minimal compactification $\cal{A}_{g,3}^*$ so that $h_{\omega}(A)$ is exactly the height of $A$ as a $K$-point on $\cal{A}_{g,3}^*$ with respect to $\omega$. 
    \par 
    Using an analogue of analogue of Szpiro's inequality by Deligne \cite[Lemma 3.2]{Del87}, we can show that $M_0$ depends only on $g$ and $g(B)$ for $A/K$ semistable, which proves uniformity for semistable abelian varieties. 
    \par 
    Our strategy follows that of Hindry--Silverman \cite{HS88} which proves both uniform boundedness and Lang's conjecture for elliptic curves over functions fields of characteristic $0$. The first step is to reduce to the case where our small points $x_i$ reduce to the identity component of the N\'{e}ron model for a uniform positive proportion of places of bad reduction. We then apply Theorems \ref{thm:basisbound} and \ref{UniformNeronBound1} to obtain a contradiction.
    \par 
    Given \eqref{eq: BadPlaces1}, we first show we can find a positive proportion of places of $S_0$, depending on $M_0$, such that the number of connected components for the N\'{e}ron model $\cal{N}$ for these places can be uniformly bounded. To do so, we will reduce to the case of Jacobians. We first give a bound for the case of Jacobians. We note that an alternative argument is possible by using \cite[Lemma 6.6]{HP16} instead. 

    \begin{proposition} \label{NodesComponents1Chapter6}
Let $C/K$ be a semistable curve of genus $g \geq 2$ and let $v \in M_K$ be a place such that the minimal regular model $\cal{C}/O_{K,v}$ has $N$ nodes. If $\cal{N}/O_{K,v}$ denotes the N\'{e}ron model of $\Jac(C)$ over $K_v$, then the number of connected components of its special fiber is bounded in terms of $N$. 
\end{proposition} 

\begin{proof}
We apply \cite[Theorem 1.3]{Lor90}. As $C/K$ is semistable, the special fiber $\cal{C}_s'$ of $\cal{C}'$ is reduced and so we satisfy the hypothesis required. As we have $N$ nodes, it follows that the number of irreducible components of $\cal{C}'_s$ is at most $N$. We may associate to our minimal regular model a type $(n,M,R,P)$ where $n\le N$ is the number of irreducible components, $M$ is the $n\times n$ intersection matrix $((C_i \cdot C_j))$ where the $C_i,C_j$ run over the irreducible components of $\cal{C}'_s$, $R = (1,1,\ldots,1)$ as $\cal{C}'_s$ is reduced, and $P = (g(C_1),\ldots,g(C_n))$. As $\cal{C}'$ is the minimal regular model, the numerical type must be a minimal numerical type. \cite[\href{https://stacks.math.columbia.edu/tag/0C9W}{Tag 0C9W}]{stacks-project} tells us that for a minimal type, $|(C_i \cdot C_j)|$ is bounded in terms of $g$ and  \cite[\href{https://stacks.math.columbia.edu/tag/0C9V}{Tag 0C9V}]{stacks-project} tells us that $g(C_i)$ are all bounded above by $g$. Hence there are only finitely many such minimal numerical types if $N$ is fixed. Then \cite[Theorem 1.3]{Lor90} tells us that we can read off the number of components from a given numerical type, and so we have a uniform bound as desired.
\end{proof}

\begin{proposition}\label{JacobianCoverFamily1Chapter6}
We can cover $\mathcal A_{g,3}$ with finitely many connected, locally closed subvarieties $Z_1,\ldots,Z_n$ such that on each $Z_i$ there exist:
\begin{itemize}
  \item a smooth proper curve $C_i/Z_i$ of genus $g_i$;
  \item an abelian scheme $\mathcal B_i/Z_i$;
  \item a finite flat subgroup scheme $G_i\subseteq \mathrm{Pic}^0_{C_i/Z_i}$;
\end{itemize}
for which the quotient morphism \[\phi_i:\ \mathrm{Pic}^0_{C_i/Z_i}\ \twoheadrightarrow\
  \mathrm{Pic}^0_{C_i/Z_i}/G_i\ \cong\ A_{Z_i}\times \mathcal B_i\] is an isogeny of abelian schemes over $Z_i$, where $A_{Z_i}$ denotes the universal abelian scheme on $\mathcal A_{g,3}$ restricted to $Z_i$.
On each $Z_i$ the genus is constant and satisfies
\[g_i=\dim\!\big(\mathrm{Pic}^0_{C_i/Z_i}\big)=g+\dim(\mathcal B_i)\ge g.\]\end{proposition}

\begin{proof} Let $\mathcal{A}\to Z$ be an abelian scheme over an irreducible quasi-projective variety $Z/\mathbb C$, and write $A_K$ for its generic fiber over $K=K(Z)$. By Poincar\'e reducibility there exists a smooth proper curve $C_K/K$ and an abelian variety $B_K/K$ together with an isogeny \[\phi_K:\ \mathrm{Pic}^0_{C_K/K}\longrightarrow A_K\times B_K,\] whose kernel $G_K\subseteq \mathrm{Pic}^0_{C_K/K}$ is finite. Choose a finitely generated $\mathbb C$-subalgebra $R\subseteq K$ over which $C_K$ and $G_K$ descend, giving a smooth proper curve $C_R\to \operatorname{Spec}\,R$ and a finite subgroup scheme $G_R\subseteq \mathrm{Pic}^0_{C_R/R}$.
	
After shrinking to a dense open $U_0\subseteq Z$ admitting a morphism $U_0\to \operatorname{Spec}R$, we obtain a smooth proper curve $C_{U_0}\to U_0$ and a finite subgroup scheme $G_{U_0}\subseteq \mathrm{Pic}^0_{C_{U_0}/U_0}$. Shrinking further if necessary, $G_{U_0}$ becomes finite flat; write $U\subseteq U_0$ for the resulting open and $G_U$ for the finite flat subgroup scheme in $\mathrm{Pic}^0_{C_U/U}$. The quotient \[Q:=\mathrm{Pic}^0_{C_U/U}/G_U\] exists as an abelian scheme over $U$, and the quotient map \[\phi_U:\ \mathrm{Pic}^0_{C_U/U}\twoheadrightarrow Q\] is a homomorphism of abelian schemes over $U$ (the structural maps to $U$ commute). By construction the generic fiber $Q_K$ is isomorphic to $A_K\times B_K$. Since the functor of homomorphisms between abelian schemes is represented by an \'etale scheme over the base, this generic isomorphism extends, after possibly shrinking $U$ once more, to an isomorphism of abelian schemes \[Q\ \cong\ \mathcal{A}_U\times \mathcal B_U,\] where $\mathcal A_U$ is $\mathcal A$ restricted to $U$ and $\mathcal B_U/U$ is an abelian scheme extending $B_K$. Thus \[\phi_U:\ \mathrm{Pic}^0_{C_U/U}\longrightarrow \mathcal{A}_U\times \mathcal B_U\] is an isogeny of abelian schemes over $U$. Replacing $U$ by a connected component if necessary, the relative dimension is constant, and we have \[g_U:=\dim\!\big(\mathrm{Pic}^0_{C_U/U}\big)=\dim(\mathcal A_U)+\dim(\mathcal B_U)=g+\dim(\mathcal B_U),\] so the fibers of $C_U/U$ have constant genus $g_U\ge g$ on this connected $U$.
	
Applying the foregoing construction to $Z=\mathcal A_{g,3}$ with $\mathcal A$ the universal abelian scheme, we obtain a nonempty open subset $U\subseteq\mathcal A_{g,3}$ and, after replacing $U$ by its finitely many connected components, connected locally closed subsets $Z_1,\dots$ on each of which we have a smooth proper curve $C_i/Z_i$, an abelian scheme $\mathcal B_i/Z_i$, and an isogeny \[\phi_i:\mathrm{Pic}^0_{C_i/Z_i}\longrightarrow A_{Z_i}\times\mathcal{B}_i\quad\text{over }Z_i.\] On the closed complement $\mathcal{A}_{g,3}\setminus\bigcup Z_i$, repeat the same construction on each irreducible component and then pass to connected components of the resulting opens. By Noetherian induction (each step removes a nonempty open from a closed subset and yields only finitely many connected components), this recursive procedure terminates after finitely many steps, giving a finite cover \[\mathcal{A}_{g,3}=\bigcup_i Z_i\] by connected locally closed subvarieties possessing such data.\end{proof}

We now obtain the following proposition. 

\begin{proposition} \label{IdentityComponentAbelianVarieties1}
Let $A$ be a principally polarized abelian variety over a function field $K$ with $K = \bb{C}(B)$ and with semistable reduction. Let $M_0 > 0$ be a constant and let $S_0$ be the places of bad reduction for $A$. Assume that 
$$|S_0| \geq \frac{h_{\omega}(A)}{M_0}.$$
Then there exist constants $M,k > 0$ depending only on $g:= \dim A$ and $M_0$, such that there is a subset $S \subseteq M_K$ satisfying 
\begin{enumerate}
			\item $A$ has bad reduction for each $v \in S$;
			\item $|S| \geq \frac{1}{M} h_{\omega}(A)$;
			\item For any $x \in A(K)$, the point $[k]x$ lies in the identity component of the N\'eron model of the special fiber at each $v\in S$. 
		\end{enumerate}  
\end{proposition}

\begin{proof}
On each quasi-projective variety $Z_i$ from Proposition \ref{JacobianCoverFamily1Chapter6}, the abelian scheme $\Pic^0_{C_i/Z_i}$ induces a map $f_i: Z_i \xhookrightarrow{} \cal{M}_{g_i}$. Under this morphism $f_i$, by Proposition \ref{GlobalHeightAmple2}, we have a constant $c_i$ such that $h_{\lambda}(f(x)) \leq c_i h_{\omega}(x)$ for all $x \in Z_i(K)$ where $K$ is a global function field over $\bb{C}$. Here $\lambda$ is the Hodge $\bb{Q}$-line bundle on the coarse moduli space $\ovl{M}_g$ (see Section \ref{sec: ModuliCurves}).
	\par 
	Now let $A/K$ be our principally polarized abelian variety. By assumption, we have $|S_0| \geq \frac{1}{M'} h_{\omega}(A)$ where $T$ is the set of bad places for $A/K$. Let $K'$ be a finite extension so that $A$ has level $3$ structure with $[K':K]$ uniformly bounded in terms of $g$. Using $K'$ as our reference field for global heights, we have $|S_0| \geq \frac{1}{M'[K':K]} h_{\omega}(A)$. Then the $K'$-point representing $A$ on $\cal{A}_{g,3}(K')$ must factor through one of the $Z_i$'s of Proposition \ref{JacobianCoverFamily1Chapter6} and hence via pullback along this point we obtain a smooth proper curve $C/K'$ such that there is an isogeny $\pi: \Jac(C) \to A \times A'$, where $A'$ is some abelian variety and  $h_{\lambda}(C) \leq c h_{\omega}(A)$, with $\deg(\pi)$ and $c$ uniformly bounded by Proposition \ref{GlobalHeightAmple2} since $\omega$ is an ample height. Furthermore, $g(C)$ is bounded above only in terms of $g$.
	\par 
	Since $h_{\lambda}(C) \leq c h_{\omega}(A)$, by increasing $M'$ we also obtain 
\begin{equation} \label{eq: BadReductionBound1}    |S_0|\ge\frac{1}{M'[K':K]} h_{\lambda}(C)
\end{equation}
    since our constants are allowed to depend on $g$. We replace $K'$ with a larger extension so that $C/K'$ has semistable reduction. We still have at least $\frac{1}{M'[K':K]} h_{\lambda}(C)$ places of bad reduction for $A/K'$ where again $[K':K]$ is bounded purely in terms of $g$.
	\par 
	Let $S_0'$ be the places of bad reduction of $A/K'$, so that $|S_0'| \geq |T|$. Since $h_{\lambda}(C) \geq \frac{N}{12}$ by \eqref{eq: CurveHeightNode}, where $N$ is the total number of nodes of the minimal regular model of $C$ over $B'$ for $B'$ the curve corresponding to $K'$, it follows from \eqref{eq: BadReductionBound1} for at least half of these places, the number of nodes is at most $24 M' [K':K]$ which depends only on $g$ and $M'$. Let $S'$ be these subset of $S_0'$. As $g(C)$ may be bounded above in terms of $g$, it follows from Proposition \ref{NodesComponents1Chapter6} that for any $v\in S'$, the number of connected components of the special fiber of the N\'{e}ron model of $\Jac(C)$ is at most $E$, for some constant $E$ depending only on $g$ and $g(B)$.

Let $x \in A(K)$. Then since $\deg(\pi)$ is bounded solely in terms of $g$, there exists a finite extension $L/K'$ with $[L:K']$ bounded depending only on $g$ such that there is a preimage $y \in \Jac(C)(L)$ mapping to $x$ under $\pi_1$, the projection to the first factor. Fix $w\in M_L$ lying over $v\in S'$, and set $J:=\Jac(C)$ and $g':=\dim J$. Write \[\Phi_{J/K',v}:=\text{component group of the N\'eron model of } J \text{ at } v,\]
\[\Phi_{J/L,w} := \text{component group of the Néron model of } J \text{ at } w.\] By \cite[Thm.~5.7]{HN10} one has \[\#\Phi_{J/L,w}\le[L:K']^{g'}\cdot\#\Phi_{J/K',v}.\]
Using $[L:K']\le n_i$ and the uniform bound $\#\Phi_{J/K',v} \le E$ for $v\in S'$ from above, we obtain \[\#\Phi_{J/L,w} \le n_i^{g'} E\] and 
hence \[\exp(\Phi_{J/L,w}) \le n_i^{g'}E.\]
Let $k:=n_i^{g'}E$. Then $\exp(\Phi_{J/L,w})\mid k!$, so for the lift $y\in J(L)$ with $\pi_1(y)=x$ we have $[k!]y$ reducing to the identity component of the N\'{e}ron model of $J$. at $w$. Hence $[k!]x=\pi_1([k!]y)$ must lie in the identity component of $A$ over $w$. We have thus proven our theorem by letting $S$ be the set of places of $K$ that lie below those of $S'$. Then $$|S|\geq\frac{|S'|}{[K':K]}\geq\frac{|S_0|}{2[K':K]} \geq \frac{1}{2[K':K]^2 M'} h_{\omega}(A)$$ and hence we may take $M=2[K':K]^2 M'$. \end{proof}
	
Next, we cover $\cal{A}_{g,3}$ with affine opens $\{U_i\}$ such that on each $U_i$, we have an embedding $\cal{A}_{U_i} \xhookrightarrow{} \bb{P}^{N_i}_{U_i}$ along with an extension of the endomorphism $[m_i]$ for some natural number $m_i$.  

\begin{proposition} \label{ModuliOpenCover1Chapter6} We may cover $\cal{A}_{g,3}$ with finitely many affine opens $U_1,\ldots,U_m$ such that for each $i$, there exists an embedding $\cal{A}_{U_i} \xrightarrow{} \bb{P}^{N_i}_{U_i}$ and a natural number $m_i \geq 2$ such that $[m_i]$ extends to a morphism $f_i:\bb{P}_{U_i}^{N_i} \to \bb{P}_{U_i}^{N_i}$ of degree $m_i^2$, where $\cal{A}_{U_i}$ is the universal abelian scheme over $U_i$ and $\bb{P}^{N_i}_{U_i}$ is projective space over $U_i$. 
\end{proposition}

To prove Proposition \ref{ModuliOpenCover1Chapter6}, we prove the following proposition first.

\begin{proposition}\label{prop:MorphismExtension1Chapter6} Let $S$ be a noetherian affine scheme and $\mathcal{A}/S$ an abelian scheme of relative dimension $g$ with a symmetric, rigidified, very ample line bundle $\mathcal{L}$ inducing an embedding $\mathcal{A}\hookrightarrow \mathbb{P}^N_S$. Fix $s\in S$ and assume that \[\mathcal{A}_s\cap\{x_0=\cdots=x_g=0\}=\emptyset.\] Suppose moreover that $\mathcal{A}$ is cut out in $\mathbb{P}^N_S$ by homogeneous polynomials of degree at most $n^2$. Then there exists a Zariski neighborhood $U\subseteq S$ of $s$ such that the multiplication-by-$n$ morphism $[n]:\mathcal{A}_U\to\mathcal{A}_U$ extends to an endomorphism of $\mathbb{P}^N_U$ of degree $n^2$.
\end{proposition}

\begin{proof} Let $s_0,\dots,s_N$ be the sections of $\mathcal{L}$ defining the embedding $\iota:\mathcal{A}\hookrightarrow\mathbb{P}^N_S$. Let $n\ge1$. Because $\mathcal{L}$ is symmetric, the theorem of the cube gives an isomorphism \[[n]^*\mathcal{L}\cong\mathcal{L}^{\otimes n^2}\otimes(e^*\mathcal{L})^{\otimes\binom{n}{2}},\] where $e:S\to\mathcal{A}$ is the identity section. As $\mathcal{L}$ is rigidified, $e^*\mathcal{L}\cong\mathcal{O}_S$, so we obtain a canonical isomorphism $[n]^*\mathcal{L}\cong\mathcal{L}^{\otimes n^2}$. In particular, for each $j$ the pullback $[n]^*s_j$ is a section of $\mathcal{L}^{\otimes n^2}$. Since $H^0(\mathcal{A},\mathcal{L}^{\otimes n^2})$ is generated by degree-$n^2$ monomials in the $s_i$, we may choose a homogeneous polynomial $f_j(s_0,\dots,s_N)$ of degree $n^2$ representing $[n]^*s_j$. These $f_j$ are unique only modulo the homogeneous ideal of $\mathcal{A}$.
	
The tuple $(f_0,\dots,f_N)$ defines a rational map $F:\mathbb{P}^N_S\dashrightarrow\mathbb{P}^N_S$, and by construction $F$ restricts to multiplication by $n$ on $\mathcal{A}$. Indeed, for $T$ an $S$-scheme and $a\in\mathcal{A}(T)$, the image of $a$ under the embedding $\iota:\mathcal{A}\hookrightarrow\mathbb{P}^N_S$ is 
\[\iota(a)=[s_0(a):\cdots:s_N(a)]\in\mathbb{P}^N(T).\] 
Write $[n]:\mathcal{A}\to\mathcal{A}$ for the multiplication-by-$n$ morphism. Then 
\[\iota([n](a))=[s_0([n](a)):\cdots:s_N([n](a))].\] By definition of pullback of sections, 
\[s_j([n](a))=([n]^*s_j)(a)=f_j(s_0(a),\dots,s_N(a)).\] Thus 
\[\iota([n](a))=(f_0(s_0(a),\dots,s_N(a)),\dots,f_N(s_0(a),\dots,s_N(a))),\] as claimed.

To extend $F$ to a morphism on some neighborhood $U\subseteq S$ of $s$, we must arrange that the $f_j$ have no common zeroes on $\mathbb{P}^N_{U}$. The hypothesis \[\mathcal{A}_s\cap\{x_0=\cdots=x_g=0\}=\emptyset\] says that $s_0,\dots,s_g$ have no common zero on $\mathcal{A}_s$. Passing to the embedding associated to $\mathcal{L}^{\otimes n^2}$ if necessary, the pullback sections $[n]^*s_0,\dots,[n]^*s_N$ provide the projective coordinates in this embedding. Each $[n]^*s_j$ can be represented by a homogeneous polynomial $f_j(s_0,\dots,s_N)$ of degree $n^2$, unique only modulo the homogeneous ideal of $\mathcal{A}$. By choosing these representatives $f_0,\dots,f_g$ suitably, we may ensure that their common zero locus in $\mathbb{P}^N_s$, where $\bb{P}^N_s$ is the fiber above $s$, has codimension $g+1$. This furnishes the base case of the induction.

We now follow Fakhruddin's argument in \cite[Proposition~2.1]{Fak03}. Suppose inductively that for some $i\geq 0$ the common vanishing locus \[Y_s=\{f_0=\cdots=f_{g+i}=0\}\subseteq\mathbb{P}^N_s\] has codimension $g+i+1$. To define $f_{g+i+1}$, choose a homogeneous polynomial $h_1$ of degree $n^2$ restricting to $[n]^*s_{g+i+1}$ on $\mathcal{A}$, and choose another homogeneous polynomial $h_2$ of degree $n^2$ vanishing on $\mathcal{A}$ but not identically on any irreducible component of $Y_s$ (this is possible since $Y_s\cap\mathcal{A}_s=\varnothing$). For general $b$, the linear combination $f_{g+i+1}=h_1+b\,h_2$ still restricts to $[n]^*s_{g+i+1}$ on $\mathcal{A}$ and avoids vanishing identically on any component of $Y_s$, so the new common vanishing locus \[\{f_0=\cdots=f_{g+i+1}=0\}\] has codimension $g+i+2$. 
	
Iterating, we obtain $f_0,\dots,f_N$ whose common zero locus in $\mathbb{P}^N_s$ is empty. Equivalently, the resultant $\operatorname{Res}(f_0,\dots,f_N)$ is nonzero in the residue field of $s$, hence invertible on some Zariski open neighborhood $U$ of $s$. Over $U$, the polynomials $f_0,\dots,f_N$ therefore define a genuine morphism \[F:\mathbb{P}^N_U\to\mathbb{P}^N_U\] of degree $n^2$ whose restriction to $\mathcal{A}_U$ coincides with $[n]$.\end{proof}

		\begin{proof}[Proof of Proposition \ref{ModuliOpenCover1Chapter6}] For each $x$, let $U_x$ be an affine open containing $x$. If $\cal{L}$ is the universal line bundle on $\cal{A}_{g,3}$, then as $\cal{L}^3$ is very ample we can embed $\cal{A}_{U_x} \xhookrightarrow{} \bb{P}^N_{U_x}$ as a closed subscheme. By Bertini we can arrange the embedding so that $\cal{A}_x \cap \{x_0 = \cdots = x_g = 0\}$ is empty. By noetherianess, there is some $n_x > 0$ for which it is cut out by homogeneous polynomials of degree at most $n_x^2$. Proposition \ref{prop:MorphismExtension1Chapter6} implies that after shrinking $U_x$, the morphism $[n_x]$ extends to $\bb{P}^{N}_{U_x}$ and furthermore after picking some homogeneous lift $F$ of the morphism $f$, we may assume that $F$ has an invertible coefficient too. We can then multiply $F$ by the inverse of this coefficient and assume that it is $1$. Now by compactness of $\mathcal{A}_{g,3}$, we can find finitely many such affine opens as desired.\end{proof}
		
        Given $g\ge1$, we fix once and for all a finite cover $\{U_i\}$ by open affines $U_i$ of $\mathcal{A}_{g,3}$ as in Proposition \ref{ModuliOpenCover1Chapter6}, as well as the associated integers $m_i, N_i$ that arise as the extensions of $[m_i]$ to $\mathbb{P}_{U_i}^{N_i}$. We thus consider all dependencies on these $U_i,m_i,N_i$ to constitute a dependence only on $g$. Observe that on each $U_i$ we obtain a morphism $\vphi_i: U_i \to \Rat^{N_i}_{m_i^2}$ extending $[m_i]$, and as both $U_i$ and $\Rat^{N_i}_{m_i^2}$ are affine, it follows from Proposition \ref{GlobalHeightAmple2} that there exists a $c_i>0$ such that 
    \begin{equation} \label{eq : HeightExtension1} 
    c_i h_{\partial U_i}(A) \geq h(\vphi_i(A))
    \end{equation}
    where $h$ is the Weil height on $\Rat^{N_i}_{m_i^2}$. Indeed, as $U_i$ is affine, we may choose a compactification $\ovl{U}_i$ such that $\partial U_i = \ovl{U}_i\setminus U_i$ is ample. To do so, we embed $U_i$ into $\bb{A}^k$ for some $k > 0$ and then choose its closure $\ovl{U}_i \subseteq \bb{P}^k$ as the compactification. Since $\bb{P}^k \setminus \bb{A}^k$ is an ample divisor on $\bb{P}^k$, so is its restriction $\overline{U}_i\setminus U_i$ to $\ovl{U}_i$. 
    
    On the other hand, we may choose compactifications $U_i^*$ in $\cal{A}_{g,3}^*=\ovl{X}$. By Proposition \ref{LocalHeightBoundary1}, there is some $\eta\ge1$ such that \[\eta^{-1}\lambda_{\overline{U}_i\setminus U_i}\le\lambda_{U_i^*\setminus U_i}\le\eta\lambda_{\overline{U}_i\setminus U_i}.\] We may therefore without loss set $\partial U_i$ to be $U_i^*\setminus U_i$ and assume that \eqref{eq : HeightExtension1} holds for $h_{\partial U_i}$.

Since $\{U_i\}$ is a finite cover of $X = \cal{A}_{g,3}$, for any compactification $\ovl{X}$ of $X$ we have 
$$\partial U_1 \cap \partial U_2 \cap \cdots \cap \partial U_m = \partial X$$
and thus by \eqref{eqn:intht},
	$$\min\{\lambda_{\partial U_1},\ldots, \lambda_{\partial U_m} \}=\lambda_{\partial X}.$$
    	Hence for any $K$-point $A$ of $X$ where $K$ is a complete algebraically closed field with $\bb{C}$ a trivially valued subfield, there exists some $i$ such that 
	\begin{equation}\label{eqn:centeredi}\lambda_{\partial U_i}(A) \leq\lambda_{\partial X}(A).\end{equation}

We are now ready to prove the following theorem. 
	
	\begin{theorem}\label{UniformSemistable1} Let $K=\mathbb{C}(B)$ be the function field of a smooth projective curve $B$ over $\bb{C}$, and let $A/K$ be a principally polarized abelian variety with semistable reduction. Assume there exists a constant $M_0 > 0$ such that 
    $$|S_0| \geq \frac{h_{\omega}(A)}{M_0}$$
    where $S_0$ is the set of places of bad reduction for $A$. Then there are positive constants  $n = n(\mathrm{dim}(A),M_0)$ and $\eps = \eps(\dim A, M_0)$ such that if $h_{\omega}(A) > 0$, then 
    $$\{x \in A(K) \mid \h_A(x) \leq \eps h_{\Fal}(A) \}$$
    is contained in $\divv(s)$ for some nonzero $s \in H^0(A,L^n)$, where $L$ is the symmetric ample line bundle inducing the principal polarization on $A$ and $\h_A$ denotes the N\'eron--Tate height with respect to $L$. 
    \end{theorem} 

    \begin{proof}
        We may pass to an extension $K'/K$ with $[K':K]$ bounded in terms of $g$ so that $A$ attains level $3$ structure. Let $S$ be the set of places provided by Proposition \ref{IdentityComponentAbelianVarieties1} and let $S'$ be the set of places in $K'$ that is over $S$. Then for all $v \in S$, $A$ has bad reduction at $v$, and we have $M,k > 0$, depending only on $M_0$ and $\dim A$, such that
        $$|S| \geq \frac{1}{M} h_{\omega}(A)$$
         and $[k]x$ reduces to the identity component of the N\'{e}ron model over $v$ for all $x \in A(K)$. 
        \par 
        We now consider all global heights with respect to $K'$, so that $h_{\omega}(A)$ is scaled by a factor of $[K':K]$. Then we have the inequality \begin{equation}\label{eqn:S'cardbd}|S'| \geq \frac{1}{M[K':K]} h_{\omega}(A).\end{equation} For each $v \in K'$, let $x_v$ be the induced point on $\cal{A}_{g,3}(K_v')$ by $A$. For each $v \in S'$, we see from \eqref{eqn:centeredi} that there exists $i_v$ such that $x_v \in U_{i_v}$ and
	\begin{equation}\label{eqn:centered}\lambda_{\partial U_{i_v}}(x_v) \leq c \lambda_{\partial X}(x_v).\end{equation} Since there are $m$ such opens, there exists a subset $\Sigma \subseteq S'$ consisting of at least $\frac{|S'|}{m}$ places such that they share the same such $i=i_v$, say $i=1$. 
    \par 
    Note that the $K'$-point of $\mathcal{A}_{g,3}$ given by $A$ lies in $U_1$. By Proposition \ref{ModuliOpenCover1Chapter6}, we have an embedding $\cal{A}_{U_1} \xhookrightarrow{} \bb{P}_{U_1}^N$ with an endomorphism $f: \bb{P}_{U_1}^N \to \bb{P}_{U_1}^N$ extending $[m_1]$ for some $m_1\geq2$, along with a fixed homogeneous lift $F: \bb{A}_{U_1}^{N+1} \to \bb{A}_{U_1}^{N+1}$ with one of its coefficients being $1$. For each $n\ge1$, fix a good basis $\cal{B}_n\subseteq H^0(A,L^n)$ over $K'$ with respect to this lift $F$. Then since one of the coefficients of $F$ is equal to $1$, we have \eqref{eqn:Fnormbd} and therefore \eqref{eq:ElkiesLower1}. Since by \eqref{eq : HeightExtension1} there exists a $c_1 > 0$ such that 
    $$c_1 h_{\partial U_1}(A) \geq h(f),$$
    \eqref{eq:ElkiesLower1} gives us a constant $C > 0$ (depending only on $g$) such that for any $T \subseteq M_{K'}$ and any $x_1,\ldots,x_{c(n)}\in A(K')$, we have 
	\begin{equation} \label{eq:ElkiesLower2}
    \sum_{v \in T} g_{\cal{B}_n,v}(x_1,\ldots,x_{c(n)}) \geq -\frac{C \log n}{n} h_{\partial U_1}(A) + \sum_{v \in T} r_v(F)
    \end{equation} 
    where $g_{\cal{B}_n,v}$ is the Arakelov-Green's function over the non-archimedean field $K_v$, and $r_v(F)$ equals \eqref{eqn:rF} with absolute value $|\cdot|_v$. We will choose $T$ later. Let $h: \bb{N} \to \bb{N}$ be an increasing function such that for all $n\in\mathbb{N}$, $h(n) = c(k)$ for some $k$; we fix the properties of $h$ later. Then by (\ref{eqn:centered}) along with Theorem \ref{UniformNeronBound1}, there exists a nonempty finite list $\{n_1,\ldots,n_{m'}\}$ such that any $A/K'_v$ with bad reduction is $n_i$-good for some $1\le i\le m'$. If we write $h(n) = c(k_n)$ and choose our function $h$ so that
	$$\sum_{i=1}^{\infty} \frac{1}{k_i^{1/2}} < 1, $$
    then in particular
    \begin{equation}\label{eqn:knibd}\sum_{i=1}^{m'} \frac{1}{k_{n_i}^{1/2}} < 1.\end{equation}

	We now claim that there is an $1\le i\le m'$ such that for at least $\frac{1}{k_{n_i}^{1/2}}|\Sigma|$ places $v\in\Sigma$, the abelian variety $A/K_v'$ is $n_i$-good. Otherwise, for each $1\le i\le m'$, there are at most $\frac{|\Sigma|}{k_{n_i}^{1/2}}$ places of $\Sigma$ such that $A$ is $n_i$-good. Since for every $v\in\Sigma$, we have that $A/K'$ is $n_i$ good for some $1 \leq i \leq m'$, it follows that
    $$\sum_{i=1}^{m'} \frac{|\Sigma|}{k_{n_i}^{1/2}} \geq |\Sigma|,$$ and hence $$\sum_{i=1}^{m'} \frac{1}{k_{n_i}^{1/2}} \geq 1,$$
    which contradicts \eqref{eqn:knibd}. Thus there is an $1\le i\le m'$ such that for at least $\frac{1}{k_{n_i}^{1/2}}|\Sigma|$ places $v\in\Sigma$, the abelian variety $A/K_v'$ is $n_i$-good. We assume without loss that $i=1$, so that at least $\frac{|\Sigma|}{k_{n_1}^{1/2}}$ of the places of $\Sigma$ are $n_1$-good. Let $\Sigma'$ be this subset of places, which by assumption satisfies \begin{equation}\label{eqn:sigma'}|\Sigma'| \geq \frac{1}{k_{n_1}^{1/2}} |\Sigma|.\end{equation}  
    \par 
    Now assume we are given $x_1,\ldots,x_{c(k_{n_1})} \in A(K')$ that all reduce to the identity component of the N\'eron model for all $v \in \Sigma'$. As $A/K_v$ is $n_1$-good, we obtain 
    \begin{equation} \label{eq: LowerBoundArakelov1}
    g_{\cal{B}_{k_{n_1}},v}(x_1,\ldots,x_{c(k_{n_1})}) \geq \frac{1}{n_1} \lambda_{\partial U_1}(A_v) + r_v(F) \geq \frac{1}{n_1} + r(F_v)
    \end{equation}
    since $\lambda_{\partial U_1}(A_v) \geq \lambda_{\partial \cal{A}_{g,3}} \geq 1$. Here, the second equality follows from the fact that $\lambda_{\partial \cal{A}_{g,3}}(A_v) > 0$ due to having bad reduction and that $\lambda_{\partial \cal{A}_{g,3}}: \cal{A}_{g,3}(K'_v) \to \bb{R}_{\geq 0}$ is valued in non-negative integers as we take $K'$ to be our base field. Recalling that $\sum_{v \in M_{K'}} r_v(F) = 0$, we now sum over all places in $M_{K'}$ using \eqref{eq:ElkiesLower2} and \eqref{eq: LowerBoundArakelov1}. Letting $T=M_{K'} \setminus \Sigma'$ in \eqref{eq:ElkiesLower2}, we obtain 
	\begin{equation}\label{eqn:lb1}\sum_{v \in M_{K'}} g_{\cal{B}_{k_{n_1}},v}(x_1,\ldots,x_{c(k_{n_1})})\geq\frac{1}{n_1} |\Sigma'|-\frac{C \log k_{n_1}}{k_{n_1}} h_{\partial U_1}(A).\end{equation} By \eqref{eqn:S'cardbd} and the fact that $\omega$ is ample on the minimal compactification $\mathcal{A}_{g,3}^*$ combined with Proposition \ref{GlobalHeightAmple1}, we have \begin{equation}\label{eqn:sigmalb}|\Sigma|\geq\frac{1}{mM[K':K]}h_{\omega}(A)\geq\frac{1}{M'}h_{\partial U_1}(A)\end{equation} for some constant $M'$ depending only on $g$ and $g(B)$. Combining \eqref{eqn:lb1} and \eqref{eqn:sigmalb} with \eqref{eqn:sigma'}, we get
	\begin{equation*}\sum_{v \in M_{K'}} g_{\cal{B}_{k_{n_1}},v}(x_1,\ldots,x_{c(k_{n_1})})\ge\left(\frac{1}{M' n_1 k_{n_1}^{1/2}}-\frac{ C \log k_{n_1}}{k_{n_1}} \right) h_{\partial U_1}(A).\end{equation*}
	Now note that $M'$ and $C$ depend only on $g$ and $g(B)$. Hence there are constants $C'$ and $\delta>0$ depending only on $g$ and $g(B)$ such that if $2C' n_1^3 \geq k_{n_1} \geq C' n_1^3$, then
	\begin{equation}\label{eqn:deltalb}\sum_{v \in M_{K'}} g_{\cal{B}_{k_{n_1}},v}(x_1,\ldots,x_{c(k_{n_1})})\geq\delta h_{\partial U_1}(A).\end{equation} We now choose $k_i$ for our function $h$ that satisfy 
    $$2C'i^3 \geq k_i \geq C'i^3$$ for all $i\in\mathbb{N}$, so that in particular \eqref{eqn:deltalb} always holds. Suppose $\h_A(x_i)\leq\eps h_{\Fal}(A)$. Then unless $g_{\cal{B}_{k_{n_1}},v}(x_1,\ldots,x_{c(k_{n_1})})=+\infty$, we must have 
    \begin{equation}\label{eqn:guppbd}\sum_{v \in M_{K'}} g_{\cal{B}_{k_{n_1}},v}(x_1,\ldots,x_{c(k_{n_1})}) \leq \eps h_{\Fal}(A).\end{equation}

On the other hand, by Proposition \ref{GlobalHeightAmple2}, since $h_{\Fal}(A)=h_{\omega}(A)$ and $\omega$ is an ample line bundle on $\cal{A}_{g,3}^*$, there is some $c' > 0$ depending only on $g$ such that $h_{\partial U_1}(A) \geq c' h_{\Fal}(A)$. Hence \eqref{eqn:guppbd} yields a contradiction with \eqref{eqn:deltalb} for all $\eps\ll_{g,g(B)}1$. Thus $g_{\cal{B}_{k_{n_1}},v}(x_1,\ldots,x_{c(k_{n_1})})=\infty$. This means that there must exist some nonzero $s\in H^0(A,L^{k_{n_1}})$ such that $x_1,\dots,x_{c(k_{n_1})}$ all lie on $\divv(s)$. Therefore, for any set of $c(k_{n_1})$ torsion points in $A(K')$ that reduce to the identity component for all $v \in \Sigma'$, there must be a nonzero section $s\in H^0(A,L^{k_{n_1}})$ such that these points all lie on $\div(s)$.
    \par 
    We conclude that for any set of points $x_1,\dots,x_{N_0}$ (independent of $N_0$) defined over $K'$ of N\'eron--Tate height at most $\eps h_{\Fal}(A)$ which all reduce to the identity component of the N\'eron model for all $v\in S'$, there must exist a nonzero $s\in H^0(A,L^{k_{n_1}})$ such that the $x_i$ all lie on $\divv(s)$. Indeed, otherwise the rectangular matrix formed by $\det(\eta_j(\tilde{x}_i))_{\eta_j\in\mathcal{B}_{k_{n_1}}}$ has rank $c(k_{n_1})$. If $N_0 < c(k_{n_1})$ then this clearly cannot happen. If $N_0 \geq c(k_{n_1})$, then it follows that we can find a $c(k_{n_1})\times c(k_{n_1})$ minor of maximal rank and thus that there does not exist any nonzero $s \in H^0(A,L^{k_{n_1}})$ such that the $x_i$ all lie on $\divv(s)$, contradicting the preceding argument.
     \par 
     Each place $v \in \Sigma'$ lies above some $v \in S$, where $S$ is as in Proposition \ref{IdentityComponentAbelianVarieties1}. In particular if $x\in A(K)$, then $[k]x$ reduces to the identity component for all $v \in S$. Hence given any set of points $x_1,\ldots,x_{c(N_0)}\in A(K)$ with $\h_A(x_i) \leq \frac{\eps}{k^2} h_{\Fal}(A)$, there must be a nonzero $s\in H^0(A,L^{k_{n_1}})$ such that $[k]x_1,\ldots,[k]x_{N_0}$ all lie on $\divv(s)$, and hence $x_1,\ldots,x_{N_0}$ all lie on $[k]^*(\divv(s))\in H^0(A,L^{k_{n_1} k^2})$. Our theorem is thus proven by taking $n=k_{n_1} k^2$. 
    \end{proof}

We now note that Theorem \ref{UniformSemistable1} is enough to imply the case of semistable abelian varieties with at least one place of bad reduction, where the constants only depend on the dimension $g$ and the genus $g(B)$. 

	\begin{corollary}\label{UniformSemistable2} Let $K=\mathbb{C}(B)$ be the function field of a smooth projective curve $B$ over $\bb{C}$, and let $A/K$ be a principally polarized abelian variety with semistable reduction.  Then there are positive constants  $n = n(\mathrm{dim}(A),g(B))$ and $\eps = \eps(\dim A, g(B))$ such that if $A$ has at least one place of bad reduction, then 
    $$\{x \in A(K) \mid \h_A(x) \leq \eps h_{\Fal}(A) \}$$
    is contained in $\divv(s)$ for some nonzero section $s \in H^0(A,L^n)$, where $L$ is the symmetric ample line bundle inducing the principal polarization and $\h_A$ denotes the N\'eron--Tate height with respect to $L$.
    \end{corollary} 

    \begin{proof}
     Since $A$ has at least one place of bad reduction, we have $h_{\omega}(A) > 0$. We now have to show that there is a constant $M_0$, depending only on $g$ and $g(B)$, such that 
     $$|S_0| \geq \frac{h_{\omega}(A)}{M_0}$$
     where $S_0$ is the set of places of bad reduction for $A$. By \cite[Lemma 3.2]{Del87}, as $A/K$ has semistable reduction, we know that 
     \begin{equation}\label{eqn:Deligne}h_{\omega}(A) \leq \frac{g}{2}(2g(B)-2 + |S_0|)\end{equation}
     if $2g(B)-2 + |S_0| \geq 0$. This gives us our desired constant $M_0$ if $g(B) \geq 1$. If $g(B) = 0$, then the only case in which we are not done is if $|S_0| = 1$. But as $\cal{A}_g$ is hyperbolic, any map from $\bb{C}$ to it must be a constant map and so $|S_0| = 0$ in this case and we do not have any place of bad reduction, which contradicts our hypotheses.
    \end{proof}

\section{N\'{e}ron Models for Potential Good Reduction} \label{sec: NeronModels}
In this section, we study the N\'{e}ron model $\cal{N}$ for an abelian variety $A/K$ with potential good reduction. If $K'/K$ is a finite extension for which $A/K'$ has good reduction with abelian scheme $\cal{A}'/R'$, then we are interested in bounding the degree of the image of the special fiber of $\cal{N}$ when mapped to $\cal{A}'$ using the universal property of N\'{e}ron models. This will essentially follow from results of Edixhoven \cite{Edi92} which are reinterpreted in Halle--Nicaise \cite{HN11}. 
\par 
Let $K$ be a non-archimedean field with a discrete valuation and having residue characteristic zero. Let $A/K$ be an abelian variety with potential good reduction.  This means that over some finite extension $K'/K$, the abelian variety $A_{K'}/K'$ extends to an abelian scheme $\cal{A}'/R'$ where $R'$ is the valuation ring of $K'$. In fact we may take $K' = K(A[3])$. 
\par 
If $R$ is the valuation ring of $R'$, then we still have a N\'{e}ron model $\cal{N}/R$ for $A/K$. Then the base change $\cal{N}_{R'}$ is a smooth model for $A_{K'}$ and hence by the universal property of N\'{e}ron models, there must be a $R'$-morphism $f: \cal{N}_{R'} \to \cal{A}'$. We now assume that the residue field of $K$ is algebraically closed, so that $K'/K$ is totally ramified and $K'$ and $K$ have isomorphic residue fields. In particular, passing to special fibers we obtain a morphism 
$$f_s : \cal{N}_s \to \cal{A}'_s.$$
We are interested in understanding the image $f(\cal{N}_s)$ inside $\cal{A}'_s$. To do so, we will apply \cite[Lemma 3.2]{HN11}. We summarize the setup below. 
\par 
Since $K'/K$ is totally ramified, if $d = [K':K]$ then we may identify $\Gal(K'/K)$ with $\mu_d$, the group of $d$th roots of unity. By the N\'eron mapping property, the action of $\mu_d$ on $A'/K'$ extends to the abelian scheme $\cal{A}'/R'$; this action is equivariant with respect to $\cal{A}'\to\Spec R'$ and hence induces an action on the special fiber $\cal{A}'_s$. 

\begin{proposition} \label{NeronModel1}
The morphism $f_s:\cal{N}_s\to\cal{A}'_s$ factors through the fixed part $(\cal{A}'_s)^{\mu_d}$ and is surjective onto $(\cal{A}'_s)^{\mu_d}$.  
\end{proposition}

\begin{proof}
This is \cite[Lemma 3.2 (2)]{HN11}.     
\end{proof}

Now let's assume that $A$ has bad reduction over $K$. This means that the special fiber $\cal{N}_s$ of the N\'{e}ron model is not an abelian variety. The identity component $(\cal{N}_s)^0$ factors as 
$$1 \to U \to (\cal{N}_s)^0 \to C \to 1$$
where $U$ is a commutative unipotent group and $C$ an abelian variety of dimension $< \dim A$. As any morphism from $U$ to $\cal{A}'_s$ must be trivial, the map $f_s\mid_{(\cal{N}_s)^0}:(\cal{N}_s)^0\to\cal{A}_s'$ must factor through $C$. As the image of $C$ under the induced map to $\cal{A}_s'$ must be a proper abelian subvariety of $\cal{A}_s'$, it follows that $f_s(\cal{N}_s)$ is a finite union of  translates of $B$.
\par 
We seek to bound the degree of $(\cal{A}'_s)^{\mu_d}$. Fix a line bundle as follows. Let $L$ be a line bundle for $A/K$ which is naturally a line bundle for $A_{K'}/K'$. This extends to a unique line bundle $\cal{L}$ for the abelian scheme $\cal{A}'/R'$ as $R'$ is a DVR and $\Pic(R') = 0$ \cite[Theorem 8.4.3]{BLR90}. We thus get a line bundle $\cal{L}_s$ on $\cal{A}'_s$. 

\begin{proposition} \label{NeronModel2}
For $g \in \mu_d$, we have an isomorphism $g^* \cal{L}_s \simeq \cal{L}_s$.
\end{proposition}

\begin{proof}
We have $g^*(\cal{L}_s) = (g^* \cal{L})_s$ since the action on the special fiber is induced by that on $\cal{A}'$. The generic fiber of $g^* \cal{L}$ is equal to $g^* L$ which is isomorphic to $L$ itself since $L$ was defined over $K$. As $g^* \cal{L}$ is the unique line bundle lifting $g^*L$, we must have $g^* \cal{L} \simeq \cal{L}$ and thus 
$$g^*(\cal{L}_s) = (g^* \cal{L})_s \simeq \cal{L}_s.$$
\end{proof}

We now bound $\deg_{\cal{L}_s}((\cal{A}_s)^{\mu_d})$. We will prove the following general proposition that bounds $\deg_L(X^G)$ for a finite group $G$ acting on a projective variety $X$ and $L$ is an ample line bundle which is $G$-invariant, i.e., which satisfies $g^*L \simeq L$ for all $g \in G$. 

\begin{proposition} \label{NeronDegreeBound1} Let $X$ be a projective variety of dimension $d$ over an algebraically closed field $k$, and let $G$ be a finite group acting on $X$ by automorphisms. For a very ample line bundle $L$ on $X$, denote by $\deg_L(\cdot)$ the degree with respect to $L$, and let \[X^G:=\{x\in X:g(x)=x\text{ for all } g\in G\}\] be the fixed-point subvariety. Assume that $g^*L$ and $L$ have the same numerical class for all $g \in G$. Then \[\deg_L(X^{G})\le|G^{\mathrm{ab}}||G|^d\deg_L(X).\]
\end{proposition}

\begin{proof} Fix once and for all an ordering $G=\{a_1,\dots,a_n\}$, and set \[M=\bigotimes_{i=1}^n a_i^*L.\]
For $g\in G$ we compute \[g^*M=g^*\Big(\bigotimes_{i=1}^n a_i^*L\Big)=\bigotimes_{i=1}^n g^*(a_i^*L)=\bigotimes_{i=1}^n (a_i\circ g)^*L.\] Since $i\mapsto a_i\circ g$ permutes $\{a_1,\dots,a_n\}$, there is a canonical permutation isomorphism \[\psi_g:g^*M\longrightarrow M\]
defined by sending each factor $(a_i\circ g)^*L$ of $g^*M$ to the unique slot in $M=\bigotimes_{j=1}^n a_j^*L$ labeled by the same $a_j$. In this way $\psi_g$ restores the fixed reference order $(a_1,\dots,a_n)$ in the codomain.

Fix $g,h\in G$. We compare two maps $(hg)^*M\to M$. First, we have \[\psi_{hg}:(hg)^*M=g^*\big(\bigotimes_{i=1}^n (a_i\circ h)^*L\big)=\bigotimes_{i=1}^n (a_i\circ h\circ g)^*L
\xrightarrow[t=a_i\circ h\circ g]{} \bigotimes_{t\in G} t^*L=M,\]
where the reindexing sends each factor $(a_i\circ h\circ g)^*L$ to the unique slot of $M=\bigotimes_{j=1}^n a_j^*L$ labeled by $a_j=t=a_i\circ h\circ g$, thereby restoring the fixed order.

Second, consider \[\psi_g\circ g^*\psi_h:(hg)^*M=g^*h^*M\xrightarrow{g^*\psi_h} g^*M\xrightarrow{\psi_g} M.\]
Here $g^*\psi_h$ is the map \[g^*h^*M=\bigotimes_{i=1}^n (a_i\circ h\circ g)^*L \longrightarrow g^*M=\bigotimes_{j=1}^n (a_j\circ g)^*L\] obtained by reindexing via $a_j=a_i\circ h$, i.e. sending the factor $(a_i\circ h\circ g)^*L$ to the factor $(a_j\circ g)^*L$. Then $\psi_g$ reindexes $g^*M=\bigotimes_{j=1}^n (a_j\circ g)^*L$ by $t=a_j\circ g$ to identify it with $M=\bigotimes_{t\in G} t^*L$, again by sending each factor to the unique slot labeled by the same $a_j$ and restoring the fixed order. Thus both constructions carry each factor $(a_i\circ h\circ g)^*L$ of $(hg)^*M$ to the slot of $M$ labeled by $a_j=a_i\circ h\circ g$, and hence
\begin{equation}\label{eqn:cocycle}\psi_{hg}=\psi_g\circ g^*\psi_h.\end{equation}

Define $T_g:H^0(X,M)\to H^0(X,M)$ by $(T_gs)(x)=(\psi_g)_x((g^*s)(x))$. By \eqref{eqn:cocycle}, for $s\in H^0(X,M)$, 
\[(T_gT_hs)(x)=(\psi_g)_x((g^*(T_hs))(x))=((\psi_g\circ g^*\psi_h)_x)(((hg)^*s)(x))\] 
\[=(\psi_{hg})_x(((hg)^*s)(x))=(T_{hg}s)(x).\]
Thus $g\mapsto T_g$ is a right action: $T_gT_h=T_{hg}$. If a left representation is desired, set $S_g:=T_{g^{-1}}$, then $S_gS_h=S_{gh}$. Either way, $H^0(X,M)$ carries a representation of $G$.

As $L$ is very ample, it follows that $a_i^*L$ is also very ample and hence $M$ is very ample. The complete linear system of $M$ defines a closed embedding \[\iota:X\longrightarrow\mathbb{P}(H^0(X,M)^\vee)=:\mathbb{P}^N,\] where $\iota(x)$ is the point of $\mathbb{P}(H^0(X,M)^\vee)$ corresponding to the line of functionals $\{\ell\mapsto\ell(s(x)):s\in H^0(X,M)\}$. Write $V:=H^0(X,M)$ and $W:=V^\vee$. Both are finite-dimensional vector spaces over $k$. The action of $G$ on $V$ induces the dual action on $W$, hence a projective action on $\mathbb{P}(W)$ defined as follows: if $0 \neq w \in W$, then the corresponding point of $\mathbb{P}(W)$ is the line $\{\ell w:\ell \in k\}\subseteq W$, and we set \[g \cdot \{\ell w:\ell\in k\}:=\{\ell(g\cdot w):\ell\in k\}.\] With this action the embedding $\iota$ is readily checked to be $G$-equivariant.

A point $\{\ell w:\ell\in k\}\in \mathbb{P}(W)$ is fixed by $G$ exactly when for each $g\in G$ the vectors $w$ and $g\cdot w$ lie in the same one-dimensional subspace of $W$. Equivalently, for each $g\in G$ there exists a scalar $\chi(g)\in k^\times$ such that $g\cdot w=\chi(g)w$, where the left-hand side uses the dual action of $G$ on $W$ defined above. This condition precisely says that the line $\{\ell w:\ell\in k\}$ is a one-dimensional $G$-subrepresentation of $W$. The assignment $\chi:G\to k^\times$, $g\mapsto\chi(g)$, is a group homomorphism: indeed, since $T_gT_h=T_{hg}$ on $V$, the dual action satisfies \[g\cdot(h\cdot w)=(hg)\cdot w.\] If $g\cdot w=\chi(g)w$ and $h\cdot w=\chi(h)w$, then  \[(hg)\cdot w=g\cdot(h\cdot w)=g\cdot(\chi(h)w)=\chi(h)(g\cdot w)=\chi(h)\chi(g)w=\chi(g)\chi(h)w.\] Thus $\chi$ is a character of $G$. Conversely, if $\chi:G\to k^\times$ is a character and there exists a nonzero $w\in W$ with $g\cdot w=\chi(g)w$ for all $g\in G$, then the line $k\cdot w$ is $G$-stable. For each character $\chi:G\to k^\times$, let $W_\chi \subseteq W$ be the subspace consisting of all $w\in W$ satisfying $g\cdot w=\chi(g)w$ for every $g\in G$. From this identification, we obtain \[(\mathbb{P}(W))^G=\bigcup_{\substack{\chi:G\to k^\times\\W_\chi\neq 0}} \mathbb{P}(W_\chi).\]
	
Since $\iota$ is $G$-equivariant,  \[\iota(X^G)=\iota(X)^G=\iota(X)\cap(\mathbb{P}(W))^G=\bigcup_{\chi}\big(\iota(X)\cap \mathbb{P}(W_\chi)\big).\] For projective $Y\subseteq\mathbb{P}(W)$ and linear $\Lambda\subseteq\mathbb{P}(W)$ one has $\deg(Y\cap\Lambda)\le \deg(Y)$. Recalling that there are at most $|G^{\mathrm{ab}}|$ characters $\chi:G\to k^\times$, it follows that \begin{equation}\label{eq:degM-fixed}
		\begin{aligned}
			\deg_M(X^G)
			&=\deg_M\Big(\bigcup_{\chi} \big(X \cap \mathbb{P}(W_\chi)\big)\Big) \\
			&\le \sum_{\chi} \deg_M\big(X \cap \mathbb{P}(W_\chi)\big) \\
			&\le\sum_{\chi} \deg_M(X) \\
			&\le |G^{\mathrm{ab}}| \deg_M(X).
		\end{aligned}
	\end{equation}
	
Let $Z$ be an irreducible component of $X^G$. For every $g \in G$ and every irreducible component $Z \subseteq X^G$, we have \[(g^*L)|_Z=(g|_Z)^*(L|_{g(Z)})=(\mathrm{id}_Z)^*(L|_Z)=L|_Z.\] Hence \[c_1(M|_Z)=\sum_{g\in G} c_1((g^*L)|_Z)=|G|\cdot c_1(L|_Z).\] Therefore \begin{equation}\label{eq:component-scaling}\deg_M(Z) = |G|^{\mathrm{dim}(Z)} \deg_L(Z).\end{equation} Summing over all components of $X^G$ gives \begin{equation}\label{eq:degL-below-degM-on-fixed}\deg_L(X^G)=\sum_Z \deg_L(Z)\le\sum_Z |G|^{\dim Z}\deg_L(Z)=\deg_M(X^G).\end{equation}
	
Finally, set $D_h:=c_1(h^*L)$. Then \[c_1(M)=\sum_{h\in G}D_h,\qquad c_1(M)^d=\sum_{(h_1,\dots,h_d)\in G^d}D_{h_1}\cdots D_{h_d}.\] By our assumption, the classes $D_h$ are all numerically equivalent to $c_1(L)$. 
Hence every term in the expansion satisfies \[(D_{h_1}\cdots D_{h_d}\cdot[X])=(c_1(L)^d\cdot[X])=\deg_L(X).\] As there are $|G|^d$ terms, we conclude that
\begin{equation}\label{eq:degM-on-X}\deg_M(X)=(c_1(M)^d\cdot[X])=|G|^d\deg_L(X).\end{equation} Combining \eqref{eq:degM-fixed}, \eqref{eq:degL-below-degM-on-fixed}, and \eqref{eq:degM-on-X} shows that \[\deg_L(X^G)\le\deg_M(X^G)\le |G^{\mathrm{ab}}|\deg_M(X)=|G^{\mathrm{ab}}||G|^d\deg_L(X).\]\end{proof}

As an immediate corollary, we obtain the following. 

\begin{corollary} \label{NeronDegreeBound2}
In the notation preceding Proposition \ref{NeronDegreeBound1}, we have 
$$\deg_{\cal{L}_s}((\cal{A}_s)^{\mu_d}) \cdot 3^{\dim ((\cal{A}_s)^{\mu_d})} = \deg_{\cal{L}^3_s}((\cal{A}_s)^{\mu_d}) \leq \deg_{L^3}(A) d^{g+1}.$$
\end{corollary}

\begin{proof}
We replace $L$ with $L^3$ so that $\cal{L}_s$ is very ample for $\cal{A}_s$. Then the corollary follows immediately from Proposition \ref{NeronDegreeBound1} where we use the fact that $|G| = d$ and $\dim A = g$.     
\end{proof}

	\section{Uniform Bounds for Principally Polarized Abelian Varieties }\label{section:fakebadbound}

We will now handle the general case of ppavs having at least one bad place of bad reduction which are not necessarily semistable. A naive approach would be to pass to the extension $K' = \bb{C}(B') = K(A[3])$ and apply \eqref{eqn:Deligne}. Although the degree $[K':K]$ is bounded in terms of $g$, the issue is that $g(B')$ can be arbitrarily large compared to $g(B)$ due to having no control a priori over the ramification. 
\par 
By N\'{e}ron--Ogg--Shafarevich, this problem can only happen when $A$ has many places of bad reduction over $K$ itself. If these places are of stable bad reduction, we may then pass to $K'$ and apply Theorem \ref{UniformSemistable1}. Hence the only remaining issue is if ``most" of these places are in fact places of potentially good reduction, i.e., places which become good after passing to a semistable extension. 
\par 
To handle this, we will show that every place $v$ of bad yet potentially good reduction contributes some uniform positive amount to our Arakelov-Green's function $g_n$. Then if there are many such places, we obtain a uniform positive lower bound on $g_n$ analogous to \eqref{eqn:lb1}. 

\par     
Let us first recall some definitions. Let $A/K$ be an abelian variety over a valued field $(K, | \cdot |)$ and let $L$ be a very ample line bundle on $A$ such that the embedding $\iota: A \xhookrightarrow{} \bb{P}^N$ by $L$ is projectively normal. Let $f: \bb{P}^N \to \bb{P}^N$ be an endomorphism over $K$ of degree $d$ that extends multiplication by $[m]$ for some $m \geq 2$ and let $F: \bb{A}^{N+1} \to \bb{A}^{N+1}$ be a homogeneous lift of $f$. Given a basis $\cal{B}_n$ of $H^0(A,L^n)$, recall from \S\ref{section:AG} that we may define the Arakelov-Green's function as 
$$g_n\left(P_1,\dots,P_{c(n)}\right)=\frac{1}{c(n)}\sum_{i=1}^{c(n)}\widehat{H}_F(\widetilde{P_i})-\frac{1}{n\cdot c(n)}\log\left|\mathrm{det}\left(\eta_j(\widetilde{P_i})\right)_{\eta_j\in\mathcal{B}_n}\right|+r(F)$$
where $\widehat{H}_F$ is the local escape rate function, $\widetilde{P_i}$ are lifts of $P_i \in A(\ovl{K}_v)$ to $\bb{A}^{N+1}(\ovl{K}_v)$. The correction $r(F)$ is a constant multiple of $\log |\Res(F)|$ to ensure that $g_n$ is independent of the choice of the lift $F$. The quantity depends on the basis chosen $\cal{B}_n$ but any other basis will change the function by an additive constant.  
\par 
In general, given a subset $T \subseteq A(\ovl{K}_v)$ where $\ovl{K}_v$ is the completion of the algebraic closure of $K$, we may define its (logarithmic) transfinite diameter with respect to $\cal{B}_n$ as 
$$\log d_{n,v}(\cal{B}_n,T,L) = -\sup_{\vec{P} \in T^{c(n)}} g_{n,v}(P_1,\ldots,P_{c(n)}),$$ where $\vec{P}=(P_1,\dots,P_{c(n)})$ and we write $g_{n,v}$ for the Arakelov-Green's function $g_n$ associated to $|\cdot|_v$.
Again, this depends on the choice of basis $\cal{B}_n$; however, the relative quantity 
$$\log d_{n,v}(T,L) := \log d_{n,v}(\cal{B}_n,T,L) - \log d_{n,v}(\cal{B}_n,A,L)$$
does not. If $K$ is a product formula field, then the global quantity 
$$\sum_{v \in M_K} \log d_{n,v}(\cal{B}_n,A,L)$$
doesn't depend on $\cal{B}_n$ and so it is useful to study this relative quantity $\log d_n(T,L)$ in order to obtain bounds on 
$$\sum_{v \in M_K} \log d_{n,v}(\cal{B}_n,T,L).$$

We now state and prove the main result of this section. 

\begin{theorem}\label{thm:compression} Let $K$ be a non-archimedean discretely valued field with algebraically closed residue field $k$. Let $A/K$ be a $g$-dimensional ppav with principal polarization $\Theta$ and let $L$ be a line bundle representing $4 \Theta$. Assume that $A$ has bad reduction over $K$ but attains good reduction over a finite extension of $K$. Let $T = A(K)$. Then there are constants $N , \delta >0$ depending only on $g$ such that for all $n \geq N$,   
    
     \[\log d_n(T,L)\le \delta \log |\pi|.\]
     where $\pi$ is a uniformizer of $K$. 
     \end{theorem}

\begin{proof} The line bundle $L$ is very ample and its induced embedding $\iota: A \xhookrightarrow{} \bb{P}^N$ is projectively normal. Let $K' = K(A[3])$ so that $A$ has good reduction over $K'$. Let $R$ and $R'$ be the valuation rings. Then there is an abelian scheme $\cal{A}'$ over $R'$ whose generic fiber is $A$. Let $\cal{N}/R$ be the N\'{e}ron model for $A/K$. Then $\cal{N}_{R'}/R'$ is a smooth model (but not the N\'{e}ron model) for $A/K'$ and by the N\'{e}ron mapping property, we obtain an $R'$-morphism $f: \cal{N}_{R'} \to \cal{A}'$. This gives us a morphism of special fibers $f_s: \cal{N}_{R',s} \to \cal{A}'_s$ and since our residue field is algebraically closed, we have $\cal{N}_{s} \simeq \cal{N}_{R',s}$ and thus a morphism $f_s: \cal{N}_s \to \cal{A}'_s$. 
\par 
The line bundle $L$ extends uniquely to a relatively ample line bundle $\cal{L}$ on $\cal{A}_{K'}$. By \cite[Proposition 6.13]{MFK94}, for each $n$ the $R'$-module $H^0(\cal{A}', \cal{L}^n)$ is free and hence has the same dimension as $H^0(A_{K'}, L^n)$ and $H^0(\cal{A}'_s, \cal{L}_s^n)$. To compute $\log d_n(T,L)$, we may choose any basis $\cal{B}_n$ we want: in particular, we choose $\cal{B}_n$ to be a basis of $H^0(\cal{A}',\cal{L}^n)$. Then it is clear that $g_{\cal{B}_n}$ will be independent of the choice of such a basis, as the change of basis matrix between any two such bases satisfies $|\det| = 1$. 
\par 
We first calculate $\log d_n(\cal{B}_n,A,L).$

\begin{lemma} \label{Transfinite1}
We have 
$$\log d_n(\cal{B}_n,A,L) = 0 $$
for all $n$. 
\end{lemma}

\begin{proof}
Our line bundle $\cal{L}$ above gives us an embedding $\cal{A}' \xhookrightarrow{} \bb{P}^N_{R'}$. By Proposition \ref{prop:MorphismExtension1Chapter6} and the fact that $R'$ is a DVR, there exists an endomorphism $f: \bb{P}_{R'}^N \to \bb{P}_{R'}^N$ extending multiplication by $[m]$ on our abelian scheme $\cal{A}$. Since $f$ is an endomorphism of $\bb{P}_{R'}^N$, we can pick a lift $F: \bb{A}_{K'}^{N+1} \to \bb{A}_{K'}^{N+1}$ such that $r(F) = 0$ and the escape rate function on $\bb{A}^{N+1}$ is given by 
$$\Ht_F((x_0,\ldots,x_N)) = \log \max\{|x_0|,\ldots,|x_N|\}.$$
Hence given $c(n)$ points $P_1,\ldots,P_{c(n)} \in A(\ovl{K})$, picking affine lifts $\widetilde{P_i} = (x_{i,0},\ldots,x_{i,N})$ to $\bb{A}^{N+1}$ such that $\max_{0 \leq j \leq N} \{|x_{i,j}|\} = 1$ for all i, we obtain 
$$g_n(P_1,\ldots,P_{c(n)}) = -\frac{1}{n\cdot c(n)}\log\left|\mathrm{det}\left(\eta_j(\widetilde{P_i})\right)_{\eta_j\in\mathcal{B}_n}\right|.$$
Now each $\eta_j$ is a homogeneous polynomial with integral coefficients, and thus we clearly have $g_n(P_1,\ldots,P_{c(n)}) \geq 0$ and $\log d_n(\cal{B}_n,A,L) \leq 0$. On the other hand, as $\cal{B}_n$ is a basis of $H^0(\cal{A}, \cal{L}^n)$, our determinant has absolute value $< 1$ if and only if the reductions of $P_i$ to the special fiber $\cal{A}_s$ lie on $\divv(t)$ for some nonzero $t \in H^0(\cal{A}_s,\cal{L}_s^n)$. Clearly we may choose $P_1,\ldots,P_{c(n)} \in A(\ovl{K})$ such that this does not happen. We conclude that $\log d_n(\cal{B}_n,A,L) = 0$.\end{proof}

We now bound $\log d_n(\cal{B}_n,T,L)$ from above. By Proposition \ref{NeronModel1} and Corollary \ref{NeronDegreeBound2}, we know that if $P_i \in A(K)$, then their reductions in $\cal{A}_s$ lie on a subvariety $B$ with $\deg_{\cal{L}_s}(B) \leq D$ for some $D$ depending only on $g$. Furthermore, this subvariety $B$ is a finite union of abelian subvarieties $B_1,\ldots,B_e$ as noted in the discussion after Proposition \ref{NeronModel1}. We may use the trivial upper bounds $e \leq D$ and $\deg_{\cal{L}_s}(B_i) \leq D$, to conclude that
$$\dim H^0(B, \cal{L}_s^n) \leq \sum_{i=1}^{e} \dim H^0(B_i, \cal{L}_s^n) \leq D^2 n^{g-1}$$
where we are using the fact that for any abelian variety $C$ with an ample line bundle $L$, we have
$$\dim H^0(C,L^n) = \frac{\deg_L(C)}{(\dim C)!}n^{\dim C}.$$
Since $\dim H^0(\cal{A}'_s,\cal{L}_s^n) \geq \frac{n^g}{g!}$, for some sufficiently large $n_0$ depending only on $g$, we may find a nonzero section $\ovl{s}$ of $H^0(\cal{A}'_s,\cal{L}_s^{n_0})$ that vanishes on $B$. We may lift this to a section $s$ of $H^0(\cal{A}',\cal{L}^{n_0})$ since $H^0(\cal{A}',\cal{L}^{n_0})$ is free. For any $k,r \geq 1$, consider the multiplication by $s^{k}$ map which sends $H^0(\cal{A}',\cal{L}^{kn_0+r})$ to $H^0(\cal{A}',\cal{L}^{2kn_0+r})$. This map is injective on both the generic fiber and the special fiber. Since $H^0(\cal{A}', \cal{L}^{2kn_0+r})$ is free, this means we may choose a basis of $H^0(\cal{A}',\cal{L}^{2kn_0+r})$ such that $c(kn_0)$ many of its elements are multiples of $s^k$. 
\par 
Given $n \geq 2n_0$, we write $n$ in the form $2kn_0+r$ for $k \geq 1$ and $0 \leq r < 2n_0$. We now compute $\log d_{n}(\cal{B}_n,T,L)$ for $T\subseteq A(K)$. Just as in Lemma \ref{Transfinite1}, picking lifts $\tilde{P}_i$ with $\max\{|x_i|\} = 1$, we obtain 
$$g_n(P_1,\ldots,P_{c(n)}) = -\frac{1}{n \cdot c(n)} \log\left|\mathrm{det}\left(\eta_j(\widetilde{P_i})\right)_{\eta_j\in\mathcal{B}_n}\right|. $$
The reduction of each $\tilde{P}_i$ lies on the subvariety $B$, and hence $|\eta(\tilde{P}_i)| < 1$ where $\eta$ is the homogeneous polynomial representing our section $s$. As both $\eta_j$ and $\tilde{P}_i$ are defined over $K'$, we have 
$$|\eta(\tilde{P}_i)| \leq |\pi|^{1/[K':K]}.$$
Now we have $c(\frac{n-r}{2})$ many $\eta_j$'s for which $\eta_j$ is an integral multiple of $\eta^{(n-r)/2n_0}$. In particular, 
$$\log |\eta_j(\tilde{P}_i)| \leq  \frac{n-r}{2 n_0 [K':K]} \log|\pi|.$$
We have $c(\frac{n-r}{2})$ many such $\eta_j$'s and all other rows have $| \cdot |$ at most $1$. Thus 
$$\frac{1}{n \cdot c(n)} \log\left|\mathrm{det}\left(\eta_j(\widetilde{P_i})\right)_{\eta_j\in\mathcal{B}_n}\right| \leq \frac{c(\frac{n-r}{2}) (n-r)}{2 n_0 [K':K] n c(n)} \log |\pi| \leq \delta \log |\pi|$$
for some $\delta > 0$ depending only on $g$ as desired. 
\end{proof}

We can now extend Corollary \ref{UniformSemistable2} to the non-semistable case. Recall that we say a place $v \in M_K$ has stable bad reduction if $A/K_v$ does not have potential good reduction. 

\begin{corollary} \label{NonSemistable1} Let $g\ge1$ and let $K = \bb{C}(B)$ be the function field of a smooth projective curve $B/\bb{C}$. Let $g(B)$ be the genus of $B$. There are positive integers $M_0=M_0(g, g(B))$ and $n=n(g,g(B))$ and an $\epsilon=\epsilon(g,g(B))>0$ with the following property. Let $(A,\Theta)$ be a dimension $g$ principally polarized abelian variety over $K$. Let $L= 4 \Theta$, and let $S_0$ be the set of places of stably bad reduction for $A$ and let $S$ be the set of all places of bad reduction. Assume that $|S| \geq 1$. Then either: 
\begin{itemize} \item $|S_0|\ge \max\left\{1,\dfrac{h_{\omega}(A)}{M_0}\right\}$, or \item for any $T\subseteq A(K)$, either $T$ is contained in $\divv(s)$ for some $s\in H^0(A,\mathcal{L}^n)$, or 
$$\frac{1}{|T|} \sum_{x \in T} \h_{A,L}(x) \geq \eps h_{\omega}(A).$$
\end{itemize} \end{corollary}

\begin{proof} Let $S$ be the places of bad reduction of $A/K$ and let $S_0'$ denote the set of places of unstable bad reduction of $A$, i.e., those places of bad reduction which become places of good reduction upon passage to a semistable extension. Then $S = S_0 \sqcup S_0'$. Let $K' = K(A[3])$ and let $g(B')$ be the genus of $B'$ where $K' = \bb{C}(B')$. By \eqref{eqn:Deligne}, there are positive constants $C_1$ and $C_2$ depending only on $g$ such that \[h_\omega(A)\le C_1|S_0|+C_2\max\{g(B'),1\}.\] 
Now observe that $K'/K$ can only be ramified over places in $S$ and thus 
$$g(B') \leq [K':K](|S| +  g(B)).$$
Hence for some $C_3 > 0$ depending only on $g$, we have 
$$h_{\omega}(A) \leq C_1 |S_0| + C_3 (|S| + g(B) + 1).$$
It follows that if $M_0\gg_{g,g(B)}1$ and $|S_0| < \max\{1,\frac{h_{\omega}(A)}{M_0}\}$, then $$|S'_0| \geq \frac{h_{\omega}(A)}{M_0},$$
as we have assumed that $|S| \geq 1$. Now for any basis $\cal{B}_n$ of $H^0(A,L^n)$, Theorem \ref{thm:basisbound} implies (similarly as in \eqref{eq:ElkiesLower1}) that 
$$\sum_{v \in M_K} \log d_n(\cal{B}_n,A,L) \leq C \frac{\log n}{n} h_{\omega}(A)$$
for some $C > 0$ depending only on $g$. For each $v \in S'_0$, Theorem \ref{thm:compression} tells us that 
$$\log d_n(\cal{B}_n,T,L) - \log d_n(\cal{B}_n,A,L) \leq \delta \log |\pi| = -\delta$$
as $\log |\pi| = -1$. Hence we get
$$\sum_{v \in M_K} \log d_n(\cal{B}_n,T,L) \leq C \frac{\log n}{n} h_{\omega}(A) - \delta \frac{h_{\omega}(A)}{M_0} \leq -\frac{\delta}{2 M_0} h_{\omega}(A)$$
when $n\gg_{g,g(B)}1$. Since for any set of $c(n)$ points $x_1,\dots,x_{c(n)}\in T$ which do not all lie in $\divv(s)$ for any nonzero section $s\in H^0(A,L^n)$, we have $$\frac{1}{c(n)} \sum_{i=1}^{c(n)} \h_{A,L}(x_i) = \sum_{v \in M_K} g_{\cal{B}_n,v}(x_1,\ldots,x_{c(n)}) \geq \sum_{v \in M_K} -\log d_n(\cal{B}_n,T,L),$$
it follows that either \[\frac{1}{c(n)} \sum_{i=1}^{c(n)} \h_{A,L}(x_i) \ge\frac{\delta}{2M_0}h_\omega(A),\] or there exists a nonzero $s \in H^0(A,L^n)$ such that any set of $c(n)$ elements of $T$ is contained in $\divv(s)$. As seen in the proof of Theorem \ref{UniformSemistable1}, the latter implies that there exists a nonzero $s \in H^0(A,L^n)$ such that all elements of $T$ lie in $\divv(s)$. This completes the proof of the Corollary.\end{proof}

Combining Corollary \ref{NonSemistable1} with Theorem \ref{UniformSemistable1}, we obtain the following. For a ppav $(A,\Theta)$, let $L = 4 \Theta$ and let $\h_{A,L}$ denote the N\'eron--Tate height with respect to $L$. 

\begin{theorem} \label{UniformGeneral1}
Let $K = \bb{C}(B)$ be the function field of a smooth projective curve $B$ over $\bb{C}$ and let $g(B)$ be the genus of $B$. Let $g \geq 1$ be a positive integer. Then there is a positive integer $n = n(g,g(B))$ and a positive constant $\eps = \eps(g,g(B)) > 0$ such that for any $g$-dimensional ppav  $(A,\Theta)$ with at least one place of bad reduction, the set 
$$\{x \in A(K) \mid \h_{A,L}(x) \leq \eps h_{\omega}(A)\}$$
is contained in $\divv(s)$ for some $s \in H^0(A,L^n)$. 

\end{theorem}

\section{The case of Everywhere Good Reduction} \label{sec: GoodReduction}
We now handle the case where we have a principally polarized abelian variety $A$ of dimension $g$ with everywhere good reduction. By \eqref{eqn:Deligne}, such abelian varieties must have Faltings height bounded uniformly from above and so it suffices to handle this setting. We will prove the following theorem. For a ppav $(A,\lambda)$, let $L$ be a symmetric ample line bundle inducing $\lambda$ and let $\h_A$ denote the N\'{e}ron--Tate height induced by $L$. 

\begin{theorem} \label{UniformBoundednessGoodReduction1}
Let $K = \bb{C}(B)$ for a smooth projective curve $B$ over $\bb{C}$ and let $M,\delta > 0$. Then there exists a natural number $N \geq 1$, depending only on $(g,g(B),M,\delta)$, such that for any principally polarized semistable abelian variety $A/K$ of dimension $g$ and $h_{\omega}(A) \leq M$, the set 
$$\left \{ x \in A(K)/\Tr_{K/\bb{C}}(A)(\bb{C}) \mid \h_A(x) \leq \delta \right\} \leq N.$$
\end{theorem}

Since $\h_A(y) = 0$ for $y \in \Tr_{K/\bb{C}}(A)(\bb{C})$, the quantity $\h_A(x)$ is independent of the choice of representative for $x \in A(K)/\Tr_{K/\bb{C}}(A)(\bb{C})$ and hence well-defined. When $\Tr_{K/\bb{C}}(A) = 0$ is trivial, this gives an upper bound on the number of small points of $A(K)$. 
\par 
For a number field $K$, the analogous statement of Theorem \ref{UniformBoundednessGoodReduction1} is relatively straightforward. Using the work of David--Philippon \cite{DP02}, we may find an embedding $A \xhookrightarrow{} \bb{P}^N$ such that any $x \in A(K)$ with $\h_A(x) \leq \delta$ has Weil height at most $\delta'$, with $\delta'$ depending only on $(g,M,\delta)$. It then follows from the Northcott property that there is a uniform bound on the number of such points. 
\par 
For a function field $K = \bb{C}(B)$, however, the Northcott property does not hold and we have to obtain finiteness through other means. Our strategy is to follow the proof of the Lang--N\'{e}ron theorem \cite{Con06}, which is the analogue of Mordell--Weil for function fields. This uses the number of irreducible components of a $\Hom$ scheme, with an upper bound on the degree, as a replacement for the Northcott property. We will then relate the components of a $\Hom$ scheme to that of a Hilbert scheme, where we can use the theory of Gotzmann numbers to obtain uniform upper bounds. Our proof can be broken into three parts:

\begin{enumerate}
\item Construct a variety $\tilde{A} \subseteq \bb{P}^N_B = \bb{P}^N \times_{\bb{C}} B$ whose generic fiber is isomorphic to $A/K$ which is cut out by equations with degree bounded only in terms of $g,g(B),M,\delta$. Observe that $\bb{P}^N \times_{\bb{C}} B = \bb{P}^N \times B$ as we are working with varieties over $\bb{C}$.

\item Reduce Theorem \ref{UniformBoundednessGoodReduction1} to a problem on bounding the number of irreducible components of the Hilbert scheme of $\tilde{A}$, with the coefficients of the Hilbert polynomial bounded only in terms of $g,g(B),M,\delta$. 

\item Apply the work of Gotzmann \cite[Section 5]{Kwe22} to obtain uniform bounds on the number of irreducible components depending only on $g,g(B),M,\delta$. 
\end{enumerate}

We first show that we may assume the existence of a level $3$ structure for $A$ over $K$. 

\begin{proposition} \label{SemistableDegreeBound1}
Let $A/K$ be a semistable abelian variety of dimension $g$ and assume that $h_{\omega}(A) \leq M$. Let $K' = K(A[3])$, so that $A/K'$ has level $3$ structure. Let $K' = \bb{C}(B')$ and $K = \bb{C}(B)$. Then $g(B')$ can be bounded purely in terms of $(g,g(B),M)$. 
\end{proposition}

\begin{proof}
By N\'{e}ron--Ogg--Shafarevich, the extension $K'/K$ can only be ramified over places of bad reduction for $A/K$. Now if $S$ denotes the number of places of bad reduction, then for $p = 3$, we have
$h_{\partial \cal{A}_{g,3}}(A) \geq |S|$
since 
$\lambda_{v, \partial \cal{A}_{g,3}}(A) \geq 1$
for each place $v$ of bad reduction for $A/K$. As there exists a constant $c = c(g) > 0$ for which 
$$M \geq h_{\omega}(A) > c h_{\partial \cal{A}_{g,3}}(A),$$
it follows that $|S|$ is bounded in terms of $g$ and $M$. Hence the degree of the ramification divisor for $\pi: B' \to B$ is bounded by $|S| [K':K]$, which is bounded in terms of $g$ and $M$. As the degree of $\pi$ is bounded in terms of $g$, by Riemann--Hurwitz, we thus have that $g(B')$ is bounded in terms of $(g,g(B),M)$ as desired.
\end{proof}

By Proposition \ref{SemistableDegreeBound1}, we may now replace $K$ with $K(A[3])$ for Theorem \ref{UniformBoundednessGoodReduction1} and $g(B')$ remains bounded in terms of the initial data $(g,g(B),M)$. The constants $M, \delta$ can be replaced by $[K':K]M$ and $[K':K]\delta$ which is bounded in terms of $(g,M,\delta)$. By the proof of \cite[Lemma 7.3]{Con06}, we have an injection 
$$A(K)/\Tr_{K/\bb{C}}(A)(\bb{C}) \xhookrightarrow{} A(K')/\Tr_{K'/\bb{C}}(A)(\bb{C})$$
and hence we may assume without loss that $A$ has level $3$ structure defined over $K$.
\par 
Recall from Proposition \ref{ModuliOpenCover1Chapter6} that we have open sets $U_1,\ldots,U_m$ covering $\mathcal{A}_{g,3}$ such that on each $U_i$, we have an embedding $\mathcal{A}_{U_i}\hookrightarrow\mathbb{P}^{N_i}_{U_i}$ and a morphism $F:\mathbb{P}^{N_i}_{U_i}\to\mathbb{P}^{N_i}_{U_i}$ of degree $n_i^2$ that extends multiplication-by-$[n_i]$. Moreover, the $n_i$ and $N_i$ may be chosen to depend only on $g$. Without loss of generality, let $A/K$ with its level $3$ structure define a point on $U_1(K)$. We let $F_A$ be the corresponding morphism on $\bb{P}_K^{N_1}$ induced by our point on $U_1(K)$. Let $F$ be given by $[F_0: \cdots : F_{N_1}]$. We define 
$$h(F) := \sum_{v \in M_K} \log \max\{|F_0|_v,\ldots,|F_{N_1}|_v\}$$
where $|F_i|_v$ is the maximum of the absolute values of the coefficients of $F_i$. 

\begin{proposition} \label{HeightInequalityEndomorphism1}
There exists a $c_1 > 0$ depending only on $g$ such that 
$$h(F_A) \leq c_1 h_{\omega}(A).$$
\end{proposition}

\begin{proof}
We apply Proposition \ref{GlobalHeightAmple2} with $X = U_1, Y=\mathrm{Mor}_{n_1^2}(\mathbb{P}^{N_1},\mathbb{P}^{N_1})$, $\varphi$ sending $[B]\in U_1(K)$ to the corresponding $F_B\in Y$, $L=\omega$ on $\mathcal{A}_{g,p}^*$, and $M=\mathcal{O}(1)$ on the compactification of $Y$ that yields the Weil height $h(F)$. This gives us our constant $c_1 > 0$ such that 
$$h(F_A) \leq c_1 h_{\omega}(A)$$,
as desired.\end{proof}

Let $\mathcal{P}_1,\ldots,\mathcal{P}_m$ be homogeneous polynomials of degree $t$ on $\mathbb{P}^{N_1}_{U_1}$ that cut out $\mathcal{A}_{U_1}$. Pulling back along the $K$-point $[A]$ gives equations $P_1,\ldots,P_m$ of degree $t$ that cut out $A$ inside $\mathbb{P}^{N_1}_K$. Let $h(P_i)$ denote the maximum height of its coefficients. 

\begin{proposition}\label{HeightBoundEquations1}
	There exists a $c_2>0$ depending only on $g$ such that $h(P_i)\le c_2\,h_\omega(A)$ for all $i$.
\end{proposition}

\begin{proof}
	Let $a$ be a coefficient of some $P_i$. It suffices to prove the existence of $c>0$ depending only on $g$ such that $h(a)\le c\,h_\omega(A)$. Let $\mathfrak{a}$ be the function on $U_1$ that represents the coefficient $a$. Then $\mathfrak{a}:U_1\to\mathbb{A}^1$ is a morphism over $\mathbb{C}$. Applying Proposition \ref{GlobalHeightAmple2} with $X=U_1$, $Y=\mathbb{A}^1$ (compactified to $\mathbb{P}^1$), $\varphi=\mathfrak{a}$, $L=\omega$ on $\cal{A}_{g,3}^*$, and $M=\mathcal{O}_{\mathbb{P}^1}(1)$, we obtain $h(a)\le c\,h_\omega(A)$ as desired.
\end{proof}

We now proceed with the first step, which is to construct a closed subvariety $\tilde{A} \subseteq  \bb{P}^{N_1}_B = \bb{P}_{\bb{C}}^{N_1} \times B$ whose generic fiber is $A$ embedded in $\bb{P}^{N_1}_K$. Furthermore, we would like to control the degree of equations that cut out $\tilde{A}$ with respect to some ample line bundle on $\bb{P}_B^{N_1}$.  

 We first pick a base point $z_0 \in B(\bb{C})$. We have equations $F_i(x_0,\ldots,x_{N_1})$ of degree $\leq t$ and height $\leq c_2 M$ that cut out $A$ inside $\bb{P}^{N_1}_{K}$. We may assume that $c_2M$ is a positive integer. By scaling, we may further assume that there is one coefficient of $F_i$ which is $1$. Then for any other coefficient $a$ of $F_i$, we have $h(a) \leq c_2 M$. Viewing $a$ as a function of $B$, this means that the number of poles $a$ has on $B$ is at most $c_2 M$. 

\begin{proposition} \label{HeightBoundEquations2}
There exists an integer $M' > 0$, depending only on $(g,g(B),M)$, and $f_A \in K^{\times}$ with $\divv(f_A) \geq -M' [z_0]$, depending on $A$, such that for every $i$, the coefficients of the polynomial $f_AF_i$ are functions on $B$ with poles only on $z_0$. In particular, $\divv(f_Aa) \geq -(M'+c_2 M)[z_0]$ for each coefficient $a$ of $F_i$. 
\end{proposition}

\begin{proof}
Let $a$ be a coefficient of some $F_i$. Then as $a$ has at most $M$ poles on $B$, by Riemann--Roch we may find a function $f_A \in K^{\times}$ (depending on $A$) with $\divv(f_A) \geq -(c_2 M+g(B)+1)[z_0]$ such that $f_a a$ only has poles at $z_0$. Doing this for each coefficient and multiplying the various $f_A$'s together gives us a $f_A$ such that $f_A a$ only has poles at $z_0$ for any coefficient $a$ of any $F_i$. As the number of coefficients is bounded in terms of $g$, it follows that the order of the pole of $f_A$ at $z_0$ is bounded in terms of $(g,g(B),M)$ as desired.
\end{proof}

Let $f=f_A$ be the function in Proposition \ref{HeightBoundEquations2}. Then the coefficients of $fF_i$ all have poles only on $z_0$ with order at most $M'+c_2M$. Let \[L=\mathcal O_B\big((M'+c_2M)z_0\big).\]
Then as each coefficient $a_\alpha$ of $F_i$ satisfies $\mathrm{div}(f a_\alpha)\ge-(M'+c_2M)[z_0]$, we may view $f a_{\alpha}$ as an element of $H^0(B,L)$. We now construct our subvariety $\tilde{A}$ of $\bb{P}_B^{N_1}$.

\begin{proposition} \label{HeightBoundEquations3}
There exists a constant $C_1 > 0$, depending only on $(g,g(B),M)$ such that the following holds: there is an embedding $\bb{P}_{\bb{C}}^{N_1} \times B \xhookrightarrow{} \bb{P}^N$ along with a closed subvariety $\tilde{A} \subseteq \bb{P}_{\bb{C}}^{N_1} \times B$, whose generic fiber is $A/K$, such that $\tilde{A}$ is cut out by equations of degree at most $C_1$.  
\end{proposition}

\begin{proof}
Let $p_1:\mathbb P^{N_1}\times B\to\mathbb P^{N_1}$ and $p_2:\mathbb P^{N_1}\times B\to B$ be the coordinate projections. Viewing $F_i$ as a homogeneous degree-$t$ polynomial in $X_0,\dots,X_{N_1}$, each monomial $X^\alpha$ with $|\alpha|=t$ is a section of $\mathcal O_{\mathbb P^{N_1}}(t)$, and hence pulls back to a section of $p_1^*\mathcal O_{\mathbb P^{N_1}}(t)$ on $\mathbb P^{N_1}\times B$. Each coefficient $f a_\alpha$ is a section of $L$ on $B$, hence pulls back to a section of $p_2^*L$ on $\mathbb P^{N_1}\times B$. The resulting products $fa_\alpha F_i$ give sections of $p_1^*\mathcal O_{\mathbb P^{N_1}}(t)\otimes p_2^*L$, and summing over all monomials $X^\alpha$ with $|\alpha|=t$ shows that $fF_i$ is a global section of this line bundle.

Choose $M'>0$ such that $\deg L=M'+c_2M>2g(B)+1$. (Recall that $c_2M\in\mathbb{Z}_{>0}$.) Then $|L|$ embeds $B$ into $\mathbb P^m$ with $L\simeq\varphi_L^*\mathcal O_{\mathbb P^m}(1)$. Under this embedding, $p_2^*L\simeq\mathcal O_{\mathbb P^{N_1}\times\mathbb P^m}(0,1)$ and hence \[p_1^*\mathcal O_{\mathbb P^{N_1}}(t)\otimes p_2^*L\simeq\mathcal O_{\mathbb P^{N_1}\times\mathbb P^m}(t,1).\]
Thus each $fF_i$ is bihomogeneous of bidegree $(t,1)$ on $\mathbb P^{N_1}\times\mathbb P^m$. The sections $\{fF_1,\dots,fF_m\}$ define a common zero locus $\tilde A\subseteq\mathbb P^{N_1}\times B$. Let $\eta=\operatorname{Spec}\,\mathbb{C}(B)$ be the generic point of $B$. Base-changing $\mathbb{P}^{N_1}\times B$ along $\eta\hookrightarrow B$ gives \[(\mathbb{P}^{N_1}\times B)\times_B\eta\;\cong\;\mathbb{P}^{N_1}_K,\quad K=\mathbb{C}(B).\] Restricting $\tilde{A}$ to this fiber yields $\tilde{A}_\eta\subseteq\mathbb{P}^{N_1}_K$. Since $f$ is a unit in $K^\times$, the equations $fF_i=0$ cut out the same subscheme as $F_i=0$, namely $A\subseteq\mathbb{P}^{N_1}_K$. As $B$ has degree $M'+c_2M$ in $\mathbb P^m$, the subvariety $\mathbb P^{N_1}\times B\subseteq\mathbb P^{N_1}\times\mathbb P^m$ has bidegree $(1,M'+c_2M)$. Using the Segre embedding to embed $\bb{P}^{N_1} \times \bb{P}^m$ into projective space $\bb{P}^N$ for $N = (N_1 + 1)(m+1)-1$ completes the proof.\end{proof}

Now given any $x \in A(K)$ with $\h_A(x) \leq \delta$, we obtain an extension $\tilde{x}: B \to \tilde{A} \xhookrightarrow{} \bb{P}^N$ where we use the embedding from Proposition \ref{HeightBoundEquations3}. We wish to bound the degree of the image $\tilde{x}(B)$. 

\begin{proposition} \label{HeightBoundEquations4}
There exists a constant $C_2 > 0$, depending only on $(g,g(B),M,\delta)$, such that the degree of $\tilde{x}(B)$ in $\bb{P}^N$ is at most $C_2$.   
\end{proposition}

\begin{proof}
We first bound the degree of $\tilde{x}(B)$ in $\bb{P}^{N_1}_{\bb{C}} \times B$, where we use the line bundle $p_1^*\mathcal{O}(1)$ for $\mathcal{O}(1)$ on $\bb{P}^{N_1}_{\bb{C}}$. Here, $p_1$ is the projection to the first coordinate. In this case, this is exactly the standard Weil height $h(x)$ when $x$ is viewed as a point of $\bb{P}^{N_1}(K)$. Recall from the discussion before Proposition \ref{HeightBoundEquations1} that we have an embedding $A \xhookrightarrow{} \bb{P}_K^{N_1}$ and an endomorphism $F: \bb{P}^{N_1}_K \to \bb{P}^{N_1}_K$ of degree $n_1^2$ that extends multiplication-by-$[n_1]$. By \cite[Lemma 6]{Ing22}, there exists a $c > 0$, depending only on $N_1$ and $n_1$, such that 
$$|\h_F(x) - h(x)| \leq ch(F)$$
for all $x \in \bb{P}^{N_1}(K)$. Since $\h_F(x)$ is a multiple of $\h_A(x)$ that depends only on $g$, it follows that $\h_F(x)$ is bounded solely in terms of $g$ and $\delta$. By Proposition \ref{HeightInequalityEndomorphism1}, we have $h(F) \leq c_1 h_{\omega}(A)$ and so there is a constant $c' > 0$ depending only on $g$ and $\delta$ for which $h(x) \leq c' (h_{\omega}(A)+1)$.
\par 
Now the embedding $B \xhookrightarrow{} \bb{P}^m$ is induced by a line bundle $L$ on $B$ of degree at most $M' + c_2 M$. Under the Segre embedding for $\bb{P}^{N_1} \times \bb{P}^m \xhookrightarrow{} \bb{P}^N$, the line bundle $\mathcal{O}(1)$ for $\bb{P}^N$ induces some power of $L$ multiplied by some power of $\mathcal{O}(1)$ for $\bb{P}^{N_1}$ on our curve $\tilde{x}(B)$, where the powers can be bounded in terms of $N_1$ and $m$. Thus the degree of $\tilde{x}(B)$ in $\bb{P}^N$ can be bounded in terms of only $M' + c_2 M, c c' h_{\omega}(A), N_1$ and $m$, all of which are bounded in terms of $(g,g(B),M,\delta)$.\end{proof}

This finishes our first step. We proceed with the next step where we reduce our problem to bounding the number of irreducible components of the Hilbert scheme on $\tilde{A}$. 
\par 
For any positive integer $\beta$, consider the Hom scheme $H_\beta$ over $\bb{C}$ that represents morphisms $\tilde{x}:B\to\tilde{A}$ satisfying the following two conditions:
\begin{itemize}
	\item[(i)] The image curve $\tilde{x}(B)\subset\tilde{A}\hookrightarrow\mathbb{P}^N$ has degree at most $\beta$;
	\item[(ii)] The composition $\pi_2\circ\tilde{x}:B\to B$ is the identity map, so that $\tilde{x}$ is a section of the fibration $\pi_2:\tilde{A}\to B$.
\end{itemize}

The $\Hom$ scheme $H_{\beta}$ is of finite type over $\bb{C}$ (see the discussion after Lemma 7.5 of \cite{Con06}). Hence it has only finitely many irreducible components. As explained before, any point $x \in A(K)$ lifts to a section $\tilde{x}: B \to \tilde{A}$, which by Proposition \ref{HeightBoundEquations4} has bounded degree. Hence for a suitable $\beta$, this gives us an element of $H_{\beta}$. Conversely, given a section $s: B \to \tilde{A}$, we may take the generic fiber to obtain a $K$-point of $A$. The following rigidity phenomena reduces Theorem \ref{UniformBoundednessGoodReduction1} to bounding the number of irreducible components of $H_{\beta}$ for a suitable $\beta$.

\begin{proposition} \label{HomScheme1}
Any two elements of $H_{\beta}(\bb{C})$ that lie on the same irreducible component correspond to the same equivalence class $x \in A(K)/\Tr_{K/\bb{C}}(A)(\bb{C})$. 
\end{proposition}

\begin{proof} 
Let $\tilde{x},\tilde{y} \in H_\beta(\bb{C})$ lie on the same irreducible component of dimension at least $1$ (the case of $0$-dimensional components is trivial). Then there exists an irreducible curve $X\subseteq H_\beta$ containing both $\tilde{x}$ and $\tilde{y}$. Since $H_\beta$ represents a Hom functor, the curve $X \subseteq H_\beta$ corresponds to a morphism \[\mathcal{P}:X \times B\to\tilde{A}\] over $B$, which defines a family of sections of $\tilde{A}\to B$ parametrized by $X$. For each $t \in X(\bb{C})$, the restriction $\mathcal{P}_t:=\mathcal{P}|_{\{t\} \times B}$ is a section $\tilde{x}_t:B\to\tilde{A}$, and in particular $\mathcal{P}_{\tilde{x}}=\tilde{x}$ and $\mathcal{P}_{\tilde{y}}=\tilde{y}$.
	
Let $\eta=\operatorname{Spec}(\mathbb{C}(B))$ denote the generic point of $B$. To study the behavior of the family on the generic fiber, we form the fiber product \[(X\times B)\times_B\eta\cong X\times_{\mathbb{C}} \eta=:X_K,\] where $K=\mathbb{C}(B)$. Pulling back $\mathcal{P}$ along this map yields a morphism \[\mathcal{P}_K: X_K \to \tilde{A}_\eta=A,\] which we interpret as a $K$-morphism $P: X_K\to A$.
	
Let $\overline{X}$ be a smooth projective compactification of $X$, and set $\overline{X}_K=\overline{X} \times_{\mathbb{C}}\operatorname{Spec}(K)$. Then $P$ extends to a morphism \[P:\overline{X}_K \to A.\] Since $\overline{X}_K$ is a smooth proper curve over $K$, after fixing some basepoint $x_0 \in \ovl{X}(\bb{C})$, the morphism $P$ factors through its Albanese variety $\mathrm{Alb}(\overline{X}_K)$, which is an abelian variety over $K$ defined over $\mathbb{C}$. Let $\iota_K: \ovl{X}_K \to \mathrm{Alb}(\ovl{X})_K$ be our embedding sending $\iota_K(x_0)$ to $0$ and $\eta: \mathrm{Alb}(\ovl{X})_K \to A$ be the morphism so that $P = \iota_K \circ \eta$. Then $\eta - P(x_0)$ is a map of abelian varieties between $\mathrm{Alb}(\ovl{X})_K$ and $A$.
\par 
By the universal property of the trace 
$$\tau: \Tr_{K/\bb{C}}(A)_K \to A,$$
there is a map of abelian varieties $f: \mathrm{Alb}(\ovl{X}_K) \to \Tr_{K/\bb{C}}(A)$ so that $\eta - P(x_0) = f \circ \tau$. Hence $P = P(x_0) + (f \circ \iota)_K$ and for any element $x \in \ovl{X}$. This implies that $P_x$ and $P_y$ agree as elements of $A(K)/\Tr_{K/\bb{C}}(A)(\bb{C})$.\end{proof}

With Propositions \ref{HeightBoundEquations4} and \ref{HomScheme1}, it suffices to bound the number of irreducible components of $H_{C_2}$ for $C_2$ in Proposition \ref{HeightBoundEquations4}. Consider the larger Hom scheme $H'_{C_2}$ which consists of all morphisms $\tilde{x}: B \to \tilde{A}$ with $\tilde{x}(B) \xhookrightarrow{} \bb{P}^N$ having degree at most $C_2$. We first state a preliminary lemma.

\begin{proposition} \label{HomScheme2}
The number of irreducible components of $H_{C_2}$ is bounded above by the number of irreducible components of $H'_{C_2}$.
\end{proposition}

\begin{proof}
Let $\Phi:H'_{C_2}\to\mathrm{Hom}_{\mathbb{C}}(B,B)$ denote the morphism of schemes given by post-composition with $\pi_2:\tilde{A}\to B$. Then $H_{C_2}\subseteq H'_{C_2}$ is the fiber $\Phi^{-1}(\mathrm{id}_B)$, representing those morphisms $f:B\to\tilde{A}$ of degree at most $C_2$ such that $\pi_2\circ f=\mathrm{id}_B$.
	
\textbf{Case 1: $g(B)\ge 2$.} In this case, $\mathrm{Hom}_{\mathbb{C}}(B,B)$ is a union of $B(\bb{C})$ with a finite discrete scheme, since a smooth projective curve of genus $\ge 2$ admits only finitely many non-trivial morphisms to itself. In particular, the identity morphism $\mathrm{id}_B$ is an isolated point. The fiber $H_{C_2}=\Phi^{-1}(\mathrm{id}_B)$ is then clopen in $H'_{C_2}$, being the fiber over an isolated point. Since clopen subsets are always unions of connected components, it follows that $H_{C_2}$ is a union of connected components of $H'_{C_2}$, and in particular \[\#\{\text{irreducible components of }H_{C_2}\}\le\#\{\text{irreducible components of }H'_{C_2}\}.\]

\textbf{Case 2: $g(B)=1$.} Fix a basepoint $x_0\in B(\mathbb{C})$ to identify $B$ with an elliptic curve over $\mathbb{C}$. For each $x\in B(\mathbb{C})$, let $\tau_x:B\to B$ denote the translation-by-$x$ automorphism. Define the subgroup \[T:=\{\tau_x\in\mathrm{Hom}_{\mathbb{C}}(B,B)\mid x\in B(\mathbb{C})\},\] which is isomorphic to $B$ as a group scheme. In particular, $T$ is an irreducible, clopen subset of $\mathrm{Hom}(B,B)$. Since $\Phi$ is a morphism of schemes and $T\subseteq\mathrm{Hom}(B,B)$ is clopen, the preimage $\Phi^{-1}(T)\subseteq H'_{C_2}$ is clopen as well, and hence a union of connected components. The identity morphism $\mathrm{id}_B\in T$, so the fiber $H_{C_2}:=\Phi^{-1}(\mathrm{id}_B)$ lies in $\Phi^{-1}(T)$.
	
Now define a morphism $\varphi:T\times H_{C_2}\to\Phi^{-1}(T)$ by $(\tau_x,f)\mapsto f\circ\tau_x$. This is well-defined because \[\Phi(f\circ\tau_x)=\pi_2\circ f\circ\tau_x=(\pi_2\circ f)\circ\tau_x=\mathrm{id}_B\circ\tau_x=\tau_x.\] We define the inverse morphism $\psi:\Phi^{-1}(T)\to T\times H_{C_2}$ by $\tilde{f}\mapsto(\pi_2\circ\tilde{f},\tilde{f}\circ(\pi_2\circ\tilde{f})^{-1})$. Since $\pi_2\circ\tilde{f}\in T$ is invertible, this map is well-defined and inverse to $\varphi$, giving an isomorphism $\Phi^{-1}(T)\cong T\times H_{C_2}$.
	
As $T\cong B$ is irreducible, the irreducible components of $\Phi^{-1}(T)$ correspond bijectively to those of $H_{C_2}$, and we conclude \begin{align*}\#\{\text{irreducible components of }H_{C_2}\}&=\#\{\text{irreducible components of }\Phi^{-1}(T)\}\\&\le\#\{\text{irreducible components of }H'_{C_2}\}.\end{align*}
	
\textbf{Case 3: $g(B)=0$.} In this case, $B\cong\mathbb{P}^1$, and the scheme $\mathrm{Hom}(B,B)$ is stratified by degree: \[\mathrm{Hom}(B,B)=\bigsqcup_{d\ge 0}\mathrm{Hom}_d(B,B).\] The degree 1 component corresponds to $\mathrm{Aut}(\mathbb{P}^1)\cong\mathrm{PGL}_2(\mathbb{C})$, which we again denote by $T\subseteq\mathrm{Hom}(B,B)$. This is a connected algebraic group, and a clopen subset of $\mathrm{Hom}(B,B)$. As before, the preimage $\Phi^{-1}(T)\subseteq H'_{C_2}$ is thus clopen, and hence a union of connected components. Define the morphism $\varphi:T\times H_{C_2}\to\Phi^{-1}(T)$ by $(\alpha,f)\mapsto f\circ\alpha$, which is well-defined (i.e., does in fact map to $\Phi^{-1}(T)$) because \[\Phi(f\circ\alpha)=\pi_2\circ f\circ\alpha=(\pi_2\circ f)\circ\alpha=\mathrm{id}_B\circ\alpha=\alpha.\] The inverse is given by $\psi:\Phi^{-1}(T)\to T\times H_{C_2}$, $h\mapsto(\Phi(h),h\circ\Phi(h)^{-1})$, and defines an isomorphism $\Phi^{-1}(T)\cong T\times H_{C_2}$. Since $T\cong\mathrm{PGL}_2(\mathbb{C})$ is irreducible, we again obtain: \begin{align*}\#\{\text{irreducible components of }H_{C_2}\}&=\#\{\text{irreducible components of }\Phi^{-1}(T)\}\\&\le\#\{\text{irreducible components of }H'_{C_2}\}.\end{align*}\end{proof}

Next, we have an embedding $H'_{C_2} \xhookrightarrow{} \Hilb_{B \times \tilde{A}}$, by sending $\tilde{x} : B \to \tilde{A}$ to its graph $\Gamma_{\tilde{x}}$. In fact this embedding is an open immersion \cite[Theorem 6.6]{Nit05} .

Furthermore there is a finite set of polynomials $\{Q_1,\ldots,Q_E\}$ such that any $\tilde{x} \in H'_{C_2}$, the graph $\Gamma_{\tilde{x}}$ has Hilbert polynomial $Q_i$ for some $i$. As the graph $\Gamma(\tilde{x})$ is isomorphic to $B$, it is a smooth projective curve of genus $g(B)$. Hence its Hilbert polynomial must be of degree one with constant coefficient $g(B)$. Its leading coefficient is given by its degree and so the $Q_i$'s are of the form $it + g(B)$, where we let $i$ vary from $1$ to $C_2$.
\par 
Let $\Hilb^{Q_i}_{B \times \tilde{A}}$ be the components corresponding to the Hilbert polynomial $Q_i$. Then $H'_{C_2}$ embeds into $\bigsqcup_{i=1}^{C_2} \Hilb^{Q_i}_{B \times \tilde{A}}$ as an open subscheme. By the Proposition \ref{IrredOpen} below, it suffices to now bound the number of irreducible components of each $\Hilb^{Q_i}_{B \times \tilde{A}}$.

\begin{proposition} \label{IrredOpen}
Let $X,Y$ be noetherian schemes and assume $Y$ is an open subscheme of $X$. Then the number of irreducible components of $Y$ is bounded above by that of $X$. 
\end{proposition}

\begin{proof}
Write $X = \bigcup_{i \in I} X_i$ for $\{X_i\}_{i\in I}$ the set of irreducible components of $X$. Then $Y = \bigcup_{i \in I} (X_i \cap Y)$. Each nonempty $X_i \cap Y$ is irreducible as $Y$ is a Zariski open in $X_i$, and since $X_i$ is closed in $X$, the intersection $X_i \cap Y$ is also closed in $Y$ and hence must be an irreducible component. Therefore the set of irreducible components of $Y$ consists exactly of those $X_i \cap Y$ which are nonempty, yielding our desired bound.\end{proof}

It now suffices to bound the number of irreducible components of each $\Hilb^{Q_i}_{B \times \tilde{A}}$. To do this, we will use the work of Gotzmann. Recall that a coherent sheaf $F$ on $\bb{P}^r$ is $m$-regular if and only if 
$$H^i(\bb{P}^r, F(m-i)) = 0$$
for every integer $i > 0$. The smallest such $m$ is the Castelnuovo--Mumford regularity of $F$. Let $P$ be the Hilbert polynomial of a homogeneous ideal $I \subset k[x_0,\ldots,x_r]$. We define the Gotzmann number $\vphi(P)$ of $P$ as the infimum of all $m$ for which the ideal sheaf $I_Z$ is $m$-regular for all closed subvarieties $Z \subseteq \bb{P}^r$ with Hilbert polynomial $P$. We will let $\vphi(Z)$ be the Gotzmann number of the Hilbert polynomial of $I_Z$. Note that 
$$\Hilb_{I_Z}(n) = \binom{n+r}{r} - \Hilb_Z(n).$$
We now state \cite[Theorem 5.8]{Kwe22}. In Section 5 of \cite{Kwe22}, the variety $X$ is assumed to be smooth but this is not necessary. Note that $\Hilb^P$ in our notation is $\Hilb^{\binom{n+r}{r} - P}$ in Kweon's notation. Let $R_n = \bb{C}[x_0,\ldots,x_r]_n$ be the $n$th graded piece.

\begin{theorem} \label{HilbertSchemeBound}
Let $X \subseteq \bb{P}^r$ be a closed subvariety defined by polynomials of degree at most $d$. Fix a Hilbert polynomial $P(n)$ and let $t \geq \max\{\vphi(P),d\}$. Then there is an embedding
$$\Hilb^{\binom{n+r}{r} - P(n)} X \xhookrightarrow{} \bb{P}(\wedge^{\binom{t+r}{r} - P(t)} R_t)$$
such that the image is cut out by equations of degree at most $\binom{t+r+1}{r} - P(t+1) + 1$.
\end{theorem}

The following proposition allows us to bound the number of irreducible components of the image in terms of the degree of the equations cutting it out. 

\begin{proposition} \label{HilbertScheme2}
Let $X \subseteq \bb{P}^n$ be a closed subscheme that is cut out by equations of degree $\leq d$. Then the number of irreducible components of $X$ is bounded in terms of $d$ and $n$. 
\end{proposition}

\begin{proof}
Since the number of irreducible components remain invariant when passing to the reduced scheme, we may assume that we are working in the category of varieties. We induct on the stronger claim that given any irreducible subvariety $Y \subseteq \bb{P}^n$ of dimension $r$ and degree $e$, any subvariety $X \subseteq Y$ that is set theoretically cut out by equations of degree $\leq d$ has its number of irreducible components bounded in terms of $r,e$ and $d$. We will induct on the dimension $r$, where the case of $r = 0$ is clear.
\par 
Assuming we have proven the claim for all dimension $\leq r$ irreducible subvarieties, let $Y$ be an irreducible subvariety of dimension $r+1$ and degree $e$. Let $H$ be a hypersurface of degree $\leq d$ and consider the subvariety $Y \cap H$ where $H$ does not contain $Y$. Then $\dim (Y \cap H) \leq r$ and as $\deg(Y \cap H) \leq ed$, there are at most $ed$ many irreducible components $Z_1,\ldots,Z_m$. For each component $Z_i$, the degree is bounded above by $ed$ and hence by induction, if $X$ is a subvariety of that irreducible component which is cut out by equations of degree $\leq d$, the number of irreducible components of $X$ will be bounded in terms of $e,d$ and $r$. Then the number of irreducible components of a subvariety $X$ of $Y \cap H$ cut out by equations of degree $\leq d$ may be bounded by taking its sum of the irreducible components of $X \cap Z_i$, which gives us a bound in terms of $e,d$ and $r$ as desired. 
\end{proof}

By Theorem \ref{HilbertSchemeBound} and Proposition \ref{HilbertScheme2}, it suffices to now obtain an upper bound on the Gotzmann number $\vphi(Q_i)$ and the degree of equations that cut out $B \times \tilde{A}$. We know that $\tilde{A}$ is cut out by equations of at most degree $C_2$, but this might not cut out $\tilde{A}$ scheme-theoretically. We remedy this through the following propositions.  

\begin{proposition} \textnormal{\cite[Corollary 3.3]{Lia22}}\label{CastelnuovoBound1}
Let $S = K[x_1,\ldots,x_n]$ and let $I$ be a homogeneous ideal of $S$. Let $I$ be a homogeneous ideal in $S$ of dimension $r \geq 1$ and generated by forms of degree at most $d$. Then 
$$\reg(\sqrt{I}) \leq d^{(n-1)2^{r-1}}$$
where $\reg$ denotes the Castelnuovo--Mumford regularity. 
\end{proposition}

Recall that $\reg(J)$ of a homogeneous ideal $J$ is an upper bound on the minimal $D$ such that $J$ may be defined by equations of degree at most $D$. Hence there is some $C_3 > 0$, depending only on $(g,g(B),M,\delta)$, such that $\tilde{A}$ is scheme-theoretically cut out by equations of at most degree $C_3$. By Proposition \ref{HeightBoundEquations3}, we may embed $B \times \tilde{A}$ into some large projective space $\bb{P}^{N'}$ so that it is cut out by equations of degree at most $d$, with $d$ and $N'$ depending only on $(g,g(B),M,\delta)$. We are now ready to complete the last step of this section.

\begin{proposition} \label{HilbertScheme1}
There exists a constant $C_4 > 0$, depending only on $(g,g(B),M,\delta)$, such that the number of irreducible components of $\Hilb^{Q_i}_{B \times \tilde{A}}$ is at most $C_4$ for $1 \leq i \leq C_2$.  
\end{proposition}

\begin{proof}
Let $Q'_i(n) = \binom{n+N'}{N'} - Q_i(n)$. Then clearly there is an upper bound on $\vphi(Q'_i)$ for $0 \leq i \leq C_2$ that depends only on $C_2, g(B)$ and $N'$ by simply taking the maximum of $\vphi(Q'_i)$, which again depends only on $(g,g(B),M,\delta)$. By Theorem \ref{HilbertSchemeBound}, we obtain an upper bound on the degree of equations cutting out $\Hilb^{Q'_i}_{B \times \tilde{A}}$ inside some $\bb{P}^{N''}$ with $N''$ bounded from above by $(g,g(B),M,\delta)$. Proposition \ref{HilbertScheme2} gives us an upper bound on the number of irreducible components as desired.

\end{proof}

\begin{proof}[Proof of Theorem \ref{UniformBoundednessGoodReduction1}]
By Propositions \ref{HomScheme1}, \ref{HomScheme2} and \ref{IrredOpen}, it suffices to bound the number of irreducible components of $\bigsqcup_{i=1}^{C_2} \Hilb^{Q_i}_{B \times \tilde{A}}$. By Proposition \ref{HilbertScheme1}, the number of irreducible components is bounded in terms of $g,g(B),M$, and $\delta$. 
\end{proof}

We obtain the following corollary for abelian varieties with everywhere good reduction immediately, using \eqref{eqn:Deligne}. Recall that for a ppav $(A,\lambda)$, we let $L$ be a symmetric ample line bundle inducing $\lambda$ and let $\h_A$ denote the N\'eron--Tate height with respect to $L$. 

\begin{theorem} \label{UniformBoundednessGoodReduction2}
Let $K = \bb{C}(B)$ for a smooth projective curve $B$ over $\bb{C}$ and let $\delta > 0$ be a constant. Then there exists a natural number $N \geq 1$, depending only on $(g,g(B),\delta)$, such that for any principally polarized abelian variety $A/K$ of dimension $g$ with $\Tr_{K/\bb{C}}(A) = 0$ having everywhere good reduction, we have
$$\left \{ x \in A(K) \mid \h_{A}(x) \leq \delta \right\} \leq N.$$
\end{theorem}

\section{Reducing to the Principally Polarized Case} \label{sec: ZarhinTrick}
We finally handle the case of general abelian varieties by reducing to the case of principally polarized abelian varieties using Zarhin's trick \cite{Zar76}. This is a standard trick and has been used in various theorems to extend the ppav case to the case of general abelian varieties, for example in Falting's proof of the Mordell conjecture \cite{Fal83b}. 
\par 
Combining Theorems \ref{UniformGeneral1} and \ref{UniformBoundednessGoodReduction1}, we have proven the following theorem.

\begin{theorem} \label{Polarized1}
Let $K = \bb{C}(B)$ be a function field over $\bb{C}$ where $B$ is a smooth projective curve over $\bb{C}$. Let $(A,\lambda)$ be a principally polarized abelian variety of dimension $g$ over $K$ with $\Tr_{K/\bb{C}}(A) = 0$. Then there exist constants $c_1,c_2 > 0$ depending on $g,g(B)$ such that the set
$$\{x \in A(K) \mid \h_A(x) \leq c_2 \max\{h_{\Fal}(A),1\}\}$$ lies on $\divv(s)$ for some nonzero $s \in H^0(A,L^{c_1})$ where $L$ induces $2 \lambda$ on $A$.
\end{theorem}

Here, we let $h_{\Fal}(A)$ be the stable Faltings height; in other words, if $K'/K$ is an extension for which $A$ has semistable reduction, then 
$$h_{\Fal}(A) = \frac{1}{[K':K]} h_{\Fal}(A/K'),$$
where $h_{\Fal}(A/K')$ is the usual Faltings height of $A$ over $K'$. Before we continue, we will state a result of Nakamayae \cite{Nakamayae} which will allow us to pass from points lying on $\divv(s)$ to points lying on a proper abelian subvariety. For a finite set $\Gamma \subseteq A(\ovl{K})$, let 
$$\Gamma(k) = \{x_1 + x_2 + \cdots + x_k \mid x_i \in \Gamma\}.$$ Set $$\ovl{\Gamma}(k) = \bigcup_{i=0}^{k} \Gamma(k)$$
where by convention we define $\Gamma(0) = \{0\}$. 
\begin{theorem} \label{Nakamayae1} \textnormal{\cite[Theorem 1]{Nakamayae}}
Let $K$ be an algebraically closed field of characteristic $0$. Let $A$ be a $g$-dimensional abelian variety over $K$ and $L$ an ample line bundle on $A$. Let $s \in H^0(A,L)$ be a nonzero section. Suppose $s$ vanishes on $\ovl{\Gamma}(g)$ where $\Gamma$ and $\ovl{\Gamma}(g)$ is as above. Then there exists a proper abelian subvariety $H \subseteq A$ such that 
$$|(\Gamma+H)/H| \cdot \deg_L(H) \leq \deg_L(A).$$
\end{theorem}

One obstacle in utilizing Theorem \ref{Nakamayae1} is that $H$ is only defined over the algebraic closure. We circumvent this in the following proposition. 

\begin{proposition} \label{Nakamayae2}
Let $K$ be a field of characteristic $0$ and let $A$ be an abelian variety over $K$. Let $L$ be an ample line bundle. Let $H\subsetneq A_{\ovl{K}}$ be a proper algebraic subgroup of $A_{\ovl{K}}$ and let $x \in A(K)$ lie on $H$. Then there exist integers $N,D > 0$, depending only on $\deg_L(H)$, such that $[N]x$ lies on a proper abelian subvariety $B'/K$ with $\deg_L(B') \leq D$.
\end{proposition}

\begin{proof}
We may replace $L$ with $L^3$ and hence assume that $L$ is very ample. Then the number of connected components of $H$ is bounded above by $\deg_L(H)$. Replacing $x$ with $[\deg_L(H)!]x$, we may assume that $x$ lies on the identity component of $H$ and thus assume that $H$ is an abelian subvariety. 
\par 
Next, let $H_1 = H, H_2,\ldots,H_k$ be the $\Gal(\ovl{K}/K)$-conjugates of $H$. Note that as $H$ is an abelian subvariety, $k$ is uniformly bounded in terms of $g$ by \cite[Theorem 1.1]{Rem20} (cf. \cite[Introduction]{Sev22}). As $x$ is a $K$-point of $H$, it must lie in the intersection $\bigcap_{i=1}^{k} H'_i$. 
\par 
Note that $\bigcap_{i=1}^{k} H'_i$ is an algebraic subgroup defined over $K$ as it is $\Gal(\ovl{K}/K)$-stable. Hence the connected component of the identity is also defined over $K$ and is geometrically irreducible \cite[\href{https://stacks.math.columbia.edu/tag/0B7R}{Tag 0B7R}]{stacks-project}. Let $B'$ be this connected component. Then $\deg_L(B') \leq \deg_L(\bigcap_{i=1}^{k} H'_i)$ and $B'$ is necessarily a proper abelian subvariety of $A$ defined over $K$. 
\par 
The number of connected components of $\bigcap_{i=1}^{k} H'_i$ is clearly bounded from above by its degree with respect to $L$. It follows that we may find an $N$, depending only on $\deg_L(H)$, so that $[N]x$ lies on $B'$, as desired.\end{proof}

We next recall Zarhin's trick \cite[Section 4]{Zar23}. Let $A$ be an abelian variety over $K$ with a polarization $\lambda: A \to A^t$ of degree $d$. We wish to construct a principal polarization on $A^4 \times (A^t)^4$ over $K$. Fix some $s \geq 1$ such that $s\equiv -1 \pmod d$ and write $s = \sum_{i=1}^{4} a_i^2$ as a sum of four squares, with the $a_i$ being non-negative integers. Let 
$$I = \left( \begin{matrix} a_1 & -a_2 & -a_3 & -a_4 \\ a_2 & a_1 & a_4 & -a_3 \\ a_3 & -a_4 & a_1 & a_2 \\ a_4 & a_3 & -a_2 & -a_1 \end{matrix} \right)\in M_{4\times4}(\mathbb{Z}).$$ This may naturally be viewed as an element of $\End(A^4)$. Consider the isogeny 
$$\pi: A^8 = A^4 \times A^4 \to (A^t)^4 \times A^4, (x_1,x_2,x_3,x_4,y) \mapsto (\lambda(x_1),\ldots,\lambda(x_4), I(x)-y).$$
Then there exists a principal polarization $\mu$ on $X = (A^t)^4 \times A^4$ over $K$ that pulls back to $\lambda_{A^8}$ under $\pi$. One important fact we need is that if $L_X$ is the line bundle inducing $2 \mu$, then $\deg_{L_X}(A)$ and $\deg_{L_X}(A^t)$ are bounded in terms of $d$ and $g$. 

\begin{proposition} \label{PolarizationDegree1}
Let $A'$ be any of the four coordinate copies of $A$ inside $X = (A^t)^4 \times A^4$. Then $\deg_{L_X}(A') = 2d (g!)$. Similarly if $(A')^t$ denotes any of the four coordinate copies of $A^t$ inside $X$, then $\deg_{L_X}((A')^t)$ can be bounded solely in terms of $d$ and $g$.
\end{proposition}

\begin{proof}
Let $L_A$ be a line bundle inducing $2 \lambda_A$ on $A$ and $L_{A^8}$ a line bundle inducing $2 \lambda_{A^8}$ on $A^8$. Since $\mu$ pulls back to $\lambda_{A^8}$ under $\pi$, it follows that $\pi^*L_X = L_{A^8}$. Let $A_1$ denote the first coordinate in $(A^t)^4 \times A^4$. It follows from the projection formula that if $A' \subseteq A^8$ is such that $\pi(A') = A_1$ and $\pi$ is a finite morphism when restricted to $A'$, then
\begin{equation} \label{eq: Projectionformula1} \deg(\pi: A' \to A_1) \deg_{L_X}(A_1)  = \deg_{L_{A^8}}(A').
\end{equation} We may take the fifth copy of $A$ inside $A^8$ as such an $A'$, as $\pi$ maps $A'$ isomorphically to $A_1$. Since $\lambda_A^8$ restricts to $\lambda_A$ on $A'$, equation \eqref{eq: Projectionformula1} then gives
$$\deg_{L_X}(A_1) = \deg_{L_{A^8}}(A') = \deg_{2 \lambda_A}(A) = 2d (g!),$$
where the last equality is due to the fact that $\lambda_A$ is a degree $d$ polarization. 
Similarly if $A_1^t$ denotes the first copy of $A^t$ in $X$, then the abelian subvariety $A''$ given by 
$$A'' = \{ (x,0,0,0,I(x)) \text{ for } x \in A \} \subseteq A^8$$
maps surjectively to $A_1^t$ under $\pi$ with some finite degree. We now calculate $\deg_{L_{A^8}}(A'')$. We first view $A''$ as the image of $A_1$ under $f = (\id,0,0,0,I) \in \End(A^8)$ and we identify $A_1$ with $A$. We now write $f = \sum_{i=1}^{8} f_i$, where $f_i$ is the $i$th coordinate. Then $f_1^*L_{A^8} = L_A$ and $f_i^*L_{A^8} = 0$ for $i=2,3,4$. For $5 \leq i \leq 8$, we have 
$$\sum_{i=5}^{8} f_i^*L_{A^8} = (a_1^2 + a_2^2 + a_3^2 + a_4^2) L_A = s L_A.$$ 
Hence $f^* \lambda_A^8 = (s+1) \lambda_A$, and so 
$$\deg_{f^*L_{A^8}}(A_1)  = (s+1)2d(g!)$$
which implies that $\deg_{L_{A^8}}(A'')$ is bounded in terms of $(d,g)$, as desired.\end{proof}

We are now ready to show that Theorem \ref{Polarized1} will imply the following uniform statement for abelian varieties which are polarizable of degree $d \geq 1$. 

\begin{theorem} \label{Polarized2}
Let $K = \bb{C}(B)$ be a function field, where $B/\mathbb{C}$ is a smooth projective curve. Let $A/K$ be an abelian variety of dimension $g$ with $\Tr_{K/\bb{C}}(A) = 0$ and with polarization $\lambda:A\to A^t$ of degree $d$. Then there exist constants $c_2$ and $c_3$, depending only on $(g,g(B),d)$, such that for any torsion $x\in A(K)$, there exists an $m \leq c_2$ such that either $[m]x$ or $\lambda([m]x)$ lies in a proper abelian subvariety $B'$ of either $A$ or $A^t$ respectively that is defined over $K$ and polarizable of degree $\leq c_3$.  
\end{theorem}

\begin{proof}
We first apply Zarhin's trick and embed $A$ into $X = (A \times A^t)^4 = A^4 \times (A^t)^4$. The latter is a principally polarized abelian variety over $K$ with dimension $8g$. Furthermore, the $K/\bb{C}$-trace of $X$ is $0$ as this is true for $A$ and $A^t$.
\par 
Let $x$ be our $K$-torsion point and let $y = \lambda(x) \in A^t$. Let $N$ be a positive integer to be chosen later. We consider the set of all points of the form
\begin{equation}\label{eqn:8gpoints}([a_1]x,[a_2]x,[a_3]x,[a_4]x,[a_5]y,[a_6]y,[a_7]y,[a_8]y) \in X\end{equation}
where $0 \leq a_i \leq N$ for all $i$. Let $\lambda_X$ be the principal polarization and let $L$ be the line bundle inducing $2 \lambda_X$. By Theorem \ref{Polarized1}, we know that all $K$-torsion points of $X$ lie on $\divv(s)$ for $s \in H^0(A,L^{c_1})$, with $c_1$ depending only on $(g,g(B))$. Observe that $\deg_{L^{c_1}}(X)$ is $2c_1 (g!)$. We now apply Theorem \ref{Nakamayae1}, with $\Gamma$ taken to be the set of all points of the form \eqref{eqn:8gpoints}. It follows that there is an abelian subvariety $C$ of $X$ over $\ovl{K}$ along with $N' \leq \deg_{L^{c_1}}(X)$ points $g_1,\ldots,g_{N'}$ of $X(\ovl{K})$ so that $\Gamma$ is contained in $\bigcup_{i=1}^{N'} (g_i + C)$. We now take $N = N'+1$. We also have the inequality \begin{equation}\label{eqn:degCbd}\deg_{L^{c_1}}(C) \leq \deg_{L^{c_1}}(X).\end{equation} Since $\deg_{L^{c_1}}(X)$ is bounded in terms of $c_1$  only, it is bounded in terms of $g$ and $g(B)$.
\par 
Since $C$ is a proper abelian subvariety, if we let $A_1,\ldots,A_4$ and $A^t_1,\ldots,A^t_4$ denote the four copies of $A$ and $A^t$ inside $X$, then either $C \cap A_i$ is a proper subvariety for some $i$ or $C \cap A_j^t$ is a proper subvariety for some $j$. WLOG suppose that \begin{equation}\label{eqn:propersubvar}C \cap A_4^t \text{ is a proper subvariety.} \end{equation} 
Recall that $y=\lambda(x)$. Identifying $[i]y$ with $(0,\dots,0,[i]y)\in X$, since $N \geq N'+1$, we see that for two distinct $1 \leq i < j \leq N$,  we must have $[i]y, [j]y$ lying in the same coset $g_m + C$, and so $[j-i]y$ lies in $C$. Hence we have $[M]y \in C \cap A_4^t$ where $M$ is bounded in terms of $(g,g(B),d)$. Furthermore, $\deg_L(C \cap A_4^t)$ is bounded in terms of $(g,g(B),d)$ by Proposition \ref{PolarizationDegree1} and \eqref{eqn:degCbd}. 
\par 
We now apply Proposition \ref{Nakamayae2} and hence obtain an abelian subvariety $B'$ defined over $K$ with $\deg_L(B')$ bounded in terms of $(g,g(B),d)$ and with $[N]y$ lying in $B'$ and $N$ bounded in terms of $(g,g(B),d)$, as desired.\end{proof}

\begin{theorem} \label{Polarized3}
Let $g,d$ be positive integers and let $K = \bb{C}(B)$ be a function field over $\bb{C}$. Then there exists a constant $N = N(g,g(B),d)$ such that for any abelian variety $A/K$ of dimension $g$ and polarizable of degree $d$ with no $K$-isotrivial part, any $K$-torsion point has order at most $N$.
\end{theorem}

\begin{proof}
We proceed by induction on $g$. If $A$ is polarizable of degree $d$, note that its dual $A^t$ is also polarizable of degree $d$. Let $x \in A(K)$ be a torsion point. Now applying Theorem \ref{Polarized2}, we know there exists an $m \leq c_2$ where $c_2 = c_2(g,g(B),d)$ such that $[m]x$ or $\lambda([m]x)$ lies on a proper abelian subvariety $B'$ defined over $K$ of dimension $\leq g-1$ and polarizable of degree $\leq c_3 = c_3(g,g(B),d)$. This abelian subvariety $B'$ also satisfies $\Tr_{K/\bb{C}}(B') = 0$. By induction, we obtain that $[m]x$ or $\lambda([m]x)$ has order at most $N'(g,c_3)$, from which it follows that $x$ has order at most $N'(g,c_3) \cdot (c_2 d)$, i.e., order at most a constant depending only on $(g,g(B),d)$.\end{proof}

We can now remove the dependence on $d$ by Zarhin's trick. 

\begin{theorem} \label{Polarized4}
Let $K = \bb{C}(B)$ be a function field of a curve $B$ over $\bb{C}$. There exists a constant $N = N(g,g(B))$ such that for any abelian variety $A/K$ of dimension $g$ with no $K$-isotrivial part, any torsion point $x \in A(K)$ has order at most $N$.
\end{theorem}

\begin{proof}
By Theorem \ref{Polarized3}, such a constant exists for $8g$-dimensional principally polarized abelian varieties. But by Zarhin's trick, any abelian variety $A$ of dimension $g$ over $K$ may be embedded into an $8g$-dimensional principally polarized abelian variety, also defined over $K$. We then use that constant as our $N$.      
\end{proof}

We now prove the analogous statement for the Lang--Silverman conjecture.  

\begin{theorem} \label{LangSilvermanTheorem2}
Let $K = \bb{C}(B)$ be the function field of a smooth projective curve $B$ over $\bb{C}$ and let $g,d \geq 1$ be positive integers. There exist $c,m > 0$ depending on $(g,g(B),d)$ such that if $A/K$ is a $g$-dimensional abelian variety with no $K$-isotrivial part, $L$ is a degree $d$ polarization on $A$, and $x \in A(K)$ satisfies
$$\h_{A,L}(x) \leq c \max\{h_{\Fal}(A),1\},$$
then $[m]x$ lies on an abelian subvariety $B'/K$ such that $\deg_L(B')$ is bounded only in terms of $(g,g(B),d)$.
\end{theorem}

\begin{proof}
We first handle the case where $A$ is principally polarized. We briefly sketch the argument; it is similar to the proof of Theorem \ref{Polarized2}. Let $N$ be a positive integer to be determined momentarily, and let $c_1,c_2$ be as in Theorem \ref{Polarized1}. If $\h_{A,L}(x) \leq \frac{c_2}{(gN)^2} \max\{h_{\Fal}(A),1\}$, then Theorem \ref{Polarized1} implies that $x,[2]x,\ldots,[gN]x$ all lie on $\divv(s)$ for some $s \in H^0(A,L^{c_1})$. Using Theorem \ref{Nakamayae1} on $\Gamma = \{x,[2]x,\ldots,[N]x\}$ gives us some abelian subvariety $C \subseteq A_{\ovl{K}}$ and a finite set of elements $g_1,\ldots,g_n$ so that $x,[2]x,\ldots,[N]x$ all lie on $g_i + C$ for some $i$. Here, $n$ is bounded in terms of $(g,g(B))$. Then setting $N = n+1$, there exist $1\le i<j\le N$ and $1\le k\le n$ such that $[j]x,[i]x$ lie on $g_k + C$, so that $[j-i]x \in C$. Proposition \ref{Nakamayae2} gives us a multiple $[m]x$ lying on a proper abelian subvariety $B'$ defined over $K$, with $m$ bounded purely in terms of $g$, $g(B)$, and $\deg_L(B')$ is bounded in terms of $g$ and $g(B)$.  This completes the proof in the case where $A$ is principally polarized.
\par 
We now move on to the general case where $A$ is polarized of degree $d$. We apply Zarhin's trick again to get 
$$\pi: A^8 \to X = (A^4) \times (A^t)^4$$
along with line bundle $L_X$ on $X$ inducing twice of a principal polarization. Then we have
$$\h_{A^8,L^8}(x) = \h_{X,L_X}(\pi(x)).$$
Again, as in \eqref{eqn:8gpoints}, we consider tuples 
$$([a_1]x,\ldots,[a_8]x)$$
in $A^8$ where $0 \leq a_i \leq N$, for $N$ a positive integer to be determined later. Then if $\h_{A,L}(x) \leq \frac{c_2}{N^2} \{h_{\Fal}(A),1\}$, we have 
$$\h_{X,L_X}(\pi([a_1]x,\ldots,[a_8]x)) \leq \frac{c_2} \max\{h_{\Fal}(X),1\}$$
since $h_{\Fal}(X) = 8 h_{\Fal}(A)$ by \cite{GLFP25}. Now apply Theorem \ref{Polarized1} on $X$. Then $\pi([a_1]x,\ldots,[a_8]x)$ all lie on $\divv(s)$ for some $s \in H^0(X,L_X^{c_1})$. Pulling back, we obtain that $[a_1]x,\ldots,[a_8]x$ lie on $\divv(s')$ for some $s' \in H^0(A^8,L^{c_1}_{A^8})$. As in the ppav case, applying Theorem \ref{Nakamayae1} with Proposition \ref{Nakamayae2} and taking $N$ to be larger than the number of cosets gives us a proper abelian subvariety $B'$ over $K$ with $\deg_L(B')$ bounded in terms of $(g,g(B),d)$ and an $m > 0$ bounded in terms of $(g,g(B),d)$ such that $[m]x$ lies in $B'$.\end{proof}

As an immediate corollary, we obtain the following.

\begin{corollary} \label{LangSilvermanTheorem1}
Let $K = \bb{C}(B)$ be the function field of a smooth projective curve $B$ over $\bb{C}$ and let $g \geq 1$ be a positive integer. There exists $c > 0$ depending only on $g$ and $g(B)$ such that if $A/K$ is a $g$-dimensional abelian variety with no $K$-isotrivial part, $L$ is a symmetric ample line bundle on $A$, and $x \in A(K)$ is such that $\bb{Z} \cdot x$ is Zariski dense in $A$, then
$$\h_{A,L}(x) \geq c \max\{h_{\Fal}(A),1\}.$$
Here, $h_{\Fal}(A)$ denotes the stable Faltings height of $A$. 
\end{corollary}

We will now prove that our uniform bounds only depend on the gonality of $B$ instead of the genus. This will follow from application of Weil restriction.

\begin{proof}[Proof of Theorems \ref{UniformIntroTheorem1} and \ref{UniformIntroTheorem2} for $k = \bb{C}$]
Let $K = \bb{C}(B)$ and let $e$ be the gonality of $B$. Then there exists a morphism $f: B \to \bb{P}^1$ of degree $e$ which realizes $K$ as a field extension of $\bb{C}(t)$ of degree $e$. Let $A/K$ be our abelian variety with $\Tr_{K/\bb{C}}(A) = 0$. Consider $A' = \Res_{K/\bb{C}(t)}(A)$. This is an abelian variety over $\bb{C}(t)$ with dimension $eg$, such that for any scheme $S$ over $\bb{C}(t)$, we have a functorial bijection of sets
$$A'(S) = A(S \times_{\Spec \bb{C}(t)} \Spec K).$$
In particular, we have the canonical equality of points $A'(\bb{C}(t)) = A(K)$. If $B'$ is an abelian variety over $\bb{C}$, we have 
$$\Hom_{\bb{C}(t)}(B'_{\bb{C}(t)}, A') = \Hom_{K}(B'_K,A) = 0$$
and thus we must also have $\Tr_{\bb{C}(t)/\bb{C}}(A') = 0$. We may now apply Theorem \ref{Polarized4} and deduce a uniform bound on the number of torsion points in $A(K)$ that depends only on $\gon(B)$ and $g$, as desired. 
\par 
For Theorem \ref{UniformIntroTheorem2}, we apply the theory of Weil restriction of adelic line bundles by Loughran \cite{Lou15}. Given an ample line bundle $L$ on $A$ over $K$, we obtain a line bundle $L' = \Res_{K/\bb{C}(t)}(L)$ on $A'$ over $\bb{C}(t)$. Furthermore, given $x \in A(K)$, if $y \in A'(\bb{C}(t))$ is the point corresponding to $x$, we have 
$$\h_{A',L'}(y) = \h_{A,L}(x).$$
Note that $A'_{\ovl{\bb{C}(t)}}$ is isomorphic to $A^e_{\ovl{K}}$. Let $x \in A(K)$ be such that $\bb{Z} \cdot x$ is Zariski dense in $A$. Now let $A$ be polarizable of degree $d$, so that $A'$ is polarizable by some $L'$ of degree $de$. Theorem \ref{LangSilvermanTheorem2} tells us that there exist $c_1,m_1>0$ depending only on $(g,d,e)$ such that if $\h_{A',L'}(y) \leq c_1 \max\{h_{\Fal}(A'),1\}$, then $[m_1]y$ must lie in a proper abelian subvariety $C_1$. Note that in applying Theorem \ref{LangSilvermanTheorem2}, the base curve is isomorphic to $\bb{P}^1$ and hence there is no dependence on the genus.   
\par 
We now run an iterative process. Given a proper abelian subvariety $C_n$ of $A'$ over $K$ with a positive integer $m_n$ such that $[m_n]y$ lies in $C_n$, we obtain a constant $c_n > 0$ depending only on $(\deg_{L'}(C_n),m_n,g,d,e)$ for which either:
\begin{enumerate}
\item $\h_{A',L'}(y) \geq c_n h_{\Fal}(C_n)$ and we terminate the process;
\item or there exists a proper abelian subvariety $C_{n+1}$ of $C_n$ defined over $K$ and positive integer $m_{n+1}$ for which $[m_{n+1}]y$ lies in $C_{n+1}$ and $\deg_{L'}(C_{n+1}), m_{n+1}$ are bounded in terms of $(\deg_{L'}(C_n),m_n,g,d,e)$.
\end{enumerate}

Starting with our original $C_1$, observe that our process must terminate in at most $eg$ many steps as $\dim C_{n+1} < \dim C_n$. Let $C_k$ be the final abelian variety. Then $\h_{A',L'}(y) \geq c_k h_{\Fal}(C_k)$ with $c_k$ depending only on $(g,d,e)$ since the constants for $C_1$ are bounded in terms of $(g,d,e)$. It thus suffices to show that for any proper abelian subvariety $C$ of $A'$ that contains $[N]y$ for some $N > 0$, we have $h_{\Fal}(C) \geq h_{\Fal}(A)$. 
\par 
Passing to the algebraic closure $\ovl{\bb{C}(t)}$, we have an isomorphism 
\begin{equation} \label{eq: Isom1}
\iota: A'_{\ovl{\bb{C}(t)}} \xrightarrow{\sim} \prod_{\sigma: K \xhookrightarrow{} \ovl{\bb{C}(t)}} A_{\ovl{K}}^{\sigma}.
\end{equation}
If we let $\sigma_1,\ldots,\sigma_e$ be the different embeddings of $K$ into $\overline{\bb{C}(t)}$, then the point $y$ corresponds to $(\sigma_i(x))_{i=1}^{e}$ under this isomorphism. If $C$ is a proper abelian subvariety of $A'_{\ovl{\bb{C}(t)}}$ containing $[N]y$ for some positive integer $N$, then the projection of $\iota(C)$ to one of the copies of $A_{\ovl{K}}$ must contain $A_{\ovl{K}}$, because $\bb{Z} \cdot x$ is Zariski dense in $A$. Thus $A$ is an isogeny factor of $C$ and so we must have $h_{\Fal}(C) \geq h_{\Fal}(A)$ as we are over a function field of characteristic $0$. This proves Theorem \ref{UniformIntroTheorem2} with an extra dependence on $d$ where $(A,L)$ is a polarization of degree $d$. By using Zarhin's trick as in the proof of Theorem \ref{LangSilvermanTheorem2}, we can remove the dependency on $d$. This completes the proof of Theorem \ref{UniformIntroTheorem2}.
\end{proof}

We now extend to the case where $k$ is an arbitrary field of characteristic $0$. Let $k$ be such a field, let $B/k$ be a smooth projective geometrically integral curve with $K=k(B)$, and let $A/K$ be an abelian variety of dimension $g$. Choose a finitely generated subfield $k_0\subset k$ over $\mathbb{Q}$ and models $B_0/k_0$, $K_0=k_0(B_0)$, and $A_0/K_0$ with
\[(A,B,K)\cong (A_0,B_0,K_0)\otimes_{k_0}k.\]
Fix an embedding $k_0\hookrightarrow\mathbb{C}$ and set
\[B_{\mathbb{C}}=B_0\otimes_{k_0}\mathbb{C},\qquad
K_{\mathbb{C}}=\mathbb{C}(B_{\mathbb{C}}),\qquad
A_{\mathbb{C}}=A_0\otimes_{k_0}\mathbb{C}.\]

\begin{lemma} With notation as above, the following are equivalent: \begin{enumerate}\item$\operatorname{Tr}_{K/k}(A)=0$ \item$\operatorname{Tr}_{K_0/k_0}(A_0)=0$ \item$\operatorname{Tr}_{K_{\mathbb{C}}/\mathbb{C}}(A_{\mathbb{C}})=0$.\end{enumerate} In particular, $A$ has no $K$-isotrivial part if and only if $A_{\mathbb{C}}$ has no $K_{\mathbb{C}}$-isotrivial part.\end{lemma}

\begin{proof} By definition, $A$ has no $K$-isotrivial part iff its $K/k$-trace vanishes: \[\operatorname{Tr}_{K/k}(A)=0.\] For any field extension $E\supset k_0$ there is a natural homomorphism \[I'_{E/k_0}:\operatorname{Tr}_{K_0/k_0}(A_0)\otimes_{k_0}E\longrightarrow\operatorname{Tr}_{EK_0/E}(A_0\otimes_{k_0}E)\] which is an isogeny \cite[Theorem 6.6]{Con06}. Taking $E=k$ and $E=\mathbb{C}$ yields isogenies \[\operatorname{Tr}_{K_0/k_0}(A_0)\otimes_{k_0}k\longrightarrow \operatorname{Tr}_{K/k}(A),\qquad\operatorname{Tr}_{K_0/k_0}(A_0)\otimes_{k_0}\mathbb{C}\longrightarrow\operatorname{Tr}_{K_{\mathbb{C}}/\mathbb{C}}(A_{\mathbb{C}}).\] In particular, \[\operatorname{Tr}_{K/k}(A)=0\Longleftrightarrow\operatorname{Tr}_{K_0/k_0}(A_0)=0\Longleftrightarrow\operatorname{Tr}_{K_{\mathbb{C}}/\mathbb{C}}(A_{\mathbb{C}})=0.\] Hence $A_{\mathbb{C}}$ has no $K$-isotrivial part.\end{proof}

\begin{lemma}\label{lem:degreeunchanged} Let $\pi:B_{\C}\to B_0$ be the base change of $B_0/k_0$ along an embedding $k_0\hookrightarrow\mathbb{C}$. For any divisor $D$ on $B_0$ we have
	\[\deg_{B_{\mathbb{C}}}\big(\pi^*D\big)=\deg_{B_0}(D).\]
\end{lemma}

\begin{proof} Write $D=\sum_c n_c\,[c]$ over closed points $c$ of $B_0$. Let $\pi:B_{\mathbb{C}}\to B_0$ be the projection from the base change $B_{\mathbb{C}} = B_0 \times_{\mathrm{Spec}(k_0)} \mathrm{Spec}(\mathbb{C})$. For each $c$, the pullback divisor $\pi^*[c]$ is the closed subscheme $c \times_{\mathrm{Spec}(k_0)} \mathrm{Spec}(\mathbb{C}) = \mathrm{Spec}(k(c)\otimes_{k_0}\mathbb{C})$ of $B_{\mathbb{C}}$. Because $\mathrm{char}(k_0)=0$, the finite extension $k(c)/k_0$ is separable, so the $\mathbb{C}$-algebra $k(c)\otimes_{k_0}\mathbb{C}$ is isomorphic to a product of copies of $\mathbb{C}$ indexed by the $k_0$-embeddings $k(c)\hookrightarrow\mathbb{C}$. By the definition of degree on a curve over $\mathbb{C}$, $\deg_{B_{\mathbb{C}}}(\pi^*[c])=\operatorname{length}_{\mathbb{C}}(k(c)\otimes_{k_0}\mathbb{C})=[k(c):k_0]=\deg_{B_0}([c])$. Therefore \[\deg_{B_{\mathbb{C}}}(\pi^*D)=\sum_c n_c\,\deg_{B_{\mathbb{C}}}(\pi^*[c])=\sum_c n_c\,[k(c):k_0]=\deg_{B_0}(D).\] \end{proof}

\begin{lemma}Let $L_0$ be a symmetric ample line bundle on $A_0$, and let $L_{\mathbb{C}}:=L_0\otimes_{k_0}\C$. Let $x_0\in A_0(K_0)$, and let $x_\C:=x_0\otimes_{k_0}\mathbb{C}$. With notation as above,
	\[\h_{A_{\C},L_{\C}}(x_{\C})=\h_{A_0,L_0}(x_0).\]
\end{lemma}

\begin{proof} We replace $K_0$ with a finite extension $K_0$ so that $A_0$ has semistable reduction. Let $\mathcal N_0\to B_0$ be the Néron model of $A_0$. Then for some $n \geq 1$, the symmetric ample line bundle $L_0^n$ on $A_0$ extends uniquely to a cubical line bundle $\overline{\mathcal L}_0$ on $\mathcal N_0$, and the $K_0$-point $x_0$ extends (by the Néron mapping property) to a section $s_0:B_0\to\mathcal N_0$ \cite[Section 14.6]{DS22}. The canonical height satisfies \[n \h_{A_0,L_0}(x_0)=\deg_{B_0}\big(s_0^*\overline{\mathcal{L}}_0\big).\] Base change along $k_0\hookrightarrow\C$ gives the Néron model $\mathcal N_{\C}\to B_{\C}$ of $A_{\C}$, the cubical extension $\overline{\mathcal L}_{\C}$ of $L^n_{\C}$, and the section $s_{\C}$ of $x_{\C}$, with \[s_{\C}^*\overline{\mathcal{L}}_{\C}=\pi^*\big(s_0^*\overline{\mathcal L}_0\big),\] where $\pi:B_{\C}\to B_0$ is the projection induced by base change. By Lemma \ref{lem:degreeunchanged}, \[n \h_{A_{\C},L_{\C}}(x_{\C})=\deg_{B_{\C}}\big(s_{\C}^*\overline{\mathcal{L}}_{\C}\big)=\deg_{B_{\C}}\big(\pi^*(s_0^*\overline{\mathcal{L}}_0)\big)=\deg_{B_0}\big(s_0^*\overline{\mathcal{L}}_0\big)=n \h_{A_0,L_0}(x_0).\]\end{proof}

\begin{lemma} With notation as above, let $\omega=\bigwedge^g e^*\Omega^1_{\mathcal A}$ be the Hodge bundle on $\mathcal A_{g,3}$. For $A_0/K_0$ as above, let $\lambda_{A_0/B_0}$ be the line bundle on $B_0$ obtained from $\omega$ via pullback, so that $h_{\mathrm{Fal}}(A_0):=\deg_{B_0}\lambda_{A_0/B_0}$. Then
	\[h_{\mathrm{Fal}}(A_{\C})=\deg_{B_{\C}}\lambda_{A_{\C}/B_{\C}}=\deg_{B_0}\lambda_{A_0/B_0}=h_{\mathrm{Fal}}(A_0).\]\end{lemma}

\begin{proof} Let $\eta=\mathrm{Spec} K_0$ and $\eta_{\C}=\mathrm{Spec} K_{\C}$ with $K_{\C}=\C(B_{\C})$, and let $[A_0]\in\mathcal A_{g,3}(K_0)$ be the moduli point of $A_0$, with base change $[A_{\C}]\in\mathcal A_{g,3}(K_{\C})$. By definition, \[\lambda_{A_0/B_0}\big|_{\eta}=[A_0]^*\omega,\qquad\lambda_{A_{\C}/B_{\C}}\big|_{\eta_{\C}}=[A_{\C}]^*\omega_{\C}.\] Base change along $k_0\hookrightarrow\C$ gives a canonical isomorphism of Hodge bundles on moduli, \[\omega_{\C}\cong \omega\otimes_{k_0}\C,\] hence on the generic fibers \[[A_{\C}]^*\omega_{\C}\cong \big([A_0]^*\omega\big)\otimes_{k_0}\C.\] It follows that \[\lambda_{A_{\C}/B_{\C}}\cong \pi^*\lambda_{A_0/B_0},\] where $\pi:B_{\C}\to B_0$ is the projection induced by base change from $k_0$ to $\C$. Applying Lemma \ref{lem:degreeunchanged} for $\pi$, we get \[\deg_{B_{\C}}\lambda_{A_{\C}/B_{\C}}=\deg_{B_{\C}}\pi^*\lambda_{A_0/B_0}=\deg_{B_0}\lambda_{A_0/B_0}.\] Thus $h_{\mathrm{Fal}}(A_{\C})=h_{\mathrm{Fal}}(A_0)$.\end{proof}

\begin{corollary} If \[\h_{A_{\mathbb{C}},L_{\mathbb{C}}}(x_{\mathbb{C}})\ge c\max\{\,h_{\mathrm{Fal}}(A_{\mathbb{C}}),1\}\] holds for some $c>0$, then \[\h_{A_0,L_0}(x_0)\ge c\max\{\,h_{\mathrm{Fal}}(A_0),1\}.\]\end{corollary}

Finally, note that gonality is preserved under base change of constant fields. Applying this to the extensions $\mathbb{C}/k_0$ and $k/k_0$ proves that Theorems \ref{UniformIntroTheorem1} and \ref{UniformIntroTheorem2} hold when $\mathbb{C}$ is replaced by an arbitrary field $k$ of characteristic $0$, and the $K/k$-trace of $A$ is assumed to be $0$.

\printbibliography

@article{DP02,
 author = {David, Sinnou and Philippon, Patrice},
 title = {Minorations des hauteurs normalis{\'e}es des sous-vari{\'e}t{\'e}s de vari{\'e}t{\'e}s abeliennes. {II}.},
 fjournal = {Commentarii Mathematici Helvetici},
 journal = {Comment. Math. Helv.},
 volume = {77},
 number = {4},
 pages = {639--700},
 year = {2002},
}

@article{Paz13,
 author = {Pazuki, Fabien},
 title = {Lower bound of the {N{\'e}ron}-{Tate} height on abelian surfaces},
 fjournal = {Manuscripta Mathematica},
 journal = {Manuscr. Math.},
 issn = {0025-2611},
 volume = {142},
 number = {1-2},
 pages = {61--99},
 year = {2013},
 doi = {10.1007/s00229-012-0593-7},
 zbMATH = {6199400},
 Zbl = {1304.11058}
}

@article{GR25,
 author = {Gaudron, {\'E}ric and R{\'e}mond, Ga{\"e}l},
 title = {Nombre de petits points sur une vari\'et\'e ab\'elienne},
 fjournal = {Journal of the Institute of Mathematics of Jussieu},
 journal = {J. Inst. Math. Jussieu},
 issn = {1474-7480},
 volume = {24},
 number = {3},
 pages = {705--761},
 year = {2025},
 doi = {10.1017/S1474748024000549},
 zbMATH = {8047566}
}

@article{Dav91,
 author = {David, Sinnou},
 title = {Fonctions th{\^e}ta et points de torsion des vari{\'e}t{\'e}s ab{\'e}liennes.},
 fjournal = {Compositio Mathematica},
 journal = {Compos. Math.},
 issn = {0010-437X},
 volume = {78},
 number = {2},
 pages = {121--160},
 year = {1991},
 url = {https://eudml.org/doc/90083},
 zbMATH = {13048},
 Zbl = {0741.14025}
}

@article{HP22,
 author = {Hindry, Marc and Pacheco, Am{\'{\i}}lcar},
 title = {Erratum to: ``{An} analogue of the {Brauer}-{Siegel} theorem for abelian varieties in positive characteristic''},
 fjournal = {Moscow Mathematical Journal},
 journal = {Mosc. Math. J.},
 volume = {22},
 number = {1},
 pages = {169},
 year = {2022},
}

@article{Koi76,
 author = {Koizumi, Shoji},
 title = {Theta relations and projective normality of {Abelian} varieties},
 fjournal = {American Journal of Mathematics},
 journal = {Am. J. Math.},
 issn = {0002-9327},
 volume = {98},
 pages = {865--889},
 year = {1976},
 doi = {10.2307/2374034},
 zbMATH = {3540949},
 Zbl = {0347.14023}
}

@misc{Song25,
      title={Parametrization of geometric Beilinson--Bloch heights via adelic line bundles}, 
      author={Yinchong Song},
      year={2025},
      eprint={2406.19912},
      url={https://arxiv.org/abs/2406.19912}, 
}

@misc{Poi25,
      title={Valuative compactifications of analytic varieties}, 
      author={Jérôme Poineau},
      year={2025},
      eprint={2503.18643},
      primaryClass={math.AG},
      url={https://arxiv.org/abs/2503.18643}, 
}

@misc{Loo21b,
 author = {Looper, Nicole},
 title = {The {Uniform} {Boundedness} and {Dynamical} {Lang} {Conjectures} for polynomials},
 year = {2021},
 url = {https://arxiv.org/abs/2105.05240},
 arXiv = {arXiv:2105.05240}
}

@article{HP16,
 author = {Hindry, Marc and Pacheco, Am{\'{\i}}lcar},
 title = {An analogue of the {Brauer}-{Siegel} theorem for abelian varieties in positive characteristic},
 fjournal = {Moscow Mathematical Journal},
 journal = {Mosc. Math. J.},
 volume = {16},
 number = {1},
 pages = {45--93},
 year = {2016},
}

@book{RLV00,
 author = {Rumely, Robert and Lau, Chi Fong and Varley, Robert},
 title = {Existence of the sectional capacity},
 fseries = {Memoirs of the American Mathematical Society},
 series = {Mem. Am. Math. Soc.},
 volume = {690},
 year = {2000},
}

@book{Rum89,
 author = {Rumely, Robert},
 title = {Capacity theory on algebraic curves},
 fseries = {Lecture Notes in Mathematics},
 series = {Lect. Notes Math.},
 volume = {1378},
 year = {1989},
}

@article{DS22,
 author = {de Jong, Robin and Shokrieh, Farbod},
 title = {Faltings height and {N{\'e}ron}-{Tate} height of a theta divisor},
 fjournal = {Compositio Mathematica},
 journal = {Compos. Math.},
 volume = {158},
 number = {1},
 pages = {1--32},
 year = {2022},
}

@misc{AN22,
      title={Moduli of hybrid curves II: Tropical and hybrid Laplacians}, 
      author={Omid Amini and Noema Nicolussi},
      year={2022},
      eprint={2203.12785},
      archivePrefix={arXiv},
      url={https://arxiv.org/abs/2203.12785}, 
}

@misc{AN24,
      title={Moduli of hybrid curves I: Variations of canonical measures}, 
      author={Omid Amini and Noema Nicolussi},
      year={2024},
      eprint={2007.07130},
      archivePrefix={arXiv},
      url={https://arxiv.org/abs/2007.07130}, 
}

@misc{Gon25,
      title={Multiplier scales of a sequence of rational maps}, 
      author={Chen Gong},
      year={2025},
      eprint={2510.25029},
      archivePrefix={arXiv},
      url={https://arxiv.org/abs/2510.25029}, 
}

@article{Lou15,
 author = {Loughran, Daniel},
 title = {Rational points of bounded height and the {Weil} restriction},
 fjournal = {Israel Journal of Mathematics},
 journal = {Isr. J. Math.},
 volume = {210},
 pages = {47--79},
 year = {2015},
}

@misc{LT25,
 author = {Laga, Jef and Thorne, Jack},
 title = {Lower bounds on heights of odd degree points of hyperelliptic curves},
 year = {2025},
 url = {https://arxiv.org/abs/2507.08652},
 arXiv = {arXiv:2507.08652}
}

@article{EHK12,
 author = {Ellenberg, Jordan and Hall, Chris and Kowalski, Emmanuel},
 title = {Expander graphs, gonality, and variation of {Galois} representations},
 fjournal = {Duke Mathematical Journal},
 journal = {Duke Math. J.},
 volume = {161},
 number = {7},
 pages = {1233--1275},
 year = {2012},
}

@article{CT12,
 author = {Cadoret, Anna and Tamagawa, Akio},
 title = {Uniform boundedness of {{\(p\)}}-primary torsion of abelian schemes},
 fjournal = {Inventiones Mathematicae},
 journal = {Invent. Math.},
 volume = {188},
 number = {1},
 pages = {83--125},
 year = {2012},
}

@article{Par99,
 author = {Parent, Pierre},
 title = {Bornes effectives pour la torsion des courbes elliptiques sur les corps de nombres},
 fjournal = {Journal f{\"u}r die Reine und Angewandte Mathematik},
 journal = {J. Reine Angew. Math.},
 volume = {506},
 pages = {85--116},
 year = {1999},
}

@article{Aut01,
 author = {Autissier, Pascal},
 title = {Integral points on arithmetic surfaces},
 fjournal = {Journal f{\"u}r die Reine und Angewandte Mathematik},
 journal = {J. Reine Angew. Math.},
 volume = {531},
 pages = {201--235},
 year = {2001},
}

@misc{Yap25,
      title={On the Number of Small Points for Rational Maps}, 
      author={Jit Wu Yap},
      year={2025},
      eprint={2510.12039},
      archivePrefix={arXiv},
      url={https://arxiv.org/abs/2510.12039}, 
}

@misc{Poi24a,
      title={Dynamique analytique sur $\mathbf{Z}$. I : Mesures d'\'equilibre sur une droite projective relative}, 
      author={Jérôme Poineau},
      year={2024},
      eprint={2201.08480},
      archivePrefix={arXiv},
      url={https://arxiv.org/abs/2201.08480}, 
}

@book{Ber90,
 author = {Berkovich, Vladimir G.},
 title = {Spectral theory and analytic geometry over non-{Archimedean} fields},
 fseries = {Mathematical Surveys and Monographs},
 series = {Math. Surv. Monogr.},
 issn = {0076-5376},
 volume = {33},
 year = {1990},
}

@article{DGH21,
 author = {Dimitrov, Vesselin and Gao, Ziyang and Habegger, Philipp},
 title = {Uniformity in {Mordell}-{Lang} for curves},
 fjournal = {Annals of Mathematics. Second Series},
 journal = {Ann. Math. (2)},
 volume = {194},
 number = {1},
 pages = {237--298},
 year = {2021},
}

@article{GH19,
 author = {Gao, Ziyang and Habegger, Philipp},
 title = {Heights in families of abelian varieties and the geometric {Bogomolov} conjecture},
 fjournal = {Annals of Mathematics. Second Series},
 journal = {Ann. Math. (2)},
 volume = {189},
 number = {2},
 pages = {527--604},
 year = {2019},
}

@article{BBW11,
 author = {Berman, Robert and Boucksom, S{\'e}bastien and Witt Nystr{\"o}m, David},
 title = {Fekete points and convergence towards equilibrium measures on complex manifolds},
 fjournal = {Acta Mathematica},
 journal = {Acta Math.},
 volume = {207},
 number = {1},
 pages = {1--27},
 year = {2011},
}

@article{BE21,
 author = {Boucksom, S{\'e}bastien and Eriksson, Dennis},
 title = {Spaces of norms, determinant of cohomology and {Fekete} points in non-{Archimedean} geometry},
 fjournal = {Advances in Mathematics},
 journal = {Adv. Math.},
 volume = {378},
 pages = {125},
 year = {2021},
}

@misc{Yap24,
      title={Quantitative Equidistribution of Small Points for Canonical Heights}, 
      author={Jit Wu Yap},
      year={2024},
      eprint={2410.21679},}

@article{BR07,
 author = {Baker, Matt and Rumely, Robert},
 title = {Harmonic analysis on metrized graphs},
 fjournal = {Canadian Journal of Mathematics},
 journal = {Can. J. Math.},
 volume = {59},
 number = {2},
 pages = {225--275},
 year = {2007},
}

@article{CGHX21,
 author = {Cantat, Serge and Gao, Ziyang and Habegger, Philipp and Xie, Junyi},
 title = {The geometric {Bogomolov} conjecture},
 fjournal = {Duke Mathematical Journal},
 journal = {Duke Math. J.},
 volume = {170},
 number = {2},
 pages = {247--277},
 year = {2021},
}

@article{Loo21,
 author = {Looper, Nicole},
 title = {Dynamical uniform boundedness and the {{\(abc\)}}-conjecture},
 fjournal = {Inventiones Mathematicae},
 journal = {Invent. Math.},
 volume = {225},
 number = {1},
 pages = {1--44},
 year = {2021},
}

@article{Loo19,
 author = {Looper, Nicole},
 title = {A lower bound on the canonical height for polynomials},
 fjournal = {Mathematische Annalen},
 journal = {Math. Ann.},
 volume = {373},
 number = {3-4},
 pages = {1057--1074},
 year = {2019},
}

@article{Ing09,
 author = {Ingram, Patrick},
 title = {Lower bounds on the canonical height associated to the morphism {{\({{\phi}}(z)= z^d+c\)}}},
 fjournal = {Monatshefte f{\"u}r Mathematik},
 journal = {Monatsh. Math.},
 volume = {157},
 number = {1},
 pages = {69--89},
 year = {2009},
}

@article{Ben07,
 author = {Benedetto, Robert L.},
 title = {Preperiodic points of polynomials over global fields},
 fjournal = {Journal f{\"u}r die Reine und Angewandte Mathematik},
 journal = {J. Reine Angew. Math.},
 volume = {608},
 pages = {123--153},
 year = {2007},
}

@article{XY22,
 author = {Xie, Junyi and Yuan, Xinyi},
 title = {Geometric {Bogomolov} conjecture in arbitrary characteristics},
 fjournal = {Inventiones Mathematicae},
 journal = {Invent. Math.},
 volume = {229},
 number = {2},
 pages = {607--637},
 year = {2022},
}

@article{Yam18,
 author = {Yamaki, Kazuhiko},
 title = {Trace of abelian varieties over function fields and the geometric {Bogomolov} conjecture},
 fjournal = {Journal f{\"u}r die Reine und Angewandte Mathematik},
 journal = {J. Reine Angew. Math.},
 volume = {741},
 pages = {133--159},
 year = {2018},
}

@article{Gub07b,
 author = {Gubler, Walter},
 title = {The {Bogomolov} conjecture for totally degenerate abelian varieties},
 fjournal = {Inventiones Mathematicae},
 journal = {Invent. Math.},
 volume = {169},
 number = {2},
 pages = {377--400},
 year = {2007},
}

@article{CX08,
 author = {Clark, Pete L. and Xarles, Xavier},
 title = {Local bounds for torsion points on abelian varieties},
 fjournal = {Canadian Journal of Mathematics},
 journal = {Can. J. Math.},
 volume = {60},
 number = {3},
 pages = {532--555},
 year = {2008},
}

@misc{Fre89,
 author = {Frey, Gerhard},
 title = {Links between solutions of {{\(A-B=C\)}} and elliptic curves},
 year = {1989},
 howpublished = {Number theory, {Proc}. 15th {Journ}. {Arith}., {Ulm}/{FRG} 1987, {Lect}. {Notes} {Math}. 1380, 31-62 (1989).},
}

@article{HT06,
 author = {Hwang, Jun-Muk and To, Wing-Keung},
 title = {Uniform boundedness of level structures on abelian varieties over complex function fields},
 fjournal = {Mathematische Annalen},
 journal = {Math. Ann.},
 volume = {335},
 number = {2},
 pages = {363--377},
 year = {2006},
}

@article{Nad89,
 author = {Nadel, Alan},
 title = {The nonexistence of certain level structures on abelian varieties over complex function fields},
 fjournal = {Annals of Mathematics. Second Series},
 journal = {Ann. Math. (2)},
 volume = {129},
 number = {1},
 pages = {161--178},
 year = {1989},
}

@article{BT18,
 author = {Bakker, Benjamin and Tsimerman, Jacob},
 title = {The geometric torsion conjecture for abelian varieties with real multiplication},
 fjournal = {Journal of Differential Geometry},
 journal = {J. Differ. Geom.},
 volume = {109},
 number = {3},
 pages = {379--409},
 year = {2018},
}

@article{Paz10,
 author = {Pazuki, Fabien},
 title = {Remarks on a conjecture of {Lang}},
 fjournal = {Journal de Th{\'e}orie des Nombres de Bordeaux},
 journal = {J. Th{\'e}or. Nombres Bordx.},
 volume = {22},
 number = {1},
 pages = {161--179},
 year = {2010},
}

@article{Sil84,
 author = {Silverman, Joseph H.},
 title = {Lower bounds for height functions},
 fjournal = {Duke Mathematical Journal},
 journal = {Duke Math. J.},
 volume = {51},
 pages = {395--403},
 year = {1984},
}

@article{Sil81,
 author = {Silverman, Joseph H.},
 title = {Lower bound for the canonical height on elliptic curves},
 fjournal = {Duke Mathematical Journal},
 journal = {Duke Math. J.},
 volume = {48},
 pages = {633--648},
 year = {1981},
}

@article{Poo07,
 author = {Poonen, Bjorn},
 title = {Gonality of modular curves in characteristic {{\(p\)}}},
 fjournal = {Mathematical Research Letters},
 journal = {Math. Res. Lett.},
 volume = {14},
 number = {4},
 pages = {691--701},
 year = {2007},
}

@article{CH05,
 author = {Cojocaru, Alina and Hall, Chris},
 title = {Uniform results for {Serre}'s theorem for elliptic curves},
 fjournal = {IMRN. International Mathematics Research Notices},
 journal = {Int. Math. Res. Not.},
 volume = {2005},
 number = {50},
 year = {2005},
}

@misc{NS96,
      title={d-gonality of modular curves and bounding torsions}, 
      author={Khac Viet Nguyen and Masa-Hiko Saito},
      year={1996},
      eprint={alg-geom/9603024},
      archivePrefix={arXiv},
      url={https://arxiv.org/abs/alg-geom/9603024}, 
}

@article{Abr96,
 author = {Abramovich, Dan},
 title = {A linear lower bound on the gonality of modular curves},
 fjournal = {IMRN. International Mathematics Research Notices},
 journal = {Int. Math. Res. Not.},
 volume = {1996},
 number = {20},
 pages = {1005--1011},
 year = {1996},
}

@article{Mer96,
 author = {Merel, Lo\"ic},
 title = {Bornes pour la torsion des courbes elliptiques sur les corps de nombres},
 fjournal = {Inventiones Mathematicae},
 journal = {Invent. Math.},
 volume = {124},
 number = {1-3},
 pages = {437--449},
 year = {1996},
}

@incollection{KM95,
 author = {Kamienny, Sheldon and Mazur, Barry},
 title = {Rational torsion of prime order in elliptic curves over number fields (with an appendix by {A}. {Granville})},
 booktitle = {Columbia University number theory seminar, New York, 1992},
 pages = {81--98; appendix 99--100},
 year = {1995},
}

@article{Kam86,
 author = {Kamienny, Sheldon},
 title = {Torsion points on elliptic curves over all quadratic fields},
 fjournal = {Duke Mathematical Journal},
 journal = {Duke Math. J.},
 volume = {53},
 pages = {157--162},
 year = {1986},
}

@article{Maz78,
 author = {Mazur, Barry},
 title = {Rational isogenies of prime degree. ({With} an appendix by {D}. {Goldfeld})},
 fjournal = {Inventiones Mathematicae},
 journal = {Invent. Math.},
 volume = {44},
 pages = {129--162},
 year = {1978},
}

@article{Maz77,
 author = {Mazur, Barry},
 title = {Modular curves and the {Eisenstein} ideal},
 fjournal = {Publications Math{\'e}matiques},
 journal = {Publ. Math., Inst. Hautes {\'E}tud. Sci.},
 volume = {47},
 pages = {33--186},
 year = {1977},
}

@article{Zar76,
 author = {Zarhin, Yuri},
 title = {Isogenies of abelian varieties over fields of finite characteristic},
 fjournal = {Mathematics of the USSR, Sbornik},
 journal = {Math. USSR, Sb.},
 volume = {24},
 pages = {451--461},
 year = {1976},
}

@article{Fal83b,
 author = {Faltings, Gerd},
 title = {Finiteness theorems for abelian varieties over number fields.},
 fjournal = {Inventiones Mathematicae},
 journal = {Invent. Math.},
 issn = {0020-9910},
 volume = {73},
 year = {1983},
}

@article{Con06,
 author = {Conrad, Brian},
 title = {Chow's {{\(K/k\)}}-image and {{\(K/k\)}}-trace, and the {Lang}-{N{\'e}ron} theorem},
 fjournal = {L'Enseignement Math{\'e}matique. 2e S{\'e}rie},
 journal = {Enseign. Math. (2)},
 volume = {52},
 number = {1-2},
 pages = {37--108},
 year = {2006},
}

@article{Edi92,
 author = {Edixhoven, Bas},
 title = {N{\'e}ron models and tame ramification},
 fjournal = {Compositio Mathematica},
 journal = {Compos. Math.},
 volume = {81},
 number = {3},
 pages = {291--306},
 year = {1992},
}

@article{HS88,
 author = {Hindry, Marc and Silverman, Joseph},
 title = {The canonical height and integral points on elliptic curves},
 fjournal = {Inventiones Mathematicae},
 journal = {Invent. Math.},
 volume = {93},
 number = {2},
 pages = {419--450},
 year = {1988},
}

@misc{Wil21,
 author = {Wilms, Robert},
 title = {Degeneration of {Riemann} theta functions and of the {Zhang}-{Kawazumi} invariant with applications to a uniform {Bogomolov} conjecture},
 year = {2021},
 url = {https://arxiv.org/abs/2101.04024},
 arXiv = {arXiv:2101.04024},
}

@article{DEJ19,
 author = {de Jong, Robin},
 title = {Faltings delta-invariant and semistable degeneration},
 fjournal = {Journal of Differential Geometry},
 journal = {J. Differ. Geom.},
 volume = {111},
 number = {2},
 pages = {241--301},
 year = {2019},
}

@misc{Poi24b,
      title={Dynamique analytique sur $\mathbf{Z}$. II : \'Ecart uniforme entre Latt\`es et conjecture de Bogomolov-Fu-Tschinkel}, 
      author={Jérôme Poineau},
      year={2024},
      eprint={2207.01574},
      archivePrefix={arXiv},
      url={https://arxiv.org/abs/2207.01574}, 
}

@article {Fal21,
    AUTHOR = {Faltings, Gerd},
     TITLE = {Arakelov geometry on degenerating curves},
   JOURNAL = {J. Reine Angew. Math.},
  FJOURNAL = {Journal f\"ur die Reine und Angewandte Mathematik. [Crelle's
              Journal]},
    VOLUME = {771},
      YEAR = {2021},
     PAGES = {65--84},
}

@misc{FG24,
      title={Non-Archimedean techniques and dynamical degenerations}, 
      author={Charles Favre and Chen Gong},
      year={2024},
      eprint={2406.15892},
      archivePrefix={arXiv},
}

@article {Luo21,
    AUTHOR = {Luo, Yusheng},
     TITLE = {Limits of rational maps, {$\Bbb{R}$}-trees and barycentric
              extension},
   JOURNAL = {Adv. Math.},
  FJOURNAL = {Advances in Mathematics},
    VOLUME = {393},
      YEAR = {2021},
}

@article {Luo22,
    AUTHOR = {Luo, Yusheng},
     TITLE = {Trees, length spectra for rational maps via barycentric
              extensions, and {B}erkovich spaces},
   JOURNAL = {Duke Math. J.},
  FJOURNAL = {Duke Mathematical Journal},
    VOLUME = {171},
      YEAR = {2022},
    NUMBER = {14},
     PAGES = {2943--3001},
}

@misc{Nit05,
      title={Construction of Hilbert and Quot Schemes}, 
      author={Nitin Nitsure},
      year={2005},
      eprint={math/0504590},
      archivePrefix={arXiv},
      url={https://arxiv.org/abs/math/0504590}, 
}

@article{Lev68,
 author = {Levin, Martin},
 title = {On the group of rational points on elliptic curves over function fields},
 fjournal = {American Journal of Mathematics},
 journal = {Am. J. Math.},
 volume = {90},
 pages = {456--462},
 year = {1968},
}

@book{BLR90,
 author = {Bosch, Siegfried and L{\"u}tkebohmert, Werner and Raynaud, Michel},
 title = {N{\'e}ron models},
 fseries = {Ergebnisse der Mathematik und ihrer Grenzgebiete. 3. Folge},
 series = {Ergeb. Math. Grenzgeb., 3. Folge},
 issn = {0071-1136},
 volume = {21},
 year = {1990},
}

@article{HN11,
 author = {Halle, Lars Halvard and Nicaise, Johannes},
 title = {Jumps and monodromy of abelian varieties},
 fjournal = {Documenta Mathematica},
 journal = {Doc. Math.},
 volume = {16},
 pages = {937--968},
 year = {2011},
}

@misc{GLFP25,
      title={Variation of height in an isogeny class over a function field}, 
      author={Richard Griffon and Samuel Le Fourn and Fabien Pazuki},
      year={2025},
      eprint={2503.14318},
      archivePrefix={arXiv},
      url={https://arxiv.org/abs/2503.14318}, 
}

@article{Gro64,
 author = {Grothendieck, Alexandre},
 title = {{\'E}l{\'e}ments de g{\'e}om{\'e}trie alg{\'e}brique. {IV}: {\'E}tude locale des sch{\'e}mas et des morphismes de sch{\'e}mas. ({Premi{\`e}re} partie). {R{\'e}dig{\'e}} avec la colloboration de {J}. {Dieudonn{\'e}}},
 fjournal = {Publications Math{\'e}matiques},
 journal = {Publ. Math., Inst. Hautes {\'E}tud. Sci.},
 volume = {20},
 pages = {101--355},
 year = {1964},
}

@article{CH88,
 author = {Cornalba, Maurizio and Harris, Joe},
 title = {Divisor classes associated to families of stable varieties with applications to the moduli space of curves},
 fjournal = {Annales Scientifiques de l'{\'E}cole Normale Sup{\'e}rieure. Quatri{\`e}me S{\'e}rie},
 journal = {Ann. Sci. {\'E}c. Norm. Sup{\'e}r. (4)},
 volume = {21},
 number = {3},
 pages = {455--475},
 year = {1988},
}

@misc{Szp79,
 author = {Szpiro, Lucien},
 title = {Sur le th{\'e}or{\`e}me de rigidite de {Parsin} et {Arakelov}},
 year = {1979},
 howpublished = {Ast{\'e}risque 64, 169-202 (1979).},
}

@article{Fal84,
 author = {Faltings, Gerd},
 title = {Calculus on arithmetic surfaces},
 fjournal = {Annals of Mathematics. Second Series},
 journal = {Ann. Math. (2)},
 volume = {119},
 pages = {387--424},
 year = {1984},
}

@misc{AMRT75,
 author = {Avner Ash et al.},
 title = {Smooth compactifications of locally symmetric varieties},
 year = {1975},
 howpublished = {Lie {Groups}: {History}, {Frontiers} and {Applications}. {Vol}. {IV}. {Brookline}, {Mass}.: {Math} {Sci} {Press}. {IV}, 335 p. (1975).},
}

@misc{Lia22,
 author = {Liang, Yihui},
 title = {Upper bounds for regularity of radicals of ideals and arithmetic degrees},
 year = {2022},
 url = {https://arxiv.org/abs/2204.08565},
 arXiv = {arXiv:2204.08565}
}

@article{Kwe22,
 author = {Kweon, Hyuk Jun},
 title = {Bounds on the torsion subgroup schemes of {N{\'e}ron}-{Severi} group schemes},
 fjournal = {Advances in Mathematics},
 journal = {Adv. Math.},
 volume = {409},
 pages = {28},
 year = {2022},
}

@article{HN10,
 author = {Halle, Lars Halvard and Nicaise, Johannes},
 title = {The {N{\'e}ron} component series of an abelian variety},
 fjournal = {Mathematische Annalen},
 journal = {Math. Ann.},
 volume = {348},
 number = {3},
 pages = {749--778},
 year = {2010},
}

@article{Sev22,
 author = {Philip, S{\'e}verin},
 title = {Fields of definition of abelian subvarieties},
 fjournal = {Journal de Th{\'e}orie des Nombres de Bordeaux},
 journal = {J. Th{\'e}or. Nombres Bordx.},
 volume = {34},
 number = {2},
 pages = {537--547},
 year = {2022},
}

@misc{stacks-project,
  author       = {The {Stacks project authors}},
  title        = {The Stacks project},
  howpublished = {\url{https://stacks.math.columbia.edu}},
}

@article{Fal83,
 author = {Faltings, Gerd},
 title = {Arakelov's theorem for abelian varieties},
 fjournal = {Inventiones Mathematicae},
 journal = {Invent. Math.},
 issn = {0020-9910},
 volume = {73},
 pages = {337--347},
 year = {1983},
}

@InProceedings{Zar23,
author="Zarhin, Yuri",
title="Abelian Varieties, Quaternion Trick and Endomorphisms",
booktitle="Birational Geometry, K{\"a}hler--Einstein Metrics and Degenerations",
year="2023",
publisher="Springer International Publishing",
}

@misc{Yua24,
      title={Arithmetic bigness and a uniform Bogomolov-type result}, 
      author={Xinyi Yuan},
      year={2024},
      eprint={2108.05625},
      archivePrefix={arXiv},
      url={https://arxiv.org/abs/2108.05625}, 
}

@article{Nakamayae,
AUTHOR = {Nakamayae, Michael},
     TITLE = {Multiplicity estimates on commutative
algebraic groups},
   JOURNAL = {J. reine angew. Math.},
  FJOURNAL = {},
    VOLUME = {607},
      YEAR = {2007},
    NUMBER = {},
     PAGES = {217--235},

}

@article {Rem20,
    AUTHOR = {R\'emond, Ga\"el},
     TITLE = {Degr\'e{} de d\'efinition des endomorphismes d'une
              vari\'et\'e{} ab\'elienne},
   JOURNAL = {J. Eur. Math. Soc. (JEMS)},
  FJOURNAL = {Journal of the European Mathematical Society (JEMS)},
    VOLUME = {22},
      YEAR = {2020},
    NUMBER = {9},
     PAGES = {3059--3099},
}

@book {MFK94,
    AUTHOR = {Mumford, David and Fogarty, John and Kirwan, Frances},
     TITLE = {Geometric invariant theory},
    SERIES = {Ergebnisse der Mathematik und ihrer Grenzgebiete (2) [Results
              in Mathematics and Related Areas (2)]},
    VOLUME = {34},
   EDITION = {Third},
 PUBLISHER = {Springer-Verlag, Berlin},
      YEAR = {1994},
     PAGES = {xiv+292},
}

@article {Lor90,
    AUTHOR = {Lorenzini, Dino J.},
     TITLE = {Groups of components of {N}\'eron models of {J}acobians},
   JOURNAL = {Compositio Math.},
  FJOURNAL = {Compositio Mathematica},
    VOLUME = {73},
      YEAR = {1990},
    NUMBER = {2},
     PAGES = {145--160},
}

@article {Gub07,
    AUTHOR = {Gubler, Walter},
     TITLE = {Tropical varieties for non-{A}rchimedean analytic spaces},
   JOURNAL = {Invent. Math.},
  FJOURNAL = {Inventiones Mathematicae},
    VOLUME = {169},
      YEAR = {2007},
    NUMBER = {2},
     PAGES = {321--376},
}

@article {BGM22,
    AUTHOR = {Boucksom, S\'ebastien and Gubler, Walter and Martin, Florent},
     TITLE = {Differentiability of relative volumes over an arbitrary
              non-{A}rchimedean field},
   JOURNAL = {Int. Math. Res. Not. IMRN},
  FJOURNAL = {International Mathematics Research Notices. IMRN},
      YEAR = {2022},
    NUMBER = {8},
     PAGES = {6214--6242},
}

@incollection {GS23,
    AUTHOR = {Gubler, Walter and Stadl\"oder, Stefan},
     TITLE = {Monge-{A}mp\`ere measures for toric metrics on abelian
              varieties},
 BOOKTITLE = {Publications math\'ematiques de {B}esan\c con. {A}lg\`ebre et
              th\'eorie des nombres. 2023},
    SERIES = {Publ. Math. Besan\c con Alg\`ebre Th\'eorie Nr.},
    YEAR = {2023},
     PAGES = {49--84},
}

@article {Con99,
    AUTHOR = {Conrad, Brian},
     TITLE = {Irreducible components of rigid spaces},
   JOURNAL = {Ann. Inst. Fourier (Grenoble)},
  FJOURNAL = {Universit\'e{} de Grenoble. Annales de l'Institut Fourier},
    VOLUME = {49},
      YEAR = {1999},
    NUMBER = {2},
     PAGES = {473--541},
}

@article {BL84,
    AUTHOR = {Bosch, Siegfried and L\"utkebohmert, Werner},
     TITLE = {Stable reduction and uniformization of abelian varieties.
              {II}},
   JOURNAL = {Invent. Math.},
  FJOURNAL = {Inventiones Mathematicae},
    VOLUME = {78},
      YEAR = {1984},
    NUMBER = {2},
     PAGES = {257--297},
}

@article {Gub10,
    AUTHOR = {Gubler, Walter},
     TITLE = {Non-{A}rchimedean canonical measures on abelian varieties},
   JOURNAL = {Compos. Math.},
  FJOURNAL = {Compositio Mathematica},
    VOLUME = {146},
      YEAR = {2010},
    NUMBER = {3},
     PAGES = {683--730},
}

@incollection {Del87,
    AUTHOR = {Deligne, Pierre},
     TITLE = {Un th\'eor\`eme de finitude pour la monodromie},
 BOOKTITLE = {Discrete groups in geometry and analysis ({N}ew {H}aven,
              {C}onn., 1984)},
    SERIES = {Progr. Math.},
    VOLUME = {67},
     PAGES = {1--19},
 PUBLISHER = {Birkh\"auser Boston, Boston, MA},
      YEAR = {1987},
}

@misc{Mat20,
      title={Height functions associated with closed subschemes}, 
      author={Yohsuke Matsuzawa},
      year={2020},
      eprint={2008.08153},
      archivePrefix={arXiv},
      url={https://arxiv.org/abs/2008.08153}, 
}

@misc{Ing22,
      title={Explicit canonical heights for divisors relative to endomorphisms of $\mathbb{P}^N$}, 
      author={Ingram, Patrick},
      year={2022},
      eprint={2207.07206},
      archivePrefix={arXiv},
      url={https://arxiv.org/pdf/2207.07206}, 
}

@misc{Loo24,
      title={Arakelov--Green's functions for dynamical systems on projective varieties}, 
      author={Looper, Nicole},
      year={2024},
      eprint={2404.06981},
      archivePrefix={arXiv},
      url={https://arxiv.org/pdf/2404.06981}, 
}

@article {Sil87,
     AUTHOR = {Silverman, Joseph H.},
     TITLE = {Arithmetic distance functions and height functions
in Diophantine geometry},
   JOURNAL = {Math.~Ann.},
  FJOURNAL = {Mathematische Annalen},
    VOLUME = {279},
      YEAR = {1987},
    NUMBER = {},
     PAGES = {193--216},
}

@article {BR06,
     AUTHOR = {Matthew Baker and Robert Rumely},
     TITLE = {Equidistribution of small points, rational dynamics, and potential theory},
   JOURNAL = {Ann. Inst. Fourier (Grenoble)},
  FJOURNAL = {},
    VOLUME = {56},
      YEAR = {2006},
    NUMBER = {3},
     PAGES = {625--688},
}

@book {BG06, 
    AUTHOR = {Enrico Bombieri and Walter Gubler.},
     TITLE = {Heights in Diophantine Geometry},
    SERIES = {},
    VOLUME = {},
 PUBLISHER = {Cambridge University Press},
      YEAR = {2006},
}

@book {FC90,
    AUTHOR = {Chai, Ching-Li and Faltings, Gerd},
     TITLE = {Degeneration of Abelian Varieties},
    SERIES = {Ergebnisse der Mathematik und ihrer Grenzgebiete},
    VOLUME = {22},
 PUBLISHER = {Springer Berlin, Heidelberg},
      YEAR = {1990},
}

@book {BR10,
    AUTHOR = {Baker, Matthew and Rumely, Robert},
     TITLE = {Potential Theory and Dynamics on the Berkovich Projective Line},
    SERIES = {Mathematical Surveys and Monographs},
    VOLUME = {159},
 PUBLISHER = {AMS, Providence},
      YEAR = {2010},
}

@article {KS07,
     AUTHOR = {Kawaguchi, Shu and Silverman, Joseph },
     TITLE = {Dynamics of projective morphisms having identical canonical heights},
   JOURNAL = {Proc.~Math.~Lond.~Soc.},
  FJOURNAL = {Proceedings of the LMS},
    VOLUME = {95},
      YEAR = {2007},
    NUMBER = {2},
     PAGES = {519--544},
}

@article {Mas84,
    AUTHOR = {Masser, David},
     TITLE = {Small values of the quadratic part of the {N}\'eron-{T}ate
              height on an abelian variety},
   JOURNAL = {Compositio Math.},
  FJOURNAL = {Compositio Mathematica},
    VOLUME = {53},
      YEAR = {1984},
}

@article {BP05,
    AUTHOR = {Baker, Matthew and Petsche, Clayton},
     TITLE = {Global discrepancy and small points on elliptic curves},
   JOURNAL = {Int. Math. Res. Not.},
  FJOURNAL = {International Mathematics Research Notices},
      YEAR = {2005},
    NUMBER = {61},
     PAGES = {3791--3834},
}

@article {Fak03,
    AUTHOR = {Fakhruddin, Najmuddin},
     TITLE = {Questions on self maps of algebraic varieties},
   JOURNAL = {J. Ramanujan Math. Soc.},
  FJOURNAL = {Journal of the Ramanujan Mathematical Society},
    VOLUME = {18},
      YEAR = {2003},
    NUMBER = {2},
     PAGES = {109--122},
}

@article {Bak06,
    AUTHOR = {Baker, Matthew},
     TITLE = {A lower bound for average values of dynamical {G}reen's
              functions},
   JOURNAL = {Math. Res. Lett.},
  FJOURNAL = {Mathematical Research Letters},
    VOLUME = {13},
      YEAR = {2006},
    NUMBER = {2-3},
     PAGES = {245--257},
}

@misc{GGK21,
      title={The Uniform Mordell-Lang Conjecture}, 
      author={Ziyang Gao and Tangli Ge and Lars Kühne},
      year={2021},
      eprint={2105.15085},
      archivePrefix={arXiv},
}

@misc{Kuh21,
      title={Equidistribution in Families of Abelian Varieties and Uniformity}, 
      author={Lars Kühne},
      year={2021},
      eprint={2101.10272},
      archivePrefix={arXiv},
}

@misc{Gau23,
      title={Good height functions on quasi-projective varieties: equidistribution and applications in dynamics}, 
      author={Thomas Gauthier},
      year={2023},
      eprint={2105.02479},
      archivePrefix={arXiv},
}

@article {Zha98,
    AUTHOR = {Zhang, Shou-Wu},
     TITLE = {Equidistribution of small points on abelian varieties},
   JOURNAL = {Ann. of Math. (2)},
  FJOURNAL = {Annals of Mathematics. Second Series},
    VOLUME = {147},
      YEAR = {1998},
    NUMBER = {1},
     PAGES = {159--165},
}

@article {Ulm98,
    AUTHOR = {Ullmo, Emmanuel},
     TITLE = {Positivit\'{e} et discr\'{e}tion des points alg\'{e}briques
              des courbes},
   JOURNAL = {Ann. of Math. (2)},
  FJOURNAL = {Annals of Mathematics. Second Series},
    VOLUME = {147},
      YEAR = {1998},
    NUMBER = {1},
     PAGES = {167--179},
}

@article {Yua08,
    AUTHOR = {Yuan, Xinyi},
     TITLE = {Big line bundles over arithmetic varieties},
   JOURNAL = {Invent. Math.},
  FJOURNAL = {Inventiones Mathematicae},
    VOLUME = {173},
      YEAR = {2008},
    NUMBER = {3},
     PAGES = {603--649},
}

@article {Zha95,
    AUTHOR = {Zhang, Shou-Wu},
     TITLE = {Small points and adelic metrics},
   JOURNAL = {J. Algebraic Geom.},
  FJOURNAL = {Journal of Algebraic Geometry},
    VOLUME = {4},
      YEAR = {1995},
    NUMBER = {2},
     PAGES = {281--300},
}

@article{SUZ97,
    AUTHOR = {Szpiro, Lucien and Ullmo, Emmanuel and Zhang, Shou-Wu},
     TITLE = {\'{E}quir\'{e}partition des petits points},
   JOURNAL = {Invent. Math.},
  FJOURNAL = {Inventiones Mathematicae},
    VOLUME = {127},
      YEAR = {1997},
    NUMBER = {2},
     PAGES = {337--347},
}

@misc{YZ24,
      title={Adelic line bundles on quasi-projective varieties}, 
      author={Xinyi Yuan and Shou-Wu Zhang},
      year={2024},
      eprint={2105.13587},
      archivePrefix={arXiv},
}

@article{DKY20,
    AUTHOR = {DeMarco, Laura and Krieger, Holly and Ye, Hexi},
     TITLE = {Uniform {M}anin-{M}umford for a family of genus 2 curves},
   JOURNAL = {Ann. of Math. (2)},
  FJOURNAL = {Annals of Mathematics. Second Series},
    VOLUME = {191},
      YEAR = {2020},
    NUMBER = {3},
     PAGES = {949--1001},
}

@article{Lan78,
    AUTHOR = {Lang, Serge},
     TITLE = {Elliptic curves: {D}iophantine analysis},
    JOURNAL = {Grundlehren der Mathematischen Wissenschaften [Fundamental
              Principles of Mathematical Sciences]},
    VOLUME = {231},
 PUBLISHER = {Springer-Verlag, Berlin-New York},
      YEAR = {1978},
     PAGES = {xi+261},
      ISBN = {3-540-08489-4},
   MRCLASS = {10B15 (10Fxx 14G25)},
  MRNUMBER = {518817},
MRREVIEWER = {A. J. van der Poorten},
}

@article{CL06,
    AUTHOR = {Chambert-Loir, Antoine},
     TITLE = {Mesures et \'{e}quidistribution sur les espaces de {B}erkovich},
   JOURNAL = {J. Reine Angew. Math.},
  FJOURNAL = {Journal f\"{u}r die Reine und Angewandte Mathematik. [Crelle's
              Journal]},
    VOLUME = {595},
      YEAR = {2006},
     PAGES = {215--235},
}

@article{FRL06,
    AUTHOR = {Favre, Charles and Rivera-Letelier, Juan},
     TITLE = {\'{E}quidistribution quantitative des points de petite hauteur sur
              la droite projective},
   JOURNAL = {Math. Ann.},
  FJOURNAL = {Mathematische Annalen},
    VOLUME = {335},
      YEAR = {2006},
    NUMBER = {2},
     PAGES = {311--361},
}

\end{document}